\documentclass[11pt]{article}
\usepackage{declare-style-article}
\usepackage{declare-maths-article}
\usepackage{declare-text-article}
\usepackage{declare-theorems-article}
\hypersetup{
	pdfauthor 
		= {Nikita Nikolaev},
	pdftitle 
		= {Existence and Uniqueness of Exact WKB Solutions for Second-Order Singularly Perturbed Linear ODEs},
	pdfsubject
		= {},
	pdfcreator
		= {},
	pdfproducer
		= {},
	pdfkeywords
		= {exact WKB analysis, exact WKB method, linear ODEs, Schrödinger equation, singular perturbation theory, exact perturbation theory, Borel resummation, Borel-Laplace theory, asymptotic analysis, exponential asymptotics, Gevrey asymptotics, resurgence},
}

\DeclareSymbolFont{bbold}{U}{bbold}{m}{n}
\DeclareSymbolFontAlphabet{\mathbbold}{bbold}

\usepackage{appendix}
\usepackage[font={footnotesize}]{caption}
\usepackage[utf8]{inputenc}
\newcommand{\pto}[1]{\text{\tiny$(#1)$}}

\fancypagestyle{frontpage}{

\cfoot{}
\lfoot{\footnotesize 
First appeared: 21 June 2021\\
Contact: \href{mailto:n.nikolaev@sheffield.ac.uk}{n.nikolaev@sheffield.ac.uk}}
}

\usepackage{titletoc}
\usepackage{subcaption}
\usepackage{changepage}

\titlecontents*{subsubsection}
  [5em]
  {\scriptsize\itshape}
  {}
  {}
  {}
  [\:\|\:]
  []
  
\begin{document}

%===============================================================================
%:	TITLE PAGE
%===============================================================================

\title{Existence and Uniqueness of Exact WKB Solutions \\ for Second-Order Singularly Perturbed Linear ODEs}

\author{Nikita Nikolaev}

\affil{\small School of Mathematics and Statistics, University of Sheffield, United Kingdom}

\date{UPDATED VERSION: \today}
%\date{21 July 2021}

\maketitle
\thispagestyle{frontpage}

\begin{abstract}
We prove an existence and uniqueness theorem for exact WKB solutions of general singularly perturbed linear second-order ODEs in the complex domain.
These include the one-dimensional time-independent complex Schrödinger equation.
Notably, our results are valid both in the case of generic WKB trajectories as well as closed WKB trajectories.
We also explain in what sense exact and formal WKB solutions form a basis.
As a corollary of the proof, we establish the Borel summability of formal WKB solutions for a large class of problems, and derive an explicit formula for the Borel transform.
\end{abstract}

{\small
\textbf{Keywords:}
exact WKB analysis, exact WKB method, linear ODEs, Schrödinger equation, singular perturbation theory, exact perturbation theory, Borel resummation, Borel-Laplace theory, asymptotic analysis, exponential asymptotics, Gevrey asymptotics, resurgence

\textbf{2020 MSC:} 
	\href{https://zbmath.org/classification/?q=cc%3A34M60}{34M60} (primary); 
	\href{https://zbmath.org/classification/?q=cc%3A34E10}{34E10},
	\href{https://zbmath.org/classification/?q=cc%3A34E20}{34E20},
	\href{https://zbmath.org/classification/?q=cc%3A40G10}{40G10},
	\href{https://zbmath.org/classification/?q=cc%3A34M03}{34M03},
	\href{https://zbmath.org/classification/?q=cc%3A34M25}{34M25}
}

%===============================================================================
%:	TOC
%===============================================================================

\enlargethispage{-0.5cm}
{\begin{spacing}{0.9}
\small
\setcounter{tocdepth}{3}
\tableofcontents
\end{spacing}
}

%===============================================================================
%===============================================================================
%===============================================================================
\section{Introduction}
%===============================================================================
%===============================================================================
%===============================================================================

Consider a singularly perturbed $2^\text{nd}$-order linear ordinary differential equation
\eqntag{\label{210223212642}
	\hbar^2 \del_x^2 \psi + p \hbar \del_x \psi + q \psi = 0
\fullstop{,}
}
where $x$ is a complex variable, $\hbar$ is a small complex perturbation parameter, and the coefficients $p, q$ are holomorphic functions of $(x, \hbar)$ in some domain in $\Complex^2_{x\hbar}$.
The question we study is a quintessential problem in singular perturbation theory.
Namely, in this paper we search for solutions of \eqref{210223212642} that are holomorphic in both variables $x$ and $\hbar$ and admit well-defined asymptotics as $\hbar \to 0$ in a specified sector.

%===============================================================================
\paragraph{Results.}
The main result of this paper (\autoref{210116200501}) establishes precise general conditions for the existence and uniqueness of \textit{exact WKB solutions}.
These are holomorphic solutions that are canonically specified (in a precise sense via Borel resummation) by their exponential asymptotic expansions as $\hbar \to 0$ in a halfplane.
They are constructed by means of the Borel-Laplace method for the associated singularly perturbed Riccati equation which we investigated in \cite{MY2008.06492}.

Our approach yields a general result (\autoref{210221112347}) about the Borel summability of \textit{formal WKB solutions}.
These are exponential formal $\hbar$-power series solutions that assume the role of asymptotic expansions as $\hbar \to 0$ of exact WKB solutions within appropriate domains in $\Complex_x$.
We also prove an existence and uniqueness result (\autoref{210118113644}) for formal WKB solutions that in particular clarifies precisely in what sense they form a basis of formal solutions.

The construction of exact WKB solutions involves the geometry of certain real curves in $\Complex_x$ (the \textit{WKB trajectories}) traced out using a type of Liouville transformation.
Two special classes of this geometry (\textit{closed} and \textit{generic WKB trajectories}) are especially important because they appear in wide a variety of applications.
Notably, our existence and uniqueness and the Borel summability results remain valid in both of these situations.
Thus, we construct an exact WKB basis both along a closed WKB trajectory (\autoref{210519102448}) as well as a generic WKB trajectory (\autoref{210519103256}).

The explicit nature of our approach yields refined information about the Borel transform of WKB solutions.
This includes an explicit recursive formula (\autoref{210527135451}) which we hope will facilitate the analysis of the singularity structure in the Borel plane and perhaps lead to a fuller understanding of the resurgent properties of WKB solutions in a large class of problems (see \autoref{210622072950}).

%===============================================================================
\paragraph{Brief literature review.}
\label{210614233439}
The \textit{WKB approximation method} was established in the mathematical context by Jeffreys \cite{zbMATH02596027} and independently in the analysis of the Schrödinger equation in quantum mechanics by Wentzel \cite{zbMATH02589478}, Kramers \cite{zbMATH02589479}, and Brillouin \cite{zbMATH02589465}.
However, it has a very long history that goes further back to at least Carlini (1817), Liouville (1837), and Green (1837); for an in-depth historical overview, see for example the books of Heading \cite[Ch.I]{MR0148995}, Fröman and Fröman \cite[Ch.1]{MR0173481}, and Dingle \cite[Ch.XIII]{MR0499926}.
A clear exposition of the asymptotic theory of the WKB approximation can be found in the remarkable textbook of Bender and Orszag \cite[Part III]{MR1721985}.
The relationship between the asymptotic properties of the WKB approximation and the geometry of WKB trajectories was comprehensively examined by Evgrafov and Fedoryuk \cite{MR0209562, MR1295032}.

In the early 1980s, influenced by the earlier work of Balian and Bloch \cite{MR438937}, a groundbreaking advancement was made by Voros \cite{MR650542, MR729194} who lay the foundations for upgrading the WKB approximation method to an exact method, dubbed the \textit{exact WKB method}.
Although the value of considering the all-orders WKB expansions was suggested earlier by Dunham \cite{zbMATH03006786}, Bender and Wu \cite{MR0260323}, Dingle \cite{MR0499926}, and later by 't Hooft \cite{t1979can} in a purely physics context,
Voros was the first to introduce in a more systematic fashion techniques from the theory of Borel-Laplace transformations.
Crucial early contributions to the development of the general theory of exact WKB analysis for second-order linear ODEs include works of Leray \cite{MR103328}, Boutet de Monvel and Krée \cite{MR226170}, Silverstone \cite{MR819680}, Aoki, Sato, Kashiwara, Kawai, Takei, and Yoshida \cite{MR0420735, MR1166808, MR1256439, zbMATH00933375}, Delabaere, Dillinger, and Pham \cite{MR1209700, MR1483488, MR1704654, MR1770284}, Dunster, Lutz, and Schäfke \cite{MR1232828}, Écalle \cite{MR1296470}, and Koike \cite{zbMATH01511798, MR1770282, MR1753205}.
For a survey of early work in exact WKB analysis, we recommend the excellent book of Kawai and Takei \cite{MR2182990}, as well as the comprehensive review article by Voros \cite{MR2970524}.

Later developments focused mainly on understanding WKB-theoretic transformation series (first introduced in \cite{MR1166808}) that transform a given differential equation in a suitable neighbourhood of critical WKB trajectories (or \textit{Stokes lines}) to one in standard form whose WKB-theoretic properties are better understood.
A very partial list of contributions includes the works by Aoki, Kamimoto, Kawai, Koike, Sasaki, and Takei, \cite{MR2499553, kamimoto2011borel, MR3289681, MR3156844, sasaki2013borel, MR3209359}.
Parallel to this activity has been the classification of WKB geometry (or Stokes graphs), which includes the works by Aoki, Kawai, Takei, Tanda, \cite{MR1845214, takei2007exact, MR3156841, takei2017wkb}, as well as a detailed analysis of some WKB-theoretic properties of special classes of equations, which includes the works by Aoki, Kamimoto, Kawai, Koike, and Takei \cite{MR2827731, MR3050813, MR3408181, MR3209360, MR3343506, MR3569645, MR3617727, MR3523540, MR3929586}.

\enlargethispage{10pt}
However, although the existence of exact WKB solutions in classes of examples has been established, a general existence theorem for second-order linear ODEs has remained unavailable.
Contributions towards such a general theory include Gerard and Grigis \cite{MR929202}, Bodine, Dunster, Lutz, and Schäfke \cite{MR1232828, MR1854432, MR1914448}, Giller and Milczarski \cite{MR1809243}, Koike and Takei \cite{MR3156848}, Ferreira, López, and Sinusía \cite{MR3252852,MR3316977}, as well as most recently by Nemes \cite{MR4226390} whose preprint appeared at roughly the same time as our previous work \cite{MY2008.06492} that underpins our results here.
Our paper contributes to this long line of work by establishing a general theory of existence and uniqueness of exact WKB solutions, which generalises the relevant results from the aforementioned works (see \autoref{210721173245} for a discussion).

The need for a general existence result for exact WKB solutions of equations of the form \eqref{210223212642} is evident from a recent surge of scientific advances that rely upon it.
For example, this includes works in cluster algebras and character varieties \cite{MR3280000, kidwai2017spectral, MR3977870, kuwagaki2020sheaf, MR4155179}, stability conditions and Donaldson-Thomas invariants \cite{MR3853051, MR3935038, MR4205119}, high energy physics \cite{210612162633, MR3976899, MR4166626, MR4212092, grassi2021exact}, Gromov-Witten theory \cite{MR4029824}, as well as further developments in WKB analysis \cite{MR3343506, MR3523540, MR3929586}.
Some of the results in these references specifically rely on a statement of Borel summability of formal WKB solutions presented in \cite[Theorem 2.17]{MR3280000}.
This statement is drawn from an unpublished work of Koike and Schäfke on the Borel summability of WKB solutions of Schrödinger equations with polynomial potentials (see \cite[\S3.1]{takei2017wkb} for a brief account of Koike-Schäfke's ideas).
As explained in \autoref{210612121553}, this statement is a special case of our main theorem.
Therefore, our paper provides a rigorous proof of Koike-Schäfke's assertion.

%===============================================================================
\textbf{Acknowledgements.}
I want to expresses special gratitude to Marco Gualtieri, Marta Mazzocco, and Jörg Teschner for their encouragement to finish this project and write this paper.
I am thankful to Francis Bischoff, Marco Gualtieri, and Kento Osuga for very helpful suggestions for improving the draft of this paper.
I also thank André Voros for very useful comments on the first preprint version of this paper, especially regarding the literature review.
I benefited from discussions with Anton Alekseev, Dylan Allegretti, Francis Bischoff, Tom Bridgeland, Marco Gualtieri, Kohei Iwaki, Omar Kidwai, Andrew Neitzke, Gergő Nemes, Kento Osuga, and Shinji Sasaki.
This work was supported by the NCCR SwissMAP of the SNSF, as well as the EPSRC Programme Grant \textit{Enhancing RNG}.

%===============================================================================
%===============================================================================
%===============================================================================
\section{Setting}
\label{210223085256}
%===============================================================================
%===============================================================================
%===============================================================================

In this section, we describe our general setup, give a few examples, and define the notion of formal and exact solutions that are sought for in this paper.

%===============================================================================
%===============================================================================
\subsection{Background Assumptions}
%===============================================================================
%===============================================================================

%===============================================================================
\paragraph{}
\label{210527190843}
Fix a complex plane $\Complex_x$ with coordinate $x$ and another complex plane $\Complex_\hbar$ with coordinate $\hbar$.
Fix a domain $X \subset \Complex_x$ and a \hyperref[210217114252]{sectorial domain} $S \subset \Complex_\hbar$ at the origin with opening arc $A = (\vartheta_-, \vartheta_+) \subset \Real$ and opening angle $\pi \leq |A| \leq 2\pi$.
See \autoref{210618093403}.

We consider the following differential equation for a scalar function $\psi = \psi (x, \hbar)$:
\eqntag{\label{210115121038}
	\hbar^2 \del^2_x \psi + p \hbar \del_x \psi + q \psi = 0
\fullstop{,}
}
where $p,q$ are holomorphic functions of $(x,\hbar) \in X \times S$ which admit locally uniform \hyperref[210225135947]{Gevrey asymptotic expansions} $\hat{p}, \hat{q} \in \cal{O} (X) \bbrac{\hbar}$ with holomorphic coefficients along the closed arc $\bar{A}$:
\begin{equation}
\label{210115161312}
\begin{gathered}
	p (x, \hbar)
		\simeq \hat{p} (x, \hbar) \coleq \sum_{k=0}^\infty p_k (x) \hbar^k
\fullstop{,}
\\
	q (x, \hbar)
		\simeq \hat{q} (x, \hbar) \coleq \sum_{k=0}^\infty q_k (x) \hbar^k
\end{gathered}
\qqquad
\text{as $\hbar \to 0$ along $\bar{A}$, loc.unif. $\forall x \in X$\fullstop}
\end{equation}

%===============================================================================
\paragraph{}
Basic notions from asymptotic analysis as well as our notation and conventions are summarised in \autoref{210217113936}.
Explicitly, assumption \eqref{210115161312} for, say, $q$ means the following.
For every $x_0 \in X$, there is a neighbourhood $U_0 \subset X$ of $x_0$, a sectorial subdomain $S_0 \subset S$ with the \textit{same} opening $A$ (see \autoref{210618100236}), and constants $\CC, \MM > 0$ such that for all $n \geq 0$, all $x \in U_0$, and all $\hbar \in S_0$,
\eqntag{
	\left| q (x, \hbar) - \sum_{k=0}^{n-1} q_k (x) \hbar^k \right|
		\leq \CC \MM^n n! |\hbar|^n
\fullstop
}
(The sum for $n=0$ is empty.)
In particular, this means that each $q_k$ is bounded on $U_0$, and $q$ is bounded on $U_0 \times S_0$.
Let us also stress that \eqref{210115161312} is stronger than the usual notion of Gevrey asymptotics in that we require the above bounds to hold $\hbar \to 0$ uniformly in all directions within $S$ (see \autoref{210225134124}).
This stronger asymptotic assumption plays a crucial role in our analysis by allowing us to draw uniqueness conclusions with the help of a theorem of Nevanlinna \cite{nevanlinna1918theorie,MR558468} (see \autoref{210617120300}; see also \cite[Theorem B.11]{MY2008.06492} where we present a detailed proof).

%===============================================================================
%==============================
\begin{figure}[t]
\centering
\begin{subfigure}[b]{0.45\textwidth}
    \centering
    \includegraphics{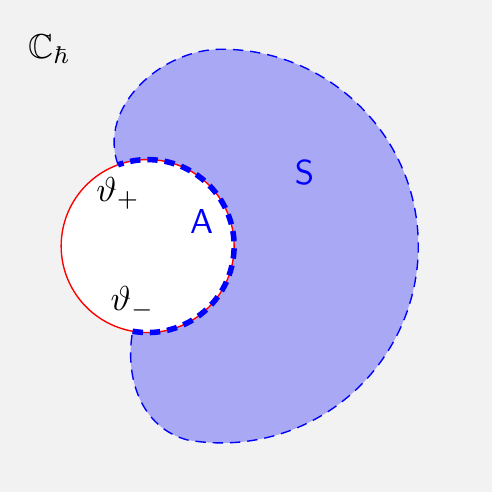}
    \caption{Sectorial domain $S$ with opening $A = (\vartheta_-, \vartheta_+)$.}
    \label{210618093403}
\end{subfigure}
\quad
\begin{subfigure}[b]{0.47\textwidth}
    \centering
    \includegraphics{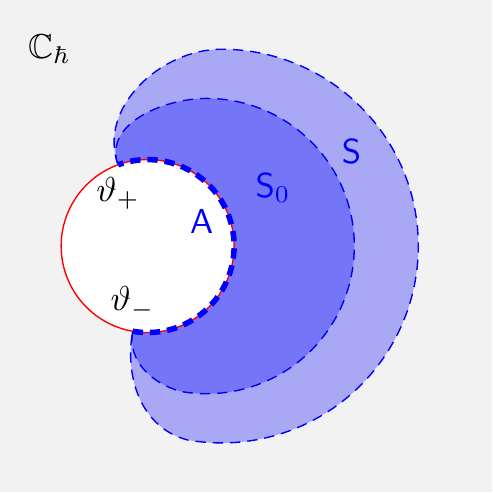}
    \caption{Sectorial subdomain $S_0$ with the same opening $A$.}
    \label{210618100236}
\end{subfigure}
\caption{}
\end{figure}
%==============================

%===============================================================================
\paragraph{}
A typical way to ensure the asymptotic condition \eqref{210115161312} is to start with holomorphic functions $p,q$, defined for $\hbar$ in a strictly larger sectorial domain $\tilde{S} \supset S$ with a strictly larger opening $\tilde{A}$ such that $\bar{A} \subset \tilde{A}$, which admit Gevrey asymptotics as $\hbar \to 0$ along the open arc $\tilde{A}$.
Then the restrictions of $p,q$ to $S$ necessarily satisfy \eqref{210115161312}.

In particular, if $p,q$ are actually holomorphic at $\hbar = 0$, then the asymptotic condition \eqref{210115161312} is automatically satisfied.
In this case, the power series $\hat{p}, \hat{q}$ are nothing but the convergent Taylor series expansions in $\hbar$ of $p,q$ at $\hbar = 0$.

%===============================================================================
%===============================================================================
\subsection{Examples}
%===============================================================================
%===============================================================================

%===============================================================================
\begin{example}[Classical differential equations]{210603192752}
The simplest interesting situation is when the coefficients $p, q$ are polynomial functions of $x$ only.
In this case, $X = \Complex_x$ and $S$ is ordinarily taken to be the right halfplane $\set{ \Re (\hbar) > 0 }$ where one demands asymptotic control on solutions as $\hbar \to 0$.
The most famous example of this situation is the $\hbar$-Airy equation:
\eqntag{\label{210603204921}
	\hbar^2 \del_x^2 \psi - x \psi = 0
\fullstop
}
Another example is the $\hbar$-Weber equation (sometimes also known as the quantum harmonic oscillator):
\eqntag{\label{210603203902}
	\hbar^2 \del_x^2 \psi - \big( x^2 - a \big) \psi = 0
\fullstop{,}
}
where $a$ is any complex number.

More general examples are provided by famous classical differential equations (with $\del_x$ replaced by $\hbar \del_x$) for which the coefficients $p,q$ are rational functions of $x$ only.
This includes the Gauss hypergeometric equation, as well as Bessel, Heun, Hermite, and many others.
In these cases, $X$ is the complement of finitely many points in $\Complex_x$.
\end{example}

%===============================================================================
\begin{example}[The Mathieu equation]{210603192602}
All equations in the previous example extend to second-order equations on the Riemann sphere with a pole at infinity (beware, however, that this extension is not unique).
A famous example where this is not the case is the Mathieu equation, sometimes written as
\eqntag{
	\hbar^2 \del_x^2 \psi - 2\big( \cos (x) - \EE \, \big) \psi = 0
\fullstop{,}
}
where $\EE$ is a complex number.
This equation has an essential singularity at infinity, yet the methods in this paper are still directly applicable to this equation.
\end{example}

%===============================================================================
\begin{example}[Mildly deformed coefficients]{210603203536}
More generally, the coefficients $p, q$ can be polynomials in $\hbar$ with coefficients which are rational or more general meromorphic functions of $x$.
Such examples appear, in Hermitian matrix models of the Gaussian potential \cite[\S6.7, equation (6.86)]{MR3694097} and more generally in the study of quantum curves, where the following deformation of the $\hbar$-Weber equation \eqref{210603203902} is encountered:
\eqntag{
	\hbar^2 \del_x^2 \psi - \big( x^2 - 4 + 2\hbar \big) \psi = 0
\fullstop
\tag*{\qedhere}
}
\end{example}

%===============================================================================
\begin{example}[A nontrivially deformed $\hbar$-Airy equation]{210603204218}
Our methods are applicable to classical differential equations with much more sophisticated $\hbar$-dependence.
For example, let $X = \Complex^\ast_x$, $S = \set{ \Re(\hbar) > 0}$, and consider the following nontrivial deformation of the $\hbar$-Airy equation \eqref{210603204921}:
\eqntag{
	\hbar^2 \del^2_x \psi - \big( x + \EE (x, \hbar) \big) \psi = 0
\fullstop{,}
}
where
\eqntag{
	\EE (x, \hbar)
		\coleq \int_0^{+\infty} \frac{e^{- \xi/\hbar}}{x + \xi} \dd{\xi}
\fullstop
}
The function $\EE (x, \hbar)$ is not holomorphic at $\hbar = 0$, but one can verify that it admits locally uniform Gevrey asymptotics as $\hbar \to 0$ along the closed arc $\bar{A} = [-\tfrac{\pi}{2}, +\tfrac{\pi}{2}]$.
\end{example}

%===============================================================================
\begin{example}[\textbf{The Schrödinger equation}]{210527155308}
The most famous special class of equations \eqref{210115121038} is the complex one-dimensional time-independent Schrödinger equation
\eqntag{
\label{210115121042}
	\Big( \hbar^2 \del_x^2 - \QQ (x, \hbar) \Big) \phi (x, \hbar) = 0
\fullstop
}
In fact, any second-order equation \eqref{210115121038} can be put into the Schrödinger form \eqref{210115121042} by means of the following transformation of the unknown variable:
\eqntag{
\label{210116161607}
	\phi (x, \hbar) = \exp \left( \frac{1}{2 \hbar} \int\nolimits^x_{x_0} p (t, \hbar) \dd{t} \right) \psi (x, \hbar)
\fullstop{,}
}
where $x_0$ is a suitably chosen basepoint.
In terms of the coefficients of \eqref{210115121038}, the resulting potential is $\QQ = \tfrac{1}{4} p^2 + \tfrac{1}{2} \hbar \del_x p - q$.

However, in this paper we prefer not to use the transformation \eqref{210116161607} and instead continue to work with equation \eqref{210115121038}.
We have two main reasons for this preference.
Firstly, the transformation \eqref{210116161607} develops essential singularities wherever $p$ has singularities, which is typically on the boundary of $X$.
Secondly, and perhaps most importantly from the geometric point of view, the Schrödinger form \eqref{210115121042} is not a coordinate-independent expression unless this differential equation is posed not on functions but on \textit{sections} of a specific line bundle over a Riemann surface (the square-root anti-canonical bundle).
These details will be explained in \cite{MY210517181728}.
\end{example}

%===============================================================================
%===============================================================================
\subsection{Formal and Exact Solutions}
%===============================================================================
%===============================================================================

Our goal is to construct holomorphic solutions of \eqref{210115121038} with prescribed asymptotic behaviour as $\hbar \to 0$.
Because of the way our linear equation is perturbed (i.e., differentiation $\del_x$ is multiplied by a single power of $\hbar$), it turns out that the correct notion of asymptotics is \textit{exponential asymptotics}.
This notion is briefly recalled in \autoref{210220170857}.
Indeed, one can easily verify that nonzero holomorphic solutions generically cannot admit a usual power series asymptotic expansion at $\hbar = 0$.
The following definition gives the precise class of solutions we seek in this paper.

%===============================================================================
\begin{defn}{210218171155}
A \dfn{weakly-exact solution} of \eqref{210115121038} on an open subset $U \subset X$ is a holomorphic solution $\psi$ defined on $U \times S' \subset X \times S$ for some sectorial subdomain $S' \subset S$ with nonempty opening $A' \subset A$, which admits locally uniform exponential asymptotics as $\hbar \to 0$ along $A'$.
If $A' = A$, then we call $\psi$ a \dfn{strongly-exact solution} on $(U;A)$, or simply \dfn{exact solution}.
\end{defn}

%===============================================================================
\paragraph{}
Existence of weakly-exact solutions is a classical fact in the theory of differential equations (see, e.g., \cite[Theorem 26.2]{MR0460820}).
However, weakly-exact solutions constructed by usual methods are inherently non-unique and in general there is no control on the size of the opening $A'$ (see, e.g., the remark in \cite[p.144]{MR0460820}, immediately following Theorem 26.1).
Our attention in this paper is instead focused on the strongly-exact solutions, which from this point of view form a more restricted class of solutions.
\textit{A priori}, these may not exist even if weakly-exact solutions are abundant.
The problem of finding strongly-exact solutions is a nontrivial sharpening of the more classical problem of finding weakly-exact solutions.

In this paper, we will in fact construct a special kind of exact solutions, called \textit{exact WKB solutions}.
Their distinguishing property is that they are canonically associated (in a precise sense) with their exponential asymptotics.

%===============================================================================
\paragraph{The space of exact solutions.}
We denote by $\mathbb{S} (U \times S)$ the space of all holomorphic solutions of \eqref{210115121038} defined on the domain $U \times S$ (without any asymptotic restrictions).
Our differential equation is linear of order two, so $\mathbb{S} (U \times S)$ is a rank-two module over the ring $\cal{O} (S)$ of holomorphic functions on $S$.
More generally, we can consider the space $\mathbb{S} (U; A)$ of \hyperref[210226112200]{\textit{semisectorial germs}} (see \autoref{210226161539}) of holomorphic solutions defined on the pair $(U;A)$.
Again, $\mathbb{S} (U; A)$ is a rank-two module for the ring $\cal{O} (A)$ of \hyperref[210225113610]{sectorial germs} (without asymptotic restrictions).

Let $\mathbb{ES} (U; A) \subset \mathbb{S} (U; A)$ be the subset of \textit{exact} solutions on $(U;A)$ in the sense of \autoref{210218171155}.
This is a module over the ring $\cal{A}^{\exp} (A)$ of holomorphic sectorial germs that admit exponential asymptotics along $A$.
We also let $\mathbb{GES} (U; A) \subset \mathbb{ES} (U; A)$ be the subset of exact solutions on $(U;A)$ which admit locally uniform exponential Gevrey asymptotics along $A$.
Similarly, it is a module over the ring $\cal{G}^{\exp} (A)$ of holomorphic sectorial germs that admit exponential asymptotics along $A$.
These are the spaces in which we search for (and find!) solutions of \eqref{210115121038}.

%===============================================================================
\paragraph{Formal solutions.}\label{210611190127}
If $\psi$ is an exact solution of \eqref{210115121038}, its exponential asymptotic expansion $\hat{\psi}$ formally satisfies the asymptotic analogue of the differential equation \eqref{210115121038} in which the coefficients $p,q$ are replaced by their asymptotic power series $\hat{p}, \hat{q}$:
\eqntag{\label{210604102933}
	\hbar^2 \del^2_x \hat{\psi} + \hat{p} \: \hbar \del_x \hat{\psi} + \hat{q} \; \hat{\psi} = 0
\fullstop{,}
}

\enlargethispage{0.5cm}
%===============================================================================
\begin{defn}{210617221701}
An \dfn{exponential power series solution} on a domain $U \subset X$ is an \hyperref[210226085039]{exponential power series}
\eqntag{\label{210617221741}
	\hat{\psi} = e^{-\Phi / \hbar} \, \hat{\Psi} ~\in~ \cal{O}^{\exp} (U) \bbrac{\hbar}
\fullstop{,}
}
where $\Phi \in \cal{O} (U)  [\hbar^{-1}]$ and $\hat{\Psi} \in \cal{O} (U)  \bbrac{\hbar}$, that formally satisfies \eqref{210604102933} for all $x \in U$.
More generally, a \dfn{formal solution} on $U$ is an \hyperref[210226114421]{exponential transseries} $\hat{\psi} \in \cal{O}^{\exp} (U) \bbrac{\hbar}$ that formally satisfies \eqref{210604102933}; i.e., $\hat{\psi}$ is a finite combination of exponential power series:
\eqntag{\label{210323092636}
	\hat{\psi}
		= \sum_\alpha^{\textup{\tiny{finite}}} e^{-\Phi_\alpha / \hbar} 
				\, \hat{\Psi}_\alpha
				~\in~ \cal{O}^{\exp} (U) \bbrac{\hbar}
}
where $\Phi_\alpha \in \cal{O} (U)  [\hbar^{-1}]$ and $\hat{\Psi}_\alpha \in \cal{O} (U)  \bbrac{\hbar}$.
\end{defn}

A brief account of exponential power series and transseries can be found in \autoref{210220170857}.
Note also that we have introduced a negative sign in the exponent in \eqref{210617221741} for future convenience.

%===============================================================================
\paragraph{The space of formal solutions.}
We denote the subset of $\cal{O}^{\exp} (U) \bbrac{\hbar}$ consisting of all formal solutions on $U$ by $\hat{\mathbb{S}} (U) \bbrac{\hbar}$.
It is a module over the ring of exponential transseries $\Complex^{\exp} \bbrac{\hbar}$.
Evidently, the asymptotic expansion defines a map $\ae : \mathbb{ES} (U;A) \to \hat{\mathbb{S}} (U) \bbrac{\hbar}$, but it is not a homomorphism because of complicated dominance relations for exponential prefactors (see \autoref{210220170857} for a comment).

%===============================================================================
%===============================================================================
%===============================================================================
\section{Formal WKB Solutions}
%===============================================================================
%===============================================================================
%===============================================================================

In this section, we analyse the differential equation \eqref{210115121038} in a purely formal setting where we ignore all analytic questions with respect to $\hbar$.
As a notational mnemonic used throughout the paper, objects decorated with a hat are formal.

%===============================================================================
\paragraph{Formal setup.}
Thus, we consider a general formal second-order differential equation of the following form:
\eqntag{\label{210302151519}
	\hbar^2 \del^2_x \hat{\psi} + \hat{p} \: \hbar \del_x \hat{\psi} + \hat{q} \; \hat{\psi} = 0
\fullstop{,}
}
where the coefficients $\hat{p},\hat{q} \in \cal{O} (X) \bbrac{\hbar}$ are formal power series in $\hbar$ with holomorphic coefficients on $X$:
\eqntag{\label{210323091619}
	\hat{p} (x, \hbar) \coleq \sum_{k=0}^\infty p_k (x) \hbar^k
\qqtext{and}
	\hat{q} (x, \hbar) \coleq \sum_{k=0}^\infty q_k (x) \hbar^k
\fullstop
}
We search for formal solutions in the sense of \autoref{210617221701}.

%===============================================================================
%===============================================================================
\subsection{The Semiclassical Limit}
%===============================================================================
%===============================================================================

%===============================================================================
\paragraph{}
A standard technique in solving linear ODEs is to introduce the characteristic equation.
We consider only the \dfn{leading-order characteristic equation}, which in for the second-order equation at hand is the following quadratic equation with holomorphic coefficients $p_0, q_0 \in \cal{O} (X)$ for a holomorphic function $\lambda = \lambda (x)$:
\eqntag{\label{210415145506}
	\lambda^2 - p_0 \lambda + q_0 = 0
\fullstop
}
We refer to its discriminant as the \dfn{leading-order characteristic discriminant}:
\eqntag{\label{210115160703}
	\DD_0 \coleq p_0^2 - 4 q_0 ~\in \cal{O} (X)
}
We always assume that $\DD_0$ is not identically zero.
The zeros of $\DD_0$ are called \dfn{turning points}, and all other points in $X$ are called \dfn{regular points}.
If $x_0 \in X$ is a regular point, then \eqref{210415145506} has two distinct holomorphic solutions.
We call them the \dfn{leading-order characteristic roots}.
Upon fixing a local square-root branch $\sqrt{\DD_0}$ near $x_0$, we will always label them as follows:
\eqntag{
\label{210305095811}
	\lambda_\pm \coleq \frac{p_0 \pm \sqrt{\DD_0}}{2}
\qqtext{so that}
	\sqrt{\DD_0} = \lambda_+ - \lambda_-
\fullstop
}

%===============================================================================
\begin{example}[The Schrödinger equation]{210526062601}
For the Schrödinger equation \eqref{210115121042}, the leading-order characteristic discriminant is $\DD_0 = 4 \QQ_0$, and the leading-order characteristic roots are $\lambda_\pm = \pm \sqrt{\QQ_0}$.
\end{example}

%===============================================================================
\begin{rem}{210527180116}
In analysis, equation \eqref{210415145506} is often called the \dfn{semiclassical limit} of the second-order differential operator in \eqref{210115121038}.
It can be obtained as the vanishing constraint of the following limit:
\eqn{
	\lim_{\substack{\hbar \to 0 \\ \hbar \in S}}
		\left[ e^{+\frac{1}{\hbar} \int\nolimits^x \lambda (t) \dd{t}}
		\Big(
			\hbar^2 \del^2_x + p (x, \hbar) \hbar \del_x + q (x, \hbar)
		\Big)
			e^{-\frac{1}{\hbar} \int\nolimits^x \lambda (t) \dd{t}}
		\right]
	= 0
\fullstop
\tag*{\qedhere}
}	
\end{rem}

%===============================================================================
%===============================================================================
\subsection{Existence and Uniqueness}
\label{210528120631}
%===============================================================================
%===============================================================================

Existence of formal solutions of \eqref{210302151519} away from turning points is well-known.
For Schrödinger equations (i.e., with $\hat{p} = 0$), they are often called \textit{(formal) WKB solutions}.
However, uniqueness statements are not usually made completely explicit.
It is also sometimes said that `formal WKB solutions are linearly independent and form a basis', but the space they generate is again not usually made explicit.
The purpose of the following proposition is to make these statements precise and explicit.

%===============================================================================
\begin{prop}[\textbf{Existence and Uniqueness of Formal WKB Solutions}]{210118113644}
\leavevmode\newline
Let $x_0 \in X$ be a regular point.
If $U \subset X$ is any simply connected neighbourhood of $x_0$ free of turning points, then \eqref{210302151519} has precisely two nonzero exponential power series solutions $\hat{\psi}_+, \hat{\psi}_- \in \cal{O}^{\exp} (U) \bbrac{\hbar}$ on $U$ normalised at $x_0$ by $\hat{\psi}_\pm (x_0, \hbar) = 1$.
They form a basis for the space $\hat{\mathbb{S}} (U) \bbrac{\hbar}$ of all formal solutions.
Moreover, once a square-root branch $\sqrt{\DD_0}$ near $x_0$ has been chosen, they can be labelled and expressed as follows:
\eqntag{\label{210115215219}
	\hat{\psi}_{\pm} (x, \hbar) 
		\coleq \exp \left( - \frac{1}{\hbar} \int_{x_0}^x \hat{s}_{\pm} (t, \hbar) \dd{t} \right)
\fullstop{,}
}
where $\hat{s}_+, \hat{s}_- \in \cal{O} (U) \bbrac{\hbar}$ are the two unique formal solutions with leading-order terms $\lambda_+, \lambda_-$ of the formal singularly perturbed Riccati equation
\eqntag{
\label{210115165557}
	\hbar \del_x \hat{s} = \hat{s}^{\,2} - \hat{p} \, \hat{s} + \hat{q}
\fullstop
}
In fact, $\hat{\psi}_\pm$ are the unique formal solutions on $U$ satisfying the following initial conditions:
\eqntag{\label{210611191005}
	\hat{\psi}_\pm (x_0, \hbar) = 1
\qqtext{and}
	\evatlong{\hbar \del_x \hat{\psi}_\pm}{(x,\hbar) = (x_0, 0)} = \lambda_\pm (x_0)
\fullstop
}
\end{prop}

The proof of this theorem is essentially a computation, presented in \autoref{210525083336}.

%===============================================================================
\begin{defn}{210302160338}
The two formal solutions $\hat{\psi}_\pm$ from \autoref{210118113644} are called the \dfn{formal WKB solutions} normalised at the regular point $x_0 \in U$.
The basis $\set{\smash{\hat{\psi}_+, \hat{\psi}_-}}$ is called the \dfn{formal WKB basis} normalised at $x_0$.
We will also refer to the two formal solutions $\hat{s}_{\pm}$ as the \dfn{formal characteristic roots}.
\end{defn}

%===============================================================================
\paragraph{}
Thus, \autoref{210118113644} explains that the formal WKB solutions are uniquely specified either by the normalisation condition and the requirement that they be exponential power series (i.e., having only one exponential prefactor) or equivalently by the two initial conditions \eqref{210611191005}.
Note that the second initial condition is nothing but a way to select the correct exponential prefactor.

%===============================================================================
\paragraph{WKB recursion.}
A practical advantage of formal WKB solutions is that the formal characteristic roots can be computed very explicitly by solving a recursive tower of linear algebraic equations.
The following lemma is a direct consequence of the proof of \autoref{210118113644}.

%===============================================================================
\begin{lem}{210415162549}
The coefficients $s^\pto{k}_\pm \in \cal{O} (U)$ of the formal characteristic roots
\eqntag{\label{210612102328}
	\hat{s}_{\pm} (x, \hbar) = \sum_{k=0}^\infty s_{\pm}^\pto{k} (x) \hbar^k
}
are given by the following recursive formula:
$s_\pm^\pto{0} = \lambda_\pm = \tfrac{1}{2} (p_0 \pm \sqrt{\DD_0})$, and for $k \geq 1$,
\eqntag{\label{210115170017}
	s_{\pm}^\pto{k} = \frac{\pm 1}{\sqrt{\DD_0}}
	\left(
			 \del_x s_{\pm}^\pto{k-1}
				- \sum_{k_1 + k_2 = k}^{k_1, k_2 \neq k} s_\pm^\pto{k_1} s_\pm^\pto{k_2}
					+ \sum_{k_1 + k_2 = k}^{k_2 \neq k} p_{k_1} s_\pm^\pto{k_2}
					- q_k
	\right)
\fullstop
}
\end{lem}

Explicitly, these formulas for low values of $k$ are:
\eqnstag{
\label{210223183123}
	s_\pm^\pto{1}
		&= \frac{\pm 1}{\sqrt{\DD_0}} 
			\Big( \del_x \lambda_\pm + p_1 \lambda_\pm - q_1 \Big)
\fullstop{,}
\\\label{210525213123}
	s_\pm^\pto{2}
		&= \frac{\pm 1}{\sqrt{\DD_0}}
			\Big( \del_x s_\pm^\pto{1} - ( s_\pm^\pto{1} )^2 + p_1 s_\pm^\pto{1} + p_2 \lambda_\pm - q_2 \Big)
\fullstop{,}
\\
	s_\pm^\pto{3}
		&= \frac{\pm 1}{\sqrt{\DD_0}}
			\Big( \del_x s_\pm^\pto{2} - 2 s_\pm^\pto{1} s_\pm^\pto{2} + p_1 s_\pm^\pto{2} + p_2 s_\pm^\pto{1} + p_3 \lambda_\pm - q_3 \Big)
\fullstop{,}
\\
	s_\pm^\pto{4}
		&= \frac{\pm 1}{\sqrt{\DD_0}}
			\Big( \del_x s_\pm^\pto{3} - 2 s_\pm^\pto{1} s_\pm^\pto{3} - ( s_\pm^\pto{2} )^2 + p_1 s_\pm^\pto{3} + p_2 s_\pm^\pto{2} + p_3 s_\pm^\pto{1} + p_4 \lambda_\pm - q_4 \Big)
\GREY{.}
}

%===============================================================================
\begin{example}[unperturbed coefficients]{210525213548}
In the simplest but very common situation where the coefficients $p,q$ are independent of $\hbar$, identities \eqref{210115170017} simplify as follows:
\eqntag{\label{210526062408}
	s_{\pm}^\pto{k} = \pm \frac{1}{\sqrt{\DD_0}}
		\left(
			 \del_x s_{\pm}^\pto{k}
				- \sum_{k_1 + k_2 = k}^{k_1, k_2 \neq k} s_{\pm}^\pto{k_1} s_{\pm}^\pto{k_2}
		\right)
\fullstop
}
For low values of $k$, these are:
\eqntag{
\begin{gathered}
	s_\pm^\pto{1}
		= \pm \frac{\del_x \lambda_\pm}{\sqrt{\DD_0}},
\quad
	s_\pm^\pto{2}
		= \pm \frac{\del_x s_\pm^\pto{1} - \big(s_\pm^\pto{1}\big)^2}{\sqrt{\DD_0}},
\quad
	s_\pm^\pto{3}
		= \pm \frac{\del_x s_\pm^\pto{2} - 2s_\pm^\pto{1} s_\pm^\pto{2}}{\sqrt{\DD_0}}
\fullstop{,}
\\
	s_\pm^\pto{4}
		= \frac{\pm 1}{\sqrt{\DD_0}}
			\Big( \del_x s_\pm^\pto{3} - 2 s_\pm^\pto{1} s_\pm^\pto{3} - ( s_\pm^\pto{2} )^2 \Big)
\fullstop
\end{gathered}
}
\end{example}

%===============================================================================
\begin{example}[Schrödinger equation]{210526062307}
For the Schrödinger equation \eqref{210115121042}, the leading-order characteristic roots are $\lambda_\pm = \pm \sqrt{\QQ_0}$.
Formula \eqref{210115170017} for the coefficients reduces to the following:
\eqntag{\label{210526062347}
	s^\pto{k}_\pm = \pm \frac{1}{2\sqrt{\QQ_0}}
	\left(
			 \del_x s^\pto{k-1}_\pm
				- \sum_{k_1 + k_2 = k}^{k_1, k_2 \neq k} s^\pto{k_1}_\pm s^\pto{k_2}_\pm
					+ \QQ_k
	\right)
\fullstop
}
If, furthermore, the potential $\QQ$ is independent of $\hbar$ (i.e., $\QQ = \QQ_0$), then every $\QQ_k$ in the recursive formula \eqref{210526062347} is $0$.
In this case, for low values of $k$:
\eqn{
\tag*{\qedhere}
	s^\pto{1}_\pm
		= \frac{1}{4} \frac{\QQ'}{\QQ},
\quad
	s^\pto{2}_\pm
		= 	\pm \frac{1}{8} \frac{\QQ''}{\QQ^{3/2}} 
			\mp \frac{5}{32} \frac{(\QQ')^2}{\QQ^{5/2}},
\quad
	s^\pto{3}_\pm
		= \frac{1}{16} \frac{\QQ'''}{\QQ^2}
			- \frac{9}{16} \frac{\QQ' \QQ''}{\QQ^3}
			+ \frac{35}{64} \frac{(\QQ')^3}{\QQ^4}
\fullstop
}
where ${}^\prime$ denotes $\del_x$.
So, for instance, for the $\hbar$-Airy equation \eqref{210603204921}, $\QQ = x$, so these formulas reduce to: $\lambda_\pm = \pm \sqrt{x}$, $s^\pto{1}_\pm = \frac{1}{4} x^{-1}$, $s^\pto{1}_\pm = \mp \frac{5}{32} x^{-5/2}$, $s^\pto{2}_\pm = \frac{35}{64} x^{-4}$.
%TODO: check again the last coefficient
\end{example}

%===============================================================================
\paragraph{WKB exponents.}
To express \eqref{210115215219} more explicitly as an exponential power series, we separate out the leading-order part of the formal solutions to the Riccati equation.
Let us define
\eqntag{\label{210526160837}
	\hat{\SS}_\pm (x, \hbar)
		\coleq \sum_{k=0}^\infty s^\pto{k+1}_\pm (x) \hbar^{k}
\qqtext{so that}
	\hat{s}_\pm = \lambda_\pm + \hbar \hat{\SS}_\pm
\fullstop
}
Then the formal WKB solutions $\hat{\psi}_\pm$ can be written as follows:
\eqnstag{\label{210115215327}
	\hat{\psi}_\pm (x, \hbar) 
		&= e^{-\Phi_\pm (x) / \hbar} \; \hat{\Psi}_\pm (x, \hbar)
\\
\label{210617192442}
		&= \exp \left( - \frac{1}{\hbar} 
				\int\nolimits_{x_0}^x \lambda_\pm (t) \dd{t} \right)
			\exp \left( 
				- \int\nolimits_{x_0}^x \hat{\SS}_\pm (t, \hbar) \dd{t} \right)
\fullstop{,}
}
where $\Phi_\pm \in \cal{O} (U)$ and $\hat{\Psi}_\pm \in \cal{O} (U) \bbrac{\hbar}$ are defined by the following formulas:
\eqnstag{
\label{210223093541}
		\Phi_\pm (x)
		&= \Phi_\pm (x)
		\coleq \int\nolimits_{x_0}^x \lambda_\pm (t) \dd{t}
\fullstop{,}
\\
\label{210116171149}
	\hat{\Psi}_\pm (x, \hbar)
		&= \hat{\Psi}_\pm (x, \hbar)
		\coleq 
		\sum_{n=0}^\infty \Psi_\pm^\pto{n} (x) \hbar^n
		\coleq \exp \left( - \int\nolimits_{x_0}^x \hat{\SS}_\pm (t, \hbar) \dd{t} \right)
\fullstop
}
The functions $\Phi_\pm$ are sometimes called the \dfn{WKB exponents}.
The integral of $\hat{\SS}_\pm$ in \eqref{210617192442} is interpreted as termwise integration.
In principle, the basepoint of integration $x_0$ may be taken on the boundary of $X$ as long as these integrals make sense.

%===============================================================================
%===============================================================================
\subsection{Remarks on Formal WKB Solutions}
%===============================================================================
%===============================================================================

%===============================================================================
\begin{rem}[\textbf{Behaviour at a turning point}]{210603182824}
Let $x_{\rm{tp}} \in X$ be a turning point of order $m \geq 1$, which means it is an $m$-th order zero of the leading-order characteristic discriminant $\DD_0$.
By examining the recursive formula \eqref{210115170017}, one can conclude that the coefficients $s^\pto{k}_\pm$ of each formal characteristic root $\hat{s}_\pm$ have the following behaviour near $x_{\rm{tp}}$.

When $m$ is odd, the leading-order characteristic roots $\lambda_\pm$ have a square-root branch singularity at $x_{\rm{tp}}$, but they are bounded as $x \to x_{\rm{tp}}$ in sectors.
Every subleading-order coefficient $s^\pto{k}_\pm$ has at worst a square-root branch singularity at $x_{\rm{tp}}$ but in general it is unbounded as $x \to x_{\rm{tp}}$.
On the other hand, when $m$ is even, the leading-order characteristic roots $\lambda_\pm$ are holomorphic at $x_{\rm{tp}}$.
Every subleading-order coefficient $s^\pto{k}_\pm$ is single-valued near $x_{\rm{tp}}$ but in general it has a pole there.

For either parity of $m$, in generic situation, the $k$-th order coefficient $s^\pto{k}_\pm$ (with $k \geq 1$) is bounded below by $x^{-km/2}$, where $x$ is a coordinate centred at $x_{\rm{tp}}$.
Therefore, the singular behaviour of the coefficients of $\hat{s}_\pm$ gets progressively worse in higher and higher orders of $\hbar$.
The upshot of this analysis is that in general it is not possible to expect the formal characteristic roots $\hat{s}_\pm$ (and therefore the corresponding formal WKB solutions $\hat{\psi}_\pm$) to be the uniform asymptotic expansions of exact solutions near a turning point.
This phenomenon lies at the heart of the breakdown of singular perturbation theory in the vicinity of a turning point.
\end{rem}

%===============================================================================
\begin{rem}[\textbf{Alternative expression for formal WKB solutions}]{210603184410}
In the literature, a different but closely related expression and normalisation for the formal WKB solutions is commonly used (see, e.g., \cite[equation (2.11)]{MR2182990} or \cite[equation (2.24)]{MR3280000}).
Consider the following \textit{odd} and \textit{even} parts of the formal characteristic solutions:
\eqntag{
\label{210221182241}
	\hat{s}_\text{od} \coleq \tfrac{1}{2} \big( \hat{s}_+ - \hat{s}_- \big)
\qqtext{and}
	\hat{s}_\text{ev} \coleq \tfrac{1}{2} \big( \hat{s}_- + \hat{s}_+ \big)
\fullstop
}
They satisfy $\hat{s}_{\pm} = \pm \hat{s}_\text{od} + \hat{s}_\text{ev}$.
It follows from the Riccati equation that the even part $\hat{s}_\text{ev}$ can be expressed in terms of the odd part $\hat{s}_\text{od}$ as follows:
\eqntag{
\label{210116173408}
	\hat{s}_\text{ev} = - \tfrac{1}{2} \hbar \del_x \log \hat{s}_\text{od} - \tfrac{1}{2} \hat{p}
\fullstop
}
Substituting these expressions into \eqref{210115215219} yields
\eqntag{
\label{210116180639}
	\hat{\psi}_{\pm} (x, \hbar)
		= \exp \left( \frac{1}{\hbar} \int_{x_0}^x
			\Big( \pm \hat{s}_\text{od} (t, \hbar) 
				- \tfrac{1}{2} \del_t \log \hat{s}_\text{od} (t, \hbar)
				- \tfrac{1}{2} \hat{p} (t, \hbar) \Big) \dd{t} 
				\right)
\fullstop
}
To integrate out the term involving the logarithmic derivative of $\hat{s}_\text{od}$, we must first make sense of choosing a square root of the formal power series $\hat{s}_\text{od}$.
To this end, we write $\hat{s}_\pm = \lambda_\pm + \hbar \hat{\SS}_\pm$ and put $\hat{\SS}_\text{od} \coleq \tfrac{1}{2} \big( \hat{\SS}_+ - \hat{\SS}_- \big)$.
Then using the identity $\sqrt{\DD_0} = \lambda_+ - \lambda_-$, we get
\eqntag{
	\hat{s}_\text{od} 
		= \tfrac{1}{2} \sqrt{\DD_0} + \hbar \hat{\SS}_\text{od}
		= \tfrac{1}{2} \sqrt{\DD_0} \Big( 1 + \tfrac{2}{\sqrt{\DD_0}} \hat{\SS}_\text{od} \hbar \Big)
\fullstop
}
Upon fixing a square-root branch $\DD_0^{\nicefrac{1}{4}}$ of $\sqrt{\DD_0}$, we let
\eqntag{
	\pm \sqrt{\hat{s}_\text{od} (x, \hbar)}
		\coleq \pm \tfrac{1}{\sqrt{2}} \DD_0^{\nicefrac{1}{4}}
			\sqrt{ 1 + \tfrac{2}{\sqrt{\DD_0}} \hat{\SS}_\text{od} \hbar }
		= \pm \tfrac{1}{\sqrt{2}} \DD_0^{\nicefrac{1}{4}}
			\sum_{k=1}^\infty \left( \tfrac{2}{\sqrt{\DD_0}} \hat{\SS}_\text{od} \hbar \right)^k
\fullstop
}
Notice that this is an invertible formal power series in $\hbar$ with holomorphic coefficients.
We therefore obtain the following four exponential power series:
\eqntag{
\label{210116190424}
	\hat{\psi}_{\pm,\varepsilon} (x, \hbar)
		= \frac{1}{\varepsilon \sqrt{\hat{s}_\text{od} (x, \hbar)}}
			\exp \left( \frac{1}{\hbar} \int_{x_0}^x
			\Big( \pm \hat{s}_\text{od} (t, \hbar)
				- \tfrac{1}{2} \hat{p} (t, \hbar) \Big) \dd{t} 
				\right)
\fullstop{,}
}
where $\varepsilon \in \set{+,-}$.
These are the expressions that in the literature are often referred to as (formal) \textit{WKB solutions}, although the choice of $\varepsilon$ is usually made implicitly.
Of course, it is always possible to choose the branch $\DD_0^{\nicefrac{1}{4}}$ such that $\hat{\psi}_{\pm, \varepsilon} = \hat{\psi}_\pm$.
\end{rem}

%===============================================================================
\begin{rem}[\textbf{The first-order WKB approximation}]{210603184521}
Expression \eqref{210116190424} is the more traditional form used to derive the WKB approximation for Schrödinger equations.
Indeed, if $p = 0$ and $q = -\QQ$, then $\DD_0 = 4 \QQ_0$ and $\lambda_\pm = \pm \sqrt{\QQ_0}$, so the leading-order term of $\hat{s}_\text{od}$ is simply $\sqrt{\QQ_0}$.
Truncating $\hat{s}_\text{od}$ at the leading order yields the famous analytic expressions
\eqntag{
	\pm \frac{1}{\sqrt[4]{\QQ_0}}
			\exp \left( \pm \frac{1}{\hbar} \int\nolimits_{x_0}^x \sqrt{\QQ_0 (t)} \dd{t} \right)
\fullstop
\tag*{\qedhere}
}
\end{rem}

%===============================================================================
\begin{rem}[\textbf{Normalisation at a turning point}]
Expression \eqref{210116190424} is often used to normalise WKB solutions at a turning point.
However, the existence of \text{exact} WKB solutions with such normalisation depends on some more global properties of the differential equation, which is not guaranteed in general.
In practical terms, even if an \text{exact} WKB solution normalised at a \text{regular} point exists, the corresponding exact WKB solution normalised at some nearby turning point may not exist due to the fact that the change of normalisation constant may not exist or does not have appropriate asymptotic behaviour as $\hbar \to 0$.
For this reason, in this paper we focus our attention exclusively on WKB solution normalised at regular points.
A more detailed explanation of this phenomenon will appear in \cite{MY210604104440}.
\end{rem}

%===============================================================================
\begin{rem}[\textbf{Formal characteristic discriminant}]{210603184758}
%\addcontentsline{toc}{subsubsection}{Formal characteristic discriminant}
Since the formal characteristic roots $\hat{s}_\pm$ are uniquely determined, we can introduce the \dfn{formal characteristic discriminant} of the differential equation \eqref{210302151519} by the usual formula for discriminants:
\eqntag{
	\hat{\DD} 
		\coleq - \big(\hat{s}_+ - \hat{s}_-\big) \big(\hat{s}_- - \hat{s}_+\big)
		= \big(\hat{s}_+ - \hat{s}_-\big)^2
		\in \cal{O} (U) \bbrac{\hbar}
\fullstop
}
Its leading-order part is the leading-order characteristic discriminant $\DD_0$.
If we furthermore define $\sqrt{\hat{\DD}} \coleq \hat{s}_+ - \hat{s}_-$, so that its leading-order part is just $\sqrt{\DD_0} = \lambda_+ - \lambda_-$, then we obtain the following relation:
\eqntag{
	\hat{s}_\text{od} = \tfrac{1}{2} \sqrt{\hat{\DD}}
\fullstop
}
Then \eqref{210116190424} can be written as follows:
\eqntag{
\label{210126103912}
	\hat{\psi}_{\pm,\varepsilon} (x, \hbar)
		= \frac{\sqrt{2}}{\varepsilon \sqrt[4]{\hat{\DD} (x, \hbar)}}
			\exp \left( \frac{1}{2\hbar} \int_{x_0}^x
			\Big( \pm \sqrt{\hat{\DD} (t, \hbar)}
				- \hat{p} (t, \hbar) \Big) \dd{t} 
				\right)
\fullstop
}
\end{rem}

%===============================================================================
\paragraph{Convergence of the formal Borel transform.}\label{210603185146}
Our final remark about formal WKB solutions in this section is a lemma that says that if $\hat{p}, \hat{q}$ are in a certain Gevrey regularity class, then the formal WKB solutions are in the corresponding exponential Gevrey regularity class.
Gevrey series are briefly reviewed in \autoref{210224181219}.

Let $x_0 \in X$ be a regular point and $U \subset X$ a simply connected neighbourhood of $x_0$ free of turning points.
Let $\hat{\psi}_+, \hat{\psi}_- \in \cal{O}^{\exp} (U) \bbrac{\hbar}$ be the two formal WKB solutions normalised at $x_0$.
Consider their formal Borel transform (see \autoref{210616130753}):
\eqntag{\label{210603121658}
	\hat{\Borel} \big[ \, \hat{\psi}_\pm \, \big] (x, \xi)
		= \exp \left( - \frac{1}{\hbar} 
			\int\nolimits_{x_0}^x \lambda_\pm (t) \dd{t} \right)
			\hat{\Borel} \big[ \, \hat{\Psi}_\pm \, \big] (x, \xi)
\fullstop
}
It is convenient to also consider the following \textit{exponentiated formal Borel transform} of $\hat{\psi}_\pm$ by using the exponential presentation \eqref{210617192442} and applying the ordinary formal Borel transform directly to the exponent:
\eqntag{\label{210526162424}
	\hat{\Borel}^{\exp} \big[ \, \hat{\psi}_\pm \, \big] (x, \xi)
		\coleq \exp \left( - \frac{1}{\hbar} 
			\int\nolimits_{x_0}^x \lambda_\pm (t) \dd{t} \right)
		\exp \left( 
			- \int\nolimits_{x_0}^x \hat{\Borel} \big[ \, \hat{\SS}_\pm \, \big] (t, \xi) \dd{t} \right)
\fullstop
}
Note that $\hat{\Borel}^{\exp} \big[ \, \hat{\psi}_\pm \, \big] \neq \hat{\Borel} \big[ \, \hat{\psi}_\pm \, \big]$ because the Borel transform converts multiplication into convolution, but clearly the convergence properties of one imply the same about the other.
Denote the formal Borel transform of $\hat{\SS}_\pm$ by
\eqntag{\label{210603121701}
	\hat{\sigma}_\pm (x, \xi)
		\coleq \hat{\Borel} \big[ \, \hat{\SS}_\pm \, \big] (x, \xi)
		= \sum_{k=0}^\infty s_{k+2}^\pto{i} (x) \frac{\xi^k}{k!}
\fullstop
}
Then we have the following assertion.

%===============================================================================
\begin{prop}{210603121655}
Suppose the coefficients $\hat{p}, \hat{q}$ are locally uniformly Gevrey power series on $U$; in symbols, $\hat{p}, \hat{q} \in \cal{G} (U) \bbrac{\hbar}$.
Then the formal Borel transforms $\hat{\sigma}_\pm$ are locally uniformly convergent power series and hence $\hat{\SS}_\pm$ are locally uniformly Gevrey series; in symbols, $\hat{\sigma}_\pm \in \cal{O} (U) \set{\xi}$ and $\hat{\SS}_\pm \in \cal{G} (U) \bbrac{\hbar}$.
Consequently, exponentiated formal Borel transforms $\hat{\Borel}^{\exp} \big[ \, \hat{\Psi}_\pm \, \big]$ and hence the ordinary formal Borel transforms $\hat{\Borel} \big[ \, \hat{\Psi}_\pm \, \big]$ are locally uniformly convergent power series in $\xi$ on $U$.
Thus, the formal WKB solutions are locally uniformly exponential Gevrey series; in symbols, $\hat{\psi}_\pm \in \cal{G}^{\exp} (U) \bbrac{\hbar}$.
\end{prop}

This proposition is not necessary for the proof of our main result in this paper (\autoref{210116200501}).
In fact, on certain subsets $U \subset X$, this proposition can be seen as a consequence of \autoref{210116200501} (or more specifically of \autoref{210603145334}).
In more generality, a direct proof can be found in \cite[Lemma 3.11]{MY2008.06492}.

%===============================================================================
\begin{example}[mildly perturbed coefficients]{210603122135}
If $\hat{p}, \hat{q}$ are polynomials in $\hbar$, then necessarily $\hat{p}, \hat{q} \in \cal{G} (U) \bbrac{\hbar}$; i.e., they automatically satisfy the hypothesis of \autoref{210603121655}.
\end{example}

Concretely, \autoref{210603121655} says that if the coefficients $p_k, q_k$ of the power series $\hat{p},\hat{q}$ grow no faster than $k!$, then the power series coefficients $\Psi^\pto{k}_\pm$ of the formal WKB solution $\hat{\psi}_\pm$ given by \eqref{210116171149} likewise grow no faster than $k!$.
This is made precise in the following corollary.

%===============================================================================
\begin{cor}[at most factorial growth]{211029102243}
Let $x_0 \in X$ be a regular point and let $U \subset X$ be any simply connected neighbourhood of $x_0$ free of turning points.
Let $\hat{\psi}_\pm = e^{-\Phi_\pm / \hbar} \hat{\Psi}_\pm$ be the formal WKB solutions on $U$ normalised at $x_0$, written as in \eqref{210116171149}.
Take any pair of nested compactly contained subsets $U_0 \Subset U_1 \Subset U$, and suppose that there are real constants $\AA, \BB > 0$ such that
\eqntag{\label{211029102554}
	\big| p_k (x) \big|,
	\big| q_k (x) \big|
	\leq \AA \BB^{k} k!
\qqquad
	(\forall k \geq 0, \forall x \in U_1)
\fullstop
}
Then there are real constants $\CC, \MM > 0$ such that
\eqntag{\label{211029102557}
	\big| \Psi^\pto{k}_\pm (x) \big|
	\leq \CC \MM^{k} k!
\qqqqquad
	(\forall k \geq 0, \forall x \in U_0)
\fullstop
}
In particular, if $\hat{p}, \hat{q}$ are polynomials in $\hbar$, then estimates \eqref{211029102557} hold.
\end{cor}

%===============================================================================
%===============================================================================
%===============================================================================
\section{WKB Geometry}
\label{210127124821}
%===============================================================================
%===============================================================================
%===============================================================================

In this section, we introduce a coordinate transformation which plays a central role in our construction of exact WKB solutions in \autoref{210418111002}.
It is used to determine regions in $\Complex_x$ where the Borel-Laplace method can be applied to our differential equation.

The material of this section can essentially be found in \cite[\S9-11]{MR743423} (see also \cite[\S3.4]{MR3349833}).
These references use the language of foliations given by quadratic differentials on Riemann surfaces, where the quadratic differential in question is $\DD_0 (x) \dd{x}^2$.
The reader may be more familiar with the set of critical leaves of this foliation which is encountered in the literature under various names including \textit{Stokes curves}, \textit{Stokes graph}, \textit{spectral network}, \textit{geodesics}, and \textit{critical trajectories} \cite{MR3003931, MR2182990, MR1209700, MR3115984, 1902.03384}.

To keep the discussion a little more elementary, we state the relevant definitions and facts by appealing directly to explicit formulas using the \textit{Liouville transformation} (defined below) commonly used in the WKB analysis of Schrödinger equations\footnote{See also \cite{fenyes2020complex} for another interesting geometric use of the Liouville transformation in the context of complex projective structures.}.

%===============================================================================
\paragraph{Résumé.}
The following is a quick summary of our conventions and notations.
It is intended for the reader who is quite familiar with this story and who may therefore wish to skip the rest of this section and go directly to \autoref{210418111002}.

Fix a phase $\theta \in \Real / 2\pi\Integer$, a regular point $x_0 \in X$, and a univalued square-root branch $\sqrt{\DD_0}$ near $x_0$.
By a \dfn{Liouville transformation} we mean the following local coordinate transformation near $x_0$:
\eqn{
	z = \Phi (x)
		\coleq \int_{x_0}^x \sqrt{\DD_0 (t)} \dd{t}
		= \int_{x_0}^x \Big(\lambda_+ (t) - \lambda_- (t) \Big) \dd{t}
\fullstop
}
A \dfn{WKB $\theta$-trajectory} through $x_0$ is given locally by the equation
$\Im \big( e^{-i\theta} \Phi (x) \big) = 0$.
It is parameterised by the real number $\tau (x) \coleq \Re \big( e^{-i\theta} \Phi (x) \big)$.
It is called \dfn{complete} if it exists for all real time $\tau$.
A complete WKB $\theta$-trajectory is mapped by $\Phi$ to the entire straight line $e^{i \theta} \Real$.

A \dfn{WKB $(\theta,\pm)$-ray} emanating from $x_0$ is the part of the WKB $\theta$-trajectory corresponding to $\tau \geq 0$ or $\tau \leq 0$, respectively.
It is called \dfn{complete} if it exists for all \textit{nonnegative} real time $\tau \geq 0$ or all \textit{nonpositive} real time $\tau \leq 0$, respectively.
In particular, it is not allowed to flow into a \dfn{finite critical point} (i.e., a turning point or a simple pole of $\DD_0$).
It is mapped by $\Phi$ to the ray $e^{i\theta} \Real_\pm$.

Two special classes of complete trajectories are considered: a \dfn{closed WKB $\theta$-trajectory} which is a closed curve in $X$, and a \dfn{generic WKB $\theta$-trajectory} which tends at both ends to \dfn{infinite critical points} (i.e., poles of $\DD_0$ of order at least $2$).
A \dfn{generic WKB $(\theta, \pm)$-ray} tends to an infinite critical point at one end.

A \dfn{WKB $\theta$-strip domain} is swept out by nonclosed complete WKB $\theta$-trajectories, mapped by $\Phi$ to an infinite strip made up of straight lines parallel to $e^{i \theta} \Real$.
Fact: any generic WKB curve can always be embedded in a WKB strip domain.

A \dfn{WKB $\theta$-ring domain} is swept out by closed WKB $\theta$-trajectories, mapped by $\Phi$ to an infinite strip made up of straight lines parallel to $e^{i \theta} \Real$.
It is homeomorphic to an annulus and $\Phi$ extends to it as a multivalued local biholomorphism.
Fact: any closed WKB trajectory can always be embedded in a WKB ring domain.

A \dfn{WKB $(\theta,\pm)$-halfstrip} is swept out by WKB $(\theta,\pm)$-rays emanating from a \dfn{WKB disc} $\Phi^{-1} \set{ |z - z_1| < \epsilon}$ around a point $x_1$ near $x_0$, where $z_1 \coleq \Phi (x_1)$.
It is the preimage under $\Phi$ of a tubular neighbourhood $\set{z ~\big|~ \op{dist} \big(z, z_1 + e^{i\theta} \Real_\pm \big) < \epsilon}$.

%===============================================================================
%===============================================================================
\subsection{The Liouville Transformation}
%===============================================================================
%===============================================================================

%===============================================================================
\paragraph{}
Recall the leading-order characteristic discriminant $\DD_0 = p_0^2 - 4 q_0$ which is a holomorphic function on $X$.
Fix a phase $\theta \in \Real / 2\pi \Integer$, a regular point $x_0 \in X$, and a univalued square-root branch $\sqrt{\DD_0}$ near $x_0$.
Consider the following local coordinate transformation near $x_0$, called the \dfn{Liouville transformation}:
\eqntag{\label{210119093734}
	z = \Phi (x)
		\coleq \int_{x_0}^x \sqrt{ \DD_0 (t) } \dd{t}
\fullstop
}
The basepoint of integration $x_0$ can in principle be chosen on the boundary of $X$ or at infinity in $\Complex_x$ provided that this integral is well-defined.
This transformation is encountered in the analysis of the Schrödinger equation \eqref{210115121042} as described for example in Olver's textbook \cite[\S6.1]{MR1429619}.
However, note that our formula \eqref{210119093734} in the special case of the Schrödinger equation \eqref{210115121042} reads
\eqntag{
	\Phi (x) = \int\nolimits_{x_0}^x \sqrt{\DD_0 (t)} \dd{t} = 2 \int\nolimits_{x_0}^x \sqrt{\QQ_0 (t)} \dd{t}
\fullstop{,}
}
which differs from formula (1.05) in \cite[\S6.1]{MR1429619} by a factor of $2$.

%===============================================================================
\paragraph{}
If $\lambda_{\pm}$ are the two leading-order characteristic roots defined near $x_0$, labelled such that $\sqrt{\DD_0} = \lambda_+ - \lambda_-$, then we have the following identity relating the Liouville transformation $\Phi$ with the WKB exponents normalised at $x_0$ from \eqref{210223093541}:
\eqntag{
	\Phi 
		= \Phi_+ - \Phi_-
		= \int_{x_0}^x \Big(\lambda_+ (t) - \lambda_- (t) \Big) \dd{t}
\fullstop
}

%===============================================================================
\paragraph{}
If $U \subset X$ is any domain that can support a univalued square-root branch $\sqrt{\DD_0}$ (e.g., if $U$ is simply connected and free of turning points), then the Liouville transformation defines a (possibly multivalued) local biholomorphism $\Phi: U \too \Complex_z$.
The main utility of the Liouville transformation is that it transforms the differential operator $\frac{1}{\sqrt{\DD_0}} \del_x$ (which appears prominently in formula \eqref{210115170017} for the formal characteristic roots) into the constant-coefficient differential operator $\del_z$.
Using the language of differential geometry,
\eqntag{
	\Phi_\ast : \frac{1}{\sqrt{\DD_0}} \del_x 
		\mapstoo \del_z
\fullstop
}
This straightening-out of the local geometry using the Liouville transformation (as explained in the next subsection) will be exploited in our construction of exact WKB solutions in \autoref{210418111002}.
This point of view also makes the troublesome nature of turning points more transparent.

\newpage
%===============================================================================
%===============================================================================
\subsection{WKB Trajectories and Rays}
%===============================================================================
%===============================================================================

%===============================================================================
\paragraph{}
A \dfn{WKB $\theta$-trajectory} passing through $x_0$ is the real $1$-dimensional smooth curve on $X$ locally determined by the equation
\eqntag{
	\Im \big( e^{-i\theta} \Phi (x) \big) = 0
\fullstop{;}
\qqtext{i.e.,}
	\Im \left( \: e^{-i\theta} \int\nolimits_{x_0}^x \sqrt{\DD_0 (t)} \dd{t} \right)
	= 0
\fullstop
}
By definition, WKB trajectories are regarded as being maximal under inclusion.
The Liouville transformation $\Phi$ maps a WKB $\theta$-trajectory to a connected subset of the straight line $e^{i\theta} \Real \subset \Complex_z$.
The image is a possibly unbounded line segment $e^{i\theta} (\tau_{-}, \tau_{+}) \subset e^{i\theta} \Real$ containing the origin $0 = \Phi (x_0)$.
Maximality means that the line segment $e^{i\theta} (\tau_-, \tau_+)$ is the largest possible image.

All other nearby WKB $\theta$-trajectory can be locally described by an equation of the form $\Im \big( e^{-i\theta} \Phi (x) \big) = c$ for some $c \in \Real$.
That is, if $V \subset X$ is a simply connected neighbourhood of $x_0$ free of turning points, then a WKB $\theta$-trajectory intersecting $V$ is locally given by this equation with $c = \Im e^{-i\theta} \Phi (x_1)$ for some $x_1 \in V$.
Its image in $\Complex_z$ under $\Phi$ is an interval of the straight line $\set{z = z_1 + \xi ~\big|~ \xi \in e^{i\theta}\Real}$ containing the point $z_1 \coleq \Phi (x_1)$.

%===============================================================================
\paragraph{}
The chosen square-root $\sqrt{\DD_0}$ near $x_0$ endows the WKB $\theta$-trajectory passing through $x_0$ with a canonical parameterisation given by the real number 
\eqntag{
	\tau (x) \coleq \Re \left( e^{-i\theta}\Phi (x) \right) \in (\tau_-, \tau_+) \subset \Real
\fullstop
}
We define the \dfn{WKB $(\theta, \pm)$-ray} emanating from $x_0$ as the preimage under $\Phi$ of the line segment $e^{i\theta} [0, \tau_{+})$ or $e^{i\theta} (\tau_-, 0]$, respectively.

%===============================================================================
\paragraph{Complete WKB trajectories and rays.}
Suppose that either $\tau_{+}$ or $\tau_{-}$ is finite.
As $\tau$ approaches $\tau_{+}$ or $\tau_{-}$ respectively, the WKB trajectory either tends to a turning point or escapes to the boundary of $X$ in finite time.
If it tends to a single point on the boundary of $X$, this point is either a turning point or a simple pole of the discriminant $\DD_0$, \cite[\S10.2]{MR743423}.
For this reason, turning points and simple poles are sometimes collectively referred to as \dfn{finite critical points}.
These situations are inadmissible for the purpose of constructing exact solutions using our methods, so we introduce the following definitions.

\begin{defn}{210603173853}
A \dfn{complete WKB $\theta$-trajectory} is one for which both $\tau_{+} = + \infty$ and $\tau_{-} = - \infty$; i.e., its image in $\Complex_z$ under the Liouville transformation $\Phi$ is the entire straight line $e^{i\theta} \Real$.
A \dfn{complete WKB $(\theta,\pm)$-ray} is one for which $\tau_{\pm} = \pm \infty$, respectively; i.e, its image under $\Phi$ is the entire ray $e^{i\theta} \Real_+$ or $e^{i\theta} \Real_-$, respectively.
\end{defn}

%===============================================================================
%===============================================================================
\subsection{Closed and generic WKB trajectories}
%===============================================================================
%===============================================================================

Two classes of complete trajectories are especially important.

%===============================================================================
\paragraph{}
A \dfn{closed WKB $\theta$-trajectory} is one with the property that there is a nonzero time $\omega \in \Real$ such that $\Phi^{-1} (e^{i\theta}\omega) = \Phi^{-1} (0)$.
This only happens when the Liouville transformation is analytically continued along the trajectory to a multivalued function (WKB trajectories are smooth so they cannot have self-intersections).
Closed WKB trajectories are necessarily complete and form closed curves in $\Complex_x$, \cite[\S9.2]{MR743423}.
We refer to the smallest possible positive such $\omega \in \Real_+$ as the \dfn{trajectory period}.

Consider now a complete WKB $(\theta,\pm)$-ray emanating from $x_0$ which is not part of a closed WKB trajectory.
Its \dfn{limit set} is by definition the limit of the set 
\eqn{
	\bar{\Phi^{-1} \big( e^{i\theta} [\tau,+\infty) \big)}
\qtext{as}
	\tau \to +\infty
\qqtext{or}
	\bar{\Phi^{-1} \big( e^{i\theta} (-\infty,\tau] \big)}
\qtext{as}
	\tau \to -\infty
\fullstop
}
Obviously, this definition is independent of the chosen basepoint $x_0$ along the trajectory.
The limit set may be empty or it may contain one or more points.
If it contains a single point $x_\infty \in \Complex_x$, then this point (sometimes called an \dfn{infinite critical point}) is necessarily a pole of $\DD_0$ of order $m \geq 2$, \cite[\S10.2]{MR743423}.
A \dfn{generic WKB ray} is one whose limit set is an infinite critical point.
A \dfn{generic WKB trajectory} is one both of whose rays are generic.

%===============================================================================
%===============================================================================
\subsection{WKB Strips, Halfstrips, and Ring Domains}
%===============================================================================
%===============================================================================

%===============================================================================
\paragraph{}
For us, the model neighbourhood of a complete WKB $\theta$-trajectory is the preimage under $\Phi$ of an infinite strip in the $z$-plane, which is a subset of the form $\set{z ~\big|~ \epsilon_- < \Im (e^{-i\theta} z) < \epsilon_+}$.
Such a neighbourhood may not exist, but if it does, it is swept out by complete WKB $\theta$-trajectories.
We consider separately the situations when these complete trajectories are closed or not.

%===============================================================================
\paragraph{}
\label{210519211406}
A \dfn{WKB $\theta$-strip domain} is the preimage under $\Phi$ of an infinite strip swept out by \textit{nonclosed} trajectories.
It is simply connected and $\Phi$ maps it to an infinite strip biholomorphically.
If one WKB trajectory in a WKB strip domain is generic, then all trajectories sweeping out this strip are generic.
The main fact we need about generic WKB trajectories is that any generic WKB $\theta$-trajectory can be embedded in a WKB $\theta$-strip domain \cite[\S10.5]{MR743423}.

%===============================================================================
\paragraph{}
A \dfn{WKB $\theta$-ring domain} is an open subset of $\Complex_x$, homeomorphic to an annulus, swept out by closed WKB $\theta$-trajectories.
The restriction of the Liouville transformation $\Phi$ is a multivalued holomorphic function.
All closed WKB trajectories sweeping out a WKB ring domain have the same trajectory period.
The main fact we need about closed WKB trajectories is that any closed WKB $\theta$-trajectory can be embedded in a WKB $\theta$-ring domain \cite[\S9.3]{MR743423}.

%===============================================================================
\paragraph{}
\label{210524105411}
Similarly, the model neighbourhood of a complete WKB $(\theta,\pm)$-ray is the preimage under $\Phi$ of a tubular neighbourhood of a ray $z_0 + e^{i\theta} \Real_\pm$ for some $z_0 \in \Complex_z$.
Such a tubular neighbourhood is the subset of the form $\set{z ~\big|~ \op{dist} (z, z_0 + e^{i\theta} \Real_\pm) < \epsilon}$.
We refer to its preimage under $\Phi$ as a \dfn{WKB $(\theta, \pm)$-halfstrip domain}.
It is swept out by complete WKB $(\theta, \pm)$-rays emanating from a \dfn{WKB disc} around $x_0$ of radius $\epsilon$; i.e., the set $V = \Phi^{-1} \big( \set{ |z| < \epsilon} \big)$.
Any generic WKB $(\theta, \pm)$-ray can be embedded in a WKB $(\theta, \pm)$-halfstrip domain.

%==============================
\begin{figure}[t]
\centering
\begin{subfigure}[b]{0.47\textwidth}
    \centering
    \includegraphics{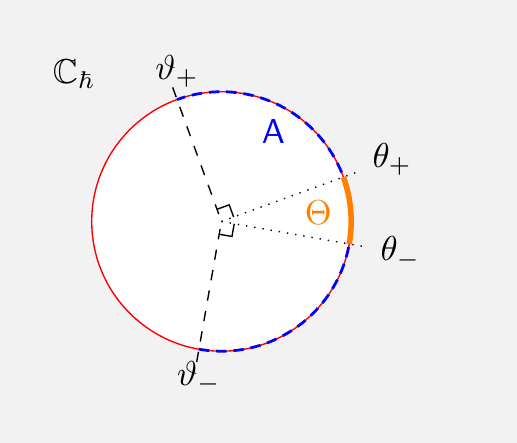}
    \caption{The arc $\Theta$ of copolar directions of $A$.}
    \label{210618101426}
\end{subfigure}
\quad
\begin{subfigure}[b]{0.47\textwidth}
    \centering
    \includegraphics{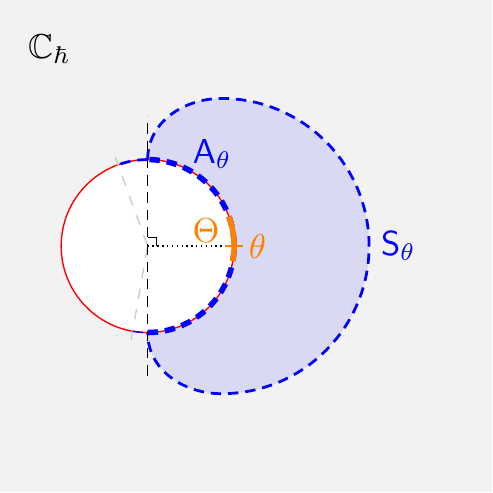}
    \caption{A Borel disc $S_\theta$ with opening $A_\theta$.}
    \label{210618105212}
\end{subfigure}
\caption{}
\end{figure}
%==============================

%==============================
\begin{figure}[t]
\centering
\begin{subfigure}[b]{\textwidth}
    \centering
\includegraphics[width=\textwidth]{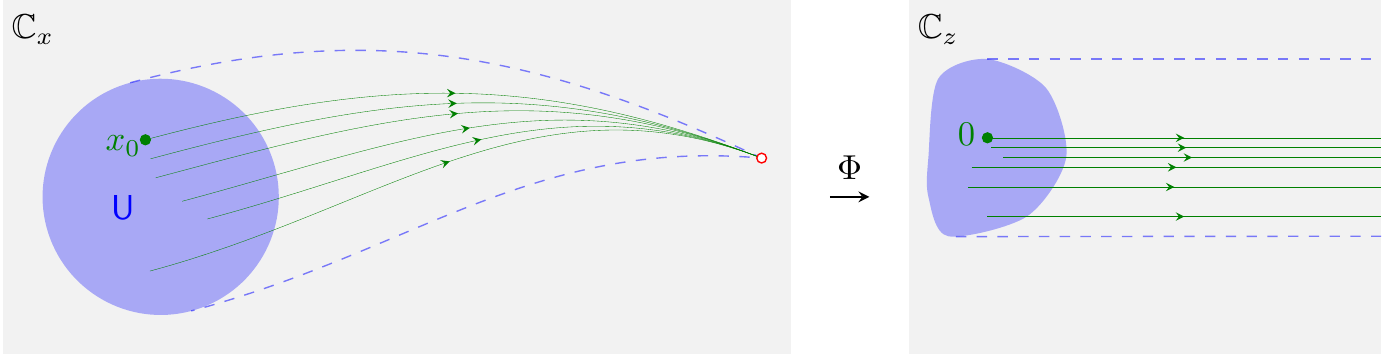}
\caption{Domain $U \subset X$ (in blue) is a simply connected neighbourhood of $x_0$ that is free of turning points.
Every WKB $(\theta,\alpha)$-ray emanating from $U$ is complete.
A few randomly chosen rays are represented by green curves, and completeness is represented by having them all flow into e.g. a pole of $\DD_0$ (red dot).}
\label{210618111906}
\end{subfigure}
\newline%
\begin{subfigure}[b]{\textwidth}
    \centering
    \includegraphics{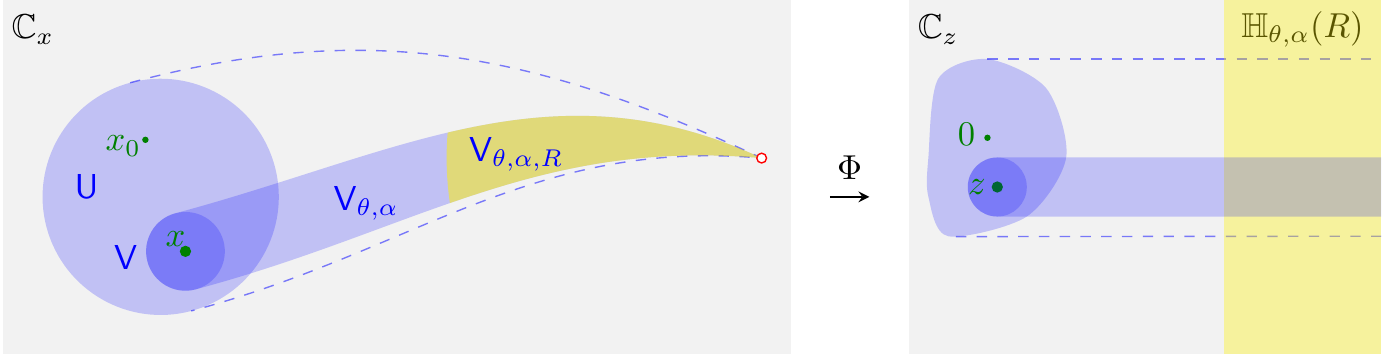}
    \caption{Every $x \in U$ has a neighbourhood $V \subset U$ which flows into the red dot.
    The full flow is denoted by $V_{\theta, \alpha}$.
    The halfplane $\mathbb{H}_{\theta, \alpha} (\RR)$ on the right is drawn in yellow, and its preimage in $V_{\theta, \alpha}$ is also drawn in yellow and thought of as a (sectorial) neighbourhood of the limit set.}
    \label{210618114017}
\end{subfigure}
\caption{}
\label{210618121411}
\end{figure}
%==============================

%===============================================================================
%===============================================================================
%\newpage
\section{Exact WKB Solutions}
\label{210418111002}
%===============================================================================
%===============================================================================

In this section, we state and prove the main results of this paper.

%===============================================================================
\paragraph{Background assumptions.}
\label{210612115147}
%===============================================================================
We remain in the setting of \autoref{210527190843}.
Throughout this section, we also fix a regular point $x_0 \in X$ and a univalued square-root branch $\sqrt{\DD_0}$ near $x_0$.
Let $\lambda_+, \lambda_-$ be the two leading-order characteristic roots given by \eqref{210305095811} so that $\sqrt{\DD_0} = \lambda_+ - \lambda_-$.
Let $\hat{\psi}_+, \hat{\psi}_-$ be the corresponding pair of formal WKB solutions normalised at $x_0$ as guaranteed by \autoref{210118113644}.
Let $z = \Phi (x)$ be the Liouville transformation with basepoint $x_0$ given by \eqref{210119093734}.

In addition, let $\Theta \coleq [\theta_-, \theta_+]$ be the closed arc such that $A = (\theta_- - \tfrac{\pi}{2}, \theta_+ + \tfrac{\pi}{2})$; i.e., $\theta_\pm \coleq \vartheta_\pm \mp \tfrac{\pi}{2}$.
See \autoref{210618101426}.
This arc $\Theta$ is sometimes called the arc of \textit{copolar directions} of $A$.
For every $\theta \in \Theta$, let $A_\theta \subset A$ be the halfplane arc bisected by $\theta$, and let $S_\theta \subset S$ be a Borel disc bisected by $\theta$ of some diameter $\delta > 0$ (see \autoref{210618105212}):
\eqntag{\label{210601181338}
	A_\theta \coleq (\theta -\tfrac{\pi}{2}, \theta +\tfrac{\pi}{2})
\qtext{and}
	S_\theta \coleq \set{ \Re \big( \smash{e^{i\theta}} / \hbar \big) > 1/\delta}
\fullstop
}
Finally, for any $\RR > 0$, $\theta \in \Theta$, let $\mathbb{H}_{\theta, \pm} (\RR) \coleq \set{ \Re (\pm e^{i\theta} z) > \RR } \subset \Complex_z$ (see \autoref{210618114017}).

%===============================================================================
%===============================================================================
\subsection{Existence and Uniqueness of Exact WKB Solutions}
\label{210612120131}
%===============================================================================
%===============================================================================

First, we investigate our problem for a single fixed direction in $\Theta$.
The main result of this paper is then the following theorem (see \autoref{210618121411} for a visual).

%===============================================================================
\begin{thm}[\textbf{Existence and Uniqueness in a Halfplane}]{210116200501}
\leavevmode\newline
Fix a sign $\alpha \in \set{+, -}$ and a phase $\theta \in \Theta$, and let $U \subset X$ be any simply connected domain containing $x_0$ which is free of turning points and such that every WKB $(\theta,\alpha)$-ray emanating from $U$ is complete.
Assume in addition that for every point $x \in U$, there is a neighbourhood $V \subset U$ of $x$ and a sufficiently large number $\RR > 0$ such that the following two conditions are satisfied on the domain $V_{\theta,\alpha, \RR} \coleq V_{\theta,\alpha} \cap \Phi^{-1} \Big( \mathbb{H}_{\theta, \alpha} (\RR) \Big)$, where $V_{\theta,\alpha}$ is the union of all WKB $(\theta,\alpha)$-rays emanating from $V$:
\begin{enumerate}
\item $\frac{1}{\sqrt{\DD_0}} \del_x \log \sqrt{\DD_0}$ is bounded on $V_{\theta,\alpha, \RR}$;
\item $\tfrac{1}{\sqrt{\DD_0}} p \simeq \tfrac{1}{\sqrt{\DD_0}} \hat{p}
\qtext{and}
	\tfrac{1}{\DD_0} q \simeq \tfrac{1}{\DD_0} \hat{q}
\quad
\text{as $\hbar \to 0$ along $\bar{A}_\theta$, unif. $\forall x \in V_{\theta,\alpha, \RR}$.}$
\end{enumerate}

\smallskip

Then the differential equation \eqref{210115121038} has a canonical exact solution $\psi_\alpha$ on $U$ whose exponential asymptotics as $\hbar \to 0$ along $A_\theta$ are given by the formal WKB solution $\hat{\psi}_\alpha$.
Namely, there is a Borel disc $S'_\theta \subset S_\theta$ of possibly smaller diameter $\delta' \in (0, \delta]$ such that \eqref{210115121038} has a unique holomorphic solution $\psi_\alpha$ defined on $U \times S'_\theta$ that satisfies the following conditions:
\eqnstag{
\label{210516153316}
	&\psi_\alpha (x_0, \hbar) = 1
	\qquad \text{for all $\hbar \in S'_\theta$\fullstop{;}}
\\
\label{210516153319}
	&\psi_\alpha (x,\hbar) \simeq \hat{\psi}_\alpha (x,\hbar)
\quad
\text{as $\hbar \to 0$ along $\bar{A}_\theta$, loc.unif. $\forall x \in U$
\fullstop}
}
Furthermore, if the hypotheses hold for both choices of the sign $\alpha$, then $\psi_+, \psi_-$ on $U$ define a basis for the space $\mathbb{ES} (U; A)$ of all exact solutions on $(U;A)$ as well as for the space $\mathbb{GES} (U; A)$ of all exact solutions on $(U;A)$ with exponential Gevrey asymptotics.
\end{thm}

The strategy of the proof of this theorem is to reduce the problem to finding exact solutions to an associated singularly perturbed Riccati equation as captured by the following lemma.

%===============================================================================
\begin{lem}[\textbf{WKB ansatz and the associated Riccati equation}]{210603110237}
The solution $\psi_\alpha$ from \autoref{210116200501} is given by the following formula: for all $(x, \hbar) \in U \times S'_\theta$,
\eqntag{
\label{210218230927}
	\psi_\alpha (x, \hbar)
		= \exp \left( - \frac{1}{\hbar} \int\nolimits_{x_0}^x s_\alpha (t, \hbar) \dd{t} \right)
\fullstop{,}
}
where $s_\alpha = \lambda_\alpha + \hbar \SS_\alpha$ is an exact solution of the singularly perturbed Riccati equation 
\begin{equation}
\label{210304143646}
	\hbar \del_x s = s^2 - ps + q
\fullstop
\end{equation}
More precisely, for any compactly contained subset $V \subset U$, there is a Borel disc $S''_\theta \subset S'_\theta$ of possibly smaller diameter $\delta'' \in (0,\delta']$ such that identity \eqref{210218230927} holds for all $(x, \hbar)$ in the domain $V \times S''_\theta$.
Here, $s_\alpha$ is the unique holomorphic solution of the Riccati equation \eqref{210304143646} on $V \times S''_\theta$ which admits the formal characteristic solution $\hat{s}_\alpha$ as its uniform Gevrey asymptotic expansion along $\bar{A}_\theta$:
\eqntag{\label{210603112057}
	s_\alpha (x,\hbar) \simeq \hat{s}_\alpha (x,\hbar)
\quad
	\text{as $\hbar \to 0$ along $\bar{A}_\theta$, unif. $\forall x \in V$
\fullstop}
}
Moreover, $s_\alpha$ uniquely extends to a meromorphic function on $U \times S'_\theta$ with poles only at the zeros of $\psi_\alpha$.
\end{lem}

Thus, our proof of \autoref{210116200501} mainly rests on the ability to construct a unique exact solution of the Riccati equation \eqref{210304143646}.
In \cite[Theorem 5.17]{MY2008.06492}, we proved a general existence and uniqueness theorem for exact solutions of a singularly perturbed Riccati equation.
This result and the sketch of its proof (specialised to our situation at hand) is presented in \autoref{210525082957}.
The proof of \autoref{210603110237} and therefore of \autoref{210116200501} is then presented in \autoref{210525083210}.

%===============================================================================
\begin{defn}{210523120000}
The solution $\psi_\alpha$ from \autoref{210116200501} is called an \dfn{exact WKB solution} normalised at $x_0$.
The basis $\set{\psi_+, \psi_-}$ is called the \dfn{exact WKB basis} normalised at $x_0$.
We will also refer to the exact solution $s_\alpha$ of the Riccati equation \eqref{210304143646} as an \dfn{exact characteristic root} for the differential equation \eqref{210115121038}.
\end{defn}

%\enlargethispage{20pt}
%===============================================================================
\begin{example}[Mildly deformed coefficients]{210614221544}
If the coefficients $p,q$ of the differential equation \eqref{210115121038} are at most polynomial in $\hbar$ (i.e., if $p,q \in \cal{O} (X) [\hbar]$) then condition (2) of \autoref{210116200501} simplifies considerably and can be replaced by the following equivalent condition: for every $k \geq 0$, the functions $p_k (x)$ and $q_k (x)$ are bounded on $V_{\theta, \alpha, \RR}$ respectively by $\sqrt{\DD_0 (x)}$ and $\DD_0 (x)$.
For the Schrödinger equation \eqref{210115121042} with $\hbar$-independent potential (i.e., with $p = 0$ and $q = -\QQ$ where $\QQ (x, \hbar) = \QQ_0 (x)$), condition (2) of \autoref{210116200501} is vacuous.
\end{example}

%===============================================================================
\paragraph{Analytic continuation to larger domains in $\Complex_x$.}
Using the usual Parametric Existence and Uniqueness Theorem for linear ODEs (see, e.g., \cite[Theorem 24.1]{MR0460820}), exact WKB solutions can be analytically continued anywhere in $X$, yielding the following statement.

%\newpage
%===============================================================================
\begin{cor}{210524175026}
For any simply connected domain $X' \subset X$ containing $x_0$, the exact WKB solution $\psi_\alpha \in \mathbb{ES} (U \times S'_\theta)$ from \autoref{210116200501} extends to a holomorphic solution on $X' \times S'_\theta$.
In fact, it is the unique holomorphic solution of \eqref{210115121038} on $X' \times S'_\theta$ satisfying the following initial conditions for all $\hbar \in S'_\theta$:
\begin{equation}
	\psi_\alpha (x_0, \hbar) = 1
\qqtext{and}
	\hbar \del_x \psi_\alpha (x_0, \hbar) = s_\alpha (x_0, \hbar)
\fullstop
\end{equation}
If $\psi_+, \psi_- \in \mathbb{ES} (U \times S'_\theta)$ is a WKB basis from \autoref{210116200501}, it defines a basis for the space $\mathbb{S} (X' \times S'_\theta)$ of all holomorphic solutions on $X' \times S'_\theta$.
\end{cor}

However, beware that the asymptotic property \eqref{210516153319} of exact WKB solutions is not necessarily continued outside the domain $U$, and indeed in general there are no global exact solutions even if $X$ is simply connected and free of turning points.

%===============================================================================
\paragraph{Existence and uniqueness in wider sectors.}
Now we extend \autoref{210116200501} to the full arc $A$.
Without loss of generality, we can assume that $S \coleq \Cup_{\theta \in \Theta} S_\theta$ is the union of Borel discs $S_\theta = \set{ \Re \big( \smash{e^{i\theta}} / \hbar \big) > 1/\delta}$, one for each bisecting direction $\theta \in \Theta$, of some diameter $\delta > 0$ independent of $\theta$.

%===============================================================================
\begin{thm}{210527195440}
Fix a sign $\alpha \in \set{+, -}$, and let $U \subset X$ be any simply connected domain containing $x_0$ which is free of turning points and has the following property: for all $\theta \in \Theta$, every WKB $(\theta,\alpha)$-ray emanating from $U$ is complete.
Assume in addition that for every point $x \in U$, there is a neighbourhood $V \subset U$ of $x$ and a sufficiently large number $\RR > 0$ such that the following two conditions are satisfied on the domain 
$V_{\Theta,\alpha, \RR} \coleq V_{\Theta,\alpha} \cap \Phi^{-1} \Big( ~ \Cup_{\theta \in \Theta} \mathbb{H}_{\theta, \alpha, \RR} \Big)$, where $V_{\Theta,\alpha}$ is the union of all WKB $(\theta,\alpha)$-rays emanating from $V$ for all $\theta \in \Theta$:
\begin{enumerate}
\item $\frac{1}{\sqrt{\DD_0}} \del_x \log \sqrt{\DD_0}$ is bounded on $V_{\Theta,\alpha, \RR}$;
\item $\tfrac{1}{\sqrt{\DD_0}} p \simeq \tfrac{1}{\sqrt{\DD_0}} \hat{p}
\qtext{and}
	\tfrac{1}{\DD_0} q \simeq \tfrac{1}{\DD_0} \hat{q}
\quad
\text{as $\hbar \to 0$ along $\bar{A}$, unif. $\forall x \in V_{\Theta,\alpha, \RR}$.}$
\end{enumerate}

\smallskip

Then the differential equation \eqref{210115121038} has a canonical exact solution $\psi_\alpha$ on $U$ whose exponential asymptotics as $\hbar \to 0$ along $A$ are given by the formal WKB solution $\hat{\psi}_\alpha$.
Namely, there is a sectorial subdomain
\eqn{
	S' \coleq \Cup_{\theta \in \Theta} S'_\theta \subset S
\qtext{with}
	S'_\theta \coleq \set{ \Re \big( \smash{e^{i\theta}} / \hbar \big) > 1/\delta'}
}
for some $\delta' \in (0, \delta]$ such that the differential equation \eqref{210115121038} has a unique holomorphic solution $\psi_\alpha$ defined on $U \times S'$ that satisfies the following conditions:
\eqnstag{
\label{210527201005}
	&\psi_\alpha (x_0, \hbar) = 1
	\qquad \text{for all $\hbar \in S'$\fullstop{;}}
\\
\label{210527201013}
	&\psi_\alpha (x,\hbar) \simeq \hat{\psi}_\alpha (x,\hbar)
\quad
\text{as $\hbar \to 0$ along $\bar{A}$, loc.unif. $\forall x \in U$
\fullstop}
}
Furthermore, $\psi_\alpha$ is again given by the formula \eqref{210218230927} for all $(x,\hbar) \in U \times S'$ where $s_\alpha$ is the unique exact solution on $U$ of the Riccati equation \eqref{210304143646} with leading-order $\lambda_\alpha$.
\end{thm}

\begin{proof}
Let $V$ be a neighbourhood of $x_0$ with all the properties in the hypothesis.
For every $\theta \in \Theta$, let $\psi_{\theta, \alpha}$ be the exact WKB solution on $V \times S'_\theta$ guaranteed by \autoref{210116200501}.
Then we can use the uniqueness property to argue that all $\psi_{\theta, \alpha}$ `glue together' to the desired exact WKB solution $\psi_\alpha$.
Indeed, for any $\theta_1, \theta_2 \in \Theta$ with $|\theta_1 - \theta_2| < \pi$, $S'_{\theta_1} \cap S'_{\theta_2} \neq \emptyset$.
By uniqueness, $\psi_{\theta_1, \alpha}$ and $\psi_{\theta_2, \alpha}$ must agree for all $\hbar \in S'_{\theta_1} \cap S'_{\theta_2} \neq \emptyset$ and all $x \in V$, and therefore extend to $V \times (S'_{\theta_1} \cup S'_{\theta_2})$.
Since $\Theta$ is closed, each $S'_{\theta}$ can be chosen to have the same diameter $\delta' \in (0, \delta]$.
\end{proof}

%===============================================================================
\begin{rem}{210603101957}
\autoref{210116200501} is a special case of \autoref{210527195440} with $\Theta = \set{\theta}$.
\end{rem}

%===============================================================================
%===============================================================================
\subsection{Borel Summability of WKB Solutions}
\label{210418171353}
%===============================================================================
%===============================================================================

In this subsection, we translate \autoref{210116200501} and its method of proof into the language of Borel-Laplace theory, the basics of which are briefly recalled in \autoref{210616130753}.
Namely, it follows directly from our construction that the exact WKB solutions are the Borel resummation of the corresponding formal WKB solutions.
In what follows, we make this statement precise and explicit.

%===============================================================================
\paragraph{}
Recall that we write the formal WKB solution $\hat{\psi}_\alpha$ as
\eqnstag{\label{210618125736}
	\hat{\psi}_\alpha (x, \hbar) 
		&= e^{-\Phi_\alpha (x) / \hbar} \; \hat{\Psi}_\alpha (x, \hbar)
\\
\label{210618125738}
		&= \exp \left( - \frac{1}{\hbar} 
				\int\nolimits_{x_0}^x \lambda_\alpha (t) \dd{t} \right)
			\exp \left( 
				- \int\nolimits_{x_0}^x \hat{\SS}_\alpha (t, \hbar) \dd{t} \right)
\fullstop{,}
}
The main result in this subsection is the following theorem.

%===============================================================================
\begin{thm}{210221112347}
Assume all the hypotheses of \autoref{210116200501}.
The exact WKB solution $\psi_\alpha$ on $U \times S'_\theta$ is the locally uniform Borel resummation in the direction $\theta$ of the formal WKB solution $\hat{\psi}_\alpha$ on $U$: for all $x \in U$ and all sufficiently small $\hbar \in S'_\theta$,
\eqntag{\label{210221151122}
\begin{aligned}
	\psi_\alpha (x, \hbar)
		&= \cal{S}_\theta \big[ \, \hat{\psi}_\alpha \, \big] (x, \hbar)
\\
		&= \exp \left( - \frac{1}{\hbar} \int\nolimits_{x_0}^x \lambda_\alpha (t) \dd{t} \right)
		\exp \left( 
			- \int\nolimits_{x_0}^x \cal{S}_\theta \big[ \, \hat{\SS}_\alpha  \, \big] (t, \hbar) \dd{t} \right)
\fullstop
\end{aligned}
}
To be more precise, for every compactly contained subset $V \subset U$, there is a Borel disc $S''_\theta \subset S'_\theta$ of possibly smaller diameter $\delta'' \in (0,\delta']$ such that identity \eqref{210221151122} is valid uniformly for all $(x, \hbar) \in V \times S''_\theta$.
\end{thm}

Thus, exact and formal WKB solutions can be thought of as being canonically specified by their asymptotic expansions, and hence in some sense `identified'.
This is the reason that the vast majority of literature in WKB analysis speaks simply of ``WKB solutions'' without specifying whether the exact or the formal object is in question.
However, it is important to stress that this `identification' is not global and highly depends on the location in $\Complex_x$.

\autoref{210221112347} packs a lot of information, which we now unpack by breaking it down into a sequence of four lemmas (\autoref{210603145334}, \ref{210603150108}, \ref{210603150616}, and \ref{210603152031}), all of which follow immediately from the proof of \autoref{210116200501}.

%===============================================================================
\paragraph{The formal Borel transform.}
Recall that the formal WKB solution $\hat{\psi}_\alpha$ is an exponential power series on $U$; in symbols, $\hat{\psi}_\alpha \in \cal{O}^{\exp} (U) \bbrac{\hbar}$.
Instead of considering directly the Borel transform $\hat{\Borel} \big[ \, \hat{\psi}_\alpha \, \big]$ of $\hat{\psi}_\alpha$, it is more convenient to consider the following \textit{exponentiated formal Borel transform} of $\hat{\psi}_\alpha$, which we define as
\eqntag{\label{210526162424}
	\hat{\Borel}^{\exp} \big[ \, \hat{\psi}_\alpha \, \big] (x, \xi)
		\coleq \exp \left( - \frac{1}{\hbar} 
			\int\nolimits_{x_0}^x \lambda_\alpha (t) \dd{t} \right)
		\exp \left( 
			- \int\nolimits_{x_0}^x \hat{\Borel} \big[ \, \hat{\SS}_\alpha \, \big] (t, \xi) \dd{t} \right)
\fullstop
}
Note that $\hat{\Borel} \big[ \, \hat{\psi}_\alpha \, \big] \neq \hat{\Borel}^{\exp} \big[ \, \hat{\psi}_\alpha \, \big]$ because the Borel transform converts multiplication into convolution.
Denote the formal Borel transform of $\hat{\SS}_\alpha$ by
\eqntag{\label{210526162605}
	\hat{\sigma}_\alpha (x, \xi)
		\coleq \hat{\Borel} \big[ \, \hat{\SS}_\alpha \, \big] (x, \xi)
		= \sum_{k=0}^\infty s_{k+2}^\pto{i} (x) \frac{\xi^k}{k!}
\fullstop
}

%===============================================================================
\begin{lem}[\textbf{Convergence of the formal Borel transform}]{210603145334}
The formal Borel transform $\hat{\sigma}_\alpha$ is a locally uniformly convergent power series in $\xi$.
Consequently, the exponentiated formal Borel transform $\hat{\Borel}^{\exp} \big[ \, \hat{\Psi}_\alpha \, \big]$ and hence the ordinary formal Borel transform $\hat{\Borel} \big[ \, \hat{\Psi}_\alpha \, \big]$ are locally uniformly convergent power series in $\xi$.
In symbols, $\hat{\sigma}_\alpha, \hat{\Borel}^{\exp} \big[ \, \hat{\Psi}_\alpha \, \big], \hat{\Borel} \big[ \, \hat{\Psi}_\alpha \, \big] \in \cal{O} (U) \set{\xi}$.
\end{lem}

%===============================================================================
\paragraph{The analytic Borel transform.}
Similarly, it is more convenient to consider the \textit{exponentiated analytic Borel} transform of the exact WKB solution $\psi_\alpha$ in the direction $\theta$, defined as
\eqntag{\label{210526163114}
	\Borel^{\exp}_\theta \big[ \, \psi_\alpha \, \big] (x, \xi)
		\coleq \exp \left( - \frac{1}{\hbar} \int\nolimits_{x_0}^x \lambda_\alpha (t) \dd{t} \right)
		\exp \left( 
			- \int\nolimits_{x_0}^x \Borel_\theta \big[ \, \SS_\alpha \, \big] (t, \xi) \dd{t} \right)
\fullstop
}
Note again that $\Borel_\theta \big[ \, \psi_\alpha \, \big] \neq \Borel^{\exp}_\theta \big[ \, \psi_\alpha \, \big]$ for the same reason as above.
However, if $\Borel^{\exp}_\theta \big[ \, \psi_\alpha \, \big]$ is holomorphic at some point $(x, \xi)$ then clearly so is $\Borel_\theta \big[ \, \psi_\alpha \, \big]$, and therefore we can deduce a lot of information about the Borel transform from the exponentiated Borel transform.
Denote the analytic Borel transform of $\SS_\alpha$ in the direction $\theta$ by
\eqntag{\label{210603151704}
	\sigma_\alpha (x, \xi) 
		\coleq \Borel_\theta [\, \SS_\alpha \,] (x, \xi)
		= \frac{1}{2\pi i} \oint\nolimits_\theta \SS_\alpha (x, \hbar) e^{\xi / \hbar} \frac{\dd{\hbar}}{\hbar^2}
\fullstop
}

%===============================================================================
\begin{lem}[\bfseries Convergence of the analytic Borel transform]{210603150108}
For any compactly contained subset $V \Subset U$, there exists an $\epsilon > 0$ such that the analytic Borel transform $\sigma_\alpha$ is uniformly convergent for all $(x, \xi) \in V \times \Xi_\theta$ where $\Xi_\theta \coleq \set{ \xi ~\big|~ \op{dist} \big(\xi, e^{i\theta} \Real_+ \big) < \epsilon}$.
Consequently, the exponentiated analytic Borel transform $\Borel^{\exp}_\theta \big[ \, \psi_\alpha \, \big]$ and hence the analytic Borel transform $\Borel_\theta \big[ \, \psi_\alpha \, \big]$ are uniformly convergent for all $(x, \xi) \in V \times \Xi_\theta$.
In particular, $\Borel^{\exp}_\theta \big[ \, \psi_\alpha \, \big]$ and $\Borel_\theta \big[ \, \psi_\alpha \, \big]$ are convergent for all $\xi \in e^{i\theta} \Real_+$, locally uniformly for all $x \in U$.
\end{lem}

%===============================================================================
\begin{lem}[\bfseries Analytic continuation of the formal Borel transform]{210603150616}
\leavevmode\newline
The analytic Borel transform $\sigma_\alpha$ defines the analytic continuation of the formal Borel transform $\hat{\sigma}_\alpha$ along the ray $e^{i\theta} \Real_+ \subset \Complex_\xi$.
Consequently, the exponentiated analytic Borel transform $\Borel^{\exp}_\theta \big[ \, \psi_\alpha \, \big]$ and hence the analytic Borel transform $\Borel_\theta \big[ \, \psi_\alpha \, \big]$ define the analytic continuations along the ray $e^{i\theta} \Real_+$ of $\hat{\Borel}^{\exp} \big[ \, \hat{\psi}_\alpha \, \big]$ and $\hat{\Borel} \big[ \, \hat{\psi}_\alpha \, \big]$, respectively.
In particular, there are no singularities in the Borel plane $\Complex_\xi$ along the ray $e^{i\theta} \Real_+$.
\end{lem}

Let us define the \textit{exponentiated Laplace transform} $\Laplace^{\exp}_\theta$ the same way by applying the ordinary Laplace transform $\Laplace_\theta$ to the exponent.

%\newpage
%===============================================================================
\begin{lem}[\bfseries Borel-Laplace identity for WKB solutions]{210603152031}
For any compactly contained subset $V \subset U$, there is a Borel disc $S''_\theta \subset S_\theta$ of possibly smaller diameter $\delta'' \in (0, \delta']$ such that the Laplace transform of $\sigma_\alpha$ in the direction $\theta$ is uniformly convergent for all $(x, \hbar) \in V \times S''_\theta$ and satisfies the following identity:
\eqn{
	\SS_\alpha (x, \hbar)
		= s_\alpha^\pto{1} (x) + \Laplace_\theta [ \, \sigma_\alpha  \, ] (x, \hbar)
		= s_\alpha^\pto{1} (x) + \int\nolimits_{e^{i\theta} \Real_+} e^{-\xi/\hbar} \sigma_\alpha (x, \xi) \dd{\xi}
\fullstop
}
Consequently, the exponentiated Laplace transform of $\Borel_\theta^{\exp} \big[ \, \psi_\alpha \, \big]$ is uniformly convergent for all $(x, \hbar) \in V \times S''_\theta$ and satisfies the following identity:
\eqn{
\begin{aligned}
	\psi_\alpha (x, \hbar)
		&= e^{-\Phi_\alpha (x) / \hbar}
			\,
			e^{- s_\alpha^\pto{1} (x)}
			\,
			\Laplace^{\exp}_\theta \Big[ \, \Borel_\theta^{\exp} \big[ \, \psi_\alpha \, \big] \, \Big]
\\		&= \exp \left( - \frac{1}{\hbar} \int\nolimits_{x_0}^x \lambda_\alpha (t) \dd{t} \right)
		\exp \left( - \int\nolimits_{x_0}^x \SS_\alpha (t, \xi) \dd{t} \right)
\fullstop
\end{aligned}
}
In particular, $\Laplace_\theta \big[ \, \Borel_\theta [ \, \psi_\alpha \, ] \, \big]$ is uniformly convergent on $V \times S''_\theta$ and equals $\psi_\alpha$.
\end{lem}

%===============================================================================
%===============================================================================
\subsection{Explicit Formula for the Borel Transform}
%===============================================================================
%===============================================================================

Thanks to the explicit nature of our construction of the exact WKB solutions, we can write down an explicit recursive formula for the analytic continuation of the exponentiated formal Borel transform \eqref{210526162424}.

%===============================================================================
\paragraph{}
To this end, it is convenient to introduce the following expressions.
First, we factorise $\SS_\alpha$ as follows:
\eqntag{
	\SS_\alpha = s_\alpha^\pto{1} + \varepsilon_\alpha \sqrt{\DD_0} \TT_\alpha
\qqtext{so that}
	\sigma_\alpha = \varepsilon_\alpha \sqrt{\DD_0} \tau_\alpha
\fullstop{,}
}
where $\varepsilon_\pm \coleq \pm 1$, $\TT_\alpha = \TT_\alpha (x, \hbar)$ is defined by this equality, and $\tau_\alpha \coleq \Borel_\theta \big[ \, \TT_\alpha \, \big]$.
We also define functions $p_\ast, q_\ast$ of $(x,\hbar)$ using the identities $p = p_0 + p_1 \hbar + p_\ast \hbar^2$ and $q = q_0 + q_1 \hbar + q_\ast \hbar^2$.
Next, introduce the following expressions:
\eqntag{\label{210527121916}
\small
\begin{aligned}
	\BB_0 &\coleq \frac{q_\ast - p_\ast (\lambda_\alpha + \hbar s_\alpha^\pto{1}) - q_2 + p_2 \lambda_\alpha}{\DD_0}
\qtext{and}
	\BB_1 \coleq \frac{\hbar p_\ast}{\varepsilon_\alpha \sqrt{\DD_0}},
\\
	b_0 &\coleq \frac{q_2 - p_2 \lambda_\alpha + (s_\alpha^\pto{1})^2 - \del_x s_\alpha^\pto{1} - p_1 s_\alpha^\pto{1}}{\DD_0}
\qtext{and}
	b_1 \coleq \frac{p_1 - 2s_\alpha^\pto{1} - \del_x \log \sqrt{\DD_0}}{\varepsilon_\alpha \sqrt{\DD_0}}.
\end{aligned}
}
Notice that $\BB_0$ and $\BB_1$ are both zero in the limit as $\hbar \to 0$.
An examination of \eqref{210525213123} reveals that $s_\alpha^\pto{2} = - \varepsilon_\alpha \sqrt{\DD_0} b_0$.
Let $\beta_1 \coleq \Borel_\theta \big[ \, \BB_1 \, \big], \beta_0 \coleq \Borel_\theta \big[ \, \BB_0 \, \big]$.

Finally, we introduce two integral operators $\II_\pm$ acting on holomorphic functions $\beta = \beta (x, \xi)$ by the following formula:
\eqntag{\label{210527130326}
	\II_\pm \big[ \, \beta \, \big] (x, \xi)
	\coleq - \int_0^\xi \beta \big( x_t^\pm, \xi - t \big) \dd{t}
\qqtext{where}
	x_t^\pm \coleq \Phi^{-1} \big( \Phi (x) \pm t \big)
\fullstop{,}
}
where the integration path is the straight line segment from $0$ to $\xi$.
Let us fix a compactly contained subset $V \subset U$, and let $V_{\theta, \alpha}$ be the union of all WKB $(\theta, \alpha)$-rays emanating from $V$ (see \autoref{210618114017}).
Then expression \eqref{210527130326} is well-defined for all $x \in V_{\theta, \alpha}$ and all $\xi \in \Complex_\xi$ provided that $x_\xi^\alpha \in V_{\theta, \alpha}$.

Heuristically, this formula should be thought of as integrating along nearby WKB curves.
Indeed, for values of $\xi$ with phase exactly $\theta$, the path $\set{x_t^\alpha ~|~ t \in [0, \xi]}$ is nothing but a segment of the WKB $(\theta, \alpha)$-ray emanating from $x_0$.
Thus, restricting $\xi$ to the ray $e^{i\theta}\Real_+ \subset \Complex_\xi$, expression \eqref{210527130326} is well-defined for all $(x, \xi) \in V_{\theta, \alpha} \times e^{i\theta}\Real_+$.
In fact, since $V$ is compactly contained in $U$, there is some tubular neighbourhood $\Xi_\theta$ of the ray $e^{i\theta}\Real_+ \subset \Complex_\xi$ such that \eqref{210527130326} is well-defined for all $(x, \xi) \in V_{\theta,\alpha} \times \Xi_\theta$.
The main result of this subsection is then the following proposition.

%===============================================================================
\begin{prop}{210527135451}
The exponentiated analytic Borel transform $\Borel^{\exp}_\theta \big[ \psi_\alpha \big]$, which equals the analytic continuation along the ray $e^{i \theta} \Real_+$ of the exponentiated formal Borel transform $\hat{\Borel}^{\exp} \big[ \hat{\psi}_\alpha \big]$, can be expressed for all $(x, \xi) \in V \times \Xi'_\theta$ as follows:
\eqntag{\label{210618144426}
	\Borel^{\exp}_\theta \big[\, \psi_\alpha \,\big] (x, \hbar)
		= \exp \left( - \frac{1}{\hbar} \int\nolimits_{x_0}^x \lambda_\alpha (t) \dd{t} \right)
		\exp \left( 
			- \varepsilon_\alpha \int\nolimits_{x_0}^x 
				\sqrt{\DD_0 (t)} \, \tau_\alpha (t, \xi) \dd{t} 
			\right)
\fullstop{,}
}
where $\tau_\alpha$ is a holomorphic function on $V \times \Xi'_\theta$ defined as the following uniformly convergent infinite series:
\eqntag{
	\tau_\alpha (x, \xi) = \sum_{n=0}^\infty \tau_{\alpha,n} (x, \xi)
\fullstop
}
The terms $\tau_{\alpha,n}$ are holomorphic functions given by the following recursive formula: $\tau_{\alpha,0} \coleq b_0, \tau_{\alpha,1} \coleq \II_\alpha \big[ \beta_0 + b_1 b_0 \big]$, and, for $n \geq 2$,
\eqntag{\label{210527144120}
	\tau_{\alpha,n}
		\coleq \II_\alpha \Bigg[ 
			b_1 \tau_{\alpha,n-1}
			+ \beta_1 \ast \tau_{\alpha,n-2}
			+ \sum_{\substack{n_1,n_2 \geq 0 \\ n_1 + n_2 = n-2}}
				\tau_{\alpha,n_1} \ast \tau_{\alpha,n_2}
			\Bigg]
\fullstop
}
\end{prop}

\begin{proof}
This is a consequence of the proof of \autoref{210421183535}.
Namely, the recursive formula \eqref{210527144120} is the formula \eqref{210522160020} but written in the coordinate $x$ instead of the coordinate $z$.
\end{proof}

%===============================================================================
\begin{example}[undeformed coefficients]{210527143514}
As ever, the above expressions are simplest when the coefficients $p,q$ of our differential equation are independent of $\hbar$.
In this case, formulas \eqref{210527121916} are considerably simplified:
\eqntag{
	\BB_0 = \BB_1 = 0,
\quad
	b_0 = \frac{(s_\alpha^\pto{1})^2 - \del_x s_\alpha^\pto{1}}{\DD_0},
\quad
	b_1 = - \frac{2s_\alpha^\pto{1} + \del_x \log \sqrt{\DD_0}}{\varepsilon_\alpha \sqrt{\DD_0}}
\fullstop
}
In particular, this means $\beta_0 = \beta_1 = 0$, and so the recursion \eqref{210527144120} reduces to
\eqntag{\label{210527145127}
	\tau_{\alpha,0} = b_0,
\quad
	\tau_{\alpha,1} = \II_\alpha \big[ b_1 b_0 \big],
\quad
	\tau_{\alpha,n}
		\coleq \II_\alpha \Bigg[ 
			b_1 \tau_{\alpha,n-1}
			+ \sum_{\substack{n_1,n_2 \geq 0 \\ n_1 + n_2 = n-2}}
				\tau_{\alpha,n_1} \ast \tau_{\alpha,n_2}
			\Bigg]
\fullstop
\qedhere
}
\end{example}

%===============================================================================
\begin{example}[Schrödinger equation]{210527144424}
Formulas \eqref{210527121916} are also considerably simplified for the Schrödinger equation:
\eqn{
\begin{gathered}
	\BB_0 = \frac{\QQ_\ast - \QQ_2}{4\QQ_0},
\quad
	\BB_1 = 0,
\quad
	b_0 = \frac{\QQ_2 + (s_\alpha^\pto{1})^2 - \del_x s_\alpha^\pto{1}}{4 \QQ_0},
\quad
	b_1 = - \frac{s_\alpha^\pto{1} + \tfrac{1}{2} \del_x \log \sqrt{\QQ_0}}{\varepsilon_\alpha \sqrt{\QQ_0}}
\fullstop
\end{gathered}
}
Thus, $\beta_1 = 0$ but $\beta_0$ is not necessarily $0$.
Recursion \eqref{210527144120} in this case has exactly the same form as in \eqref{210527145127} but with $\tau_{\alpha,1} = \II_\alpha [\beta_0 + b_1 b_0]$.

Combining this with the previous example, the formulas in the case of a Schrödinger equation with $\hbar$-independent potential (i.e., $\QQ = \QQ_0$) have the simplest possible form:
\eqn{
	\BB_0 = \BB_1 = 0,
\qquad
	b_0
		= \frac{5}{64} \frac{(\QQ')^2}{\QQ^3} - \frac{1}{16} \frac{\QQ''}{\QQ^2},
\qquad
	b_1 
		= - \varepsilon_\alpha \frac{1}{2} \frac{\QQ'}{\QQ^{3/2}}
\fullstop
}
In this case, $\beta_0 = \beta_1 = 0$ so the recursion \eqref{210527144120} is again given by \eqref{210527145127}.
\end{example}

\begin{rem}[\textbf{Resurgent nature of the exact WKB analysis}]{210622072950}
The resurgent property of WKB solutions for Schrödinger equations with polynomial potential was conjectured by Voros in \cite{MR728983,MR729194} and partially argued by Écalle in the preprint \cite[p.40]{EcalleCinqApplications} (see \cite[Comment on p.32]{MR1704654}).
We do not address this point directly in our paper.
However, we believe that formula \eqref{210618144426} for the Borel transform is sufficiently simple and explicit to keep track of the singularities in the Borel plane.
We therefore hope it can yield a full proof of the conjectured resurgence property of WKB solutions not only for Schrödinger equations with polynomial potential (as conjectured by Voros), but more generally for all second-order ODEs \eqref{210115121038} with rational dependence on $x$.
\end{rem}

%===============================================================================
%===============================================================================
%===============================================================================
\subsection{Notable Special Cases and Examples}
\label{210603163500}
%===============================================================================
%===============================================================================
%===============================================================================

In this subsection, we restate the existence and uniqueness results proved in this paper for two important classes of WKB geometry: closed and generic WKB trajectories.
In both of these cases, the technical assumptions in \autoref{210116200501} simplify considerably.
Throughout this subsection, we maintain our background assumptions of \autoref{210612115147}.

%===============================================================================
%===============================================================================
\subsubsection*{Closed WKB Trajectories}
\addcontentsline{toc}{subsubsection}{Closed WKB Trajectories}
%===============================================================================
%===============================================================================

The statement of \autoref{210116200501} is simplest for closed trajectories.
In this case, assumptions (1) and (2) are automatic because a closed trajectory can always be embedded in a WKB ring domain whose closure is a compact subset of $X$.

%===============================================================================
\begin{cor}[\bfseries Existence and uniqueness for closed WKB trajectories]{210519102448}
\leavevmode\newline
Let $\theta \in \Theta$ be fixed.
Suppose that the WKB $\theta$-trajectory passing through $x_0$ is closed.
Let $U \subset X$ be any simply connected neighbourhood of $x_0$ contained in a WKB $\theta$-ring domain $R \subset X$.
Then all the conclusions of \autoref{210116200501}, \autoref{210603110237}, \autoref{210524175026}, and \autoref{210221112347} hold verbatim simultaneously for both $\alpha = \pm$.
\end{cor}

%===============================================================================
\paragraph{Monodromy of exact WKB solutions on WKB ring domains.}
The exact WKB solutions $\psi_+, \psi_-$ from \autoref{210519102448} extend to the entire WKB ring domain $R$ but only as multivalued functions.
Thanks to the explicit formula in \autoref{210603110237}, their monodromy is easy to calculate.

%===============================================================================
\begin{prop}{210603165101}
The exact WKB solutions $\psi_+, \psi_-$ from \autoref{210519102448} extend via the formula \eqref{210218230927} to multivalued holomorphic solutions on $R \times S'_\theta$ with monodromy
\begin{equation}\label{210525155939}
	a_\pm (\hbar) 
		\coleq \exp \left( - \frac{1}{\hbar} \oint\nolimits_{\gamma_{\pm}} s_\pm (x, \hbar) \dd{x} \right)
\fullstop{,}
\end{equation}
where the integration contour $\gamma_{\pm}$ is any path contained in $R$ and homotopic to the closed WKB $\theta$-trajectory passing through $x_0$ and with orientation matching the orientation of the WKB $(\theta, \pm)$-ray.
The monodromy $a_\pm$ is a holomorphic function of $\hbar \in S'_\theta$ which admits exponential Gevrey asymptotics in a halfplane:
\eqntag{
	a_\pm (\hbar) \simeq \hat{a}_\pm (\hbar) \coleq \exp \left( - \frac{1}{\hbar} \oint\nolimits_{\gamma_{\pm}} \hat{s}_\pm (x, \hbar) \dd{x} \right)
\qquad
\text{as $\hbar \to 0$ along $\bar{A}_\theta$}
\fullstop
}
\end{prop}

%===============================================================================
\begin{cor}[\bfseries Existence and uniqueness in wider sectors]{210603165152}
More generally, suppose that for every $\theta \in \Theta$, the WKB $\theta$-trajectory passing through $x_0$ is closed.
Then $U$ can be chosen sufficiently small such that all the conclusions of \autoref{210527195440} hold verbatim, and the monodromy $a_\pm$ extends to a holomorphic function on $S'$ with exponential Gevrey asymptotics: $a_\pm (\hbar) \simeq \hat{a}_\pm (\hbar)$ as $\hbar \to 0$ along $\bar{A}$.
\end{cor}

%===============================================================================
\begin{rem}{210525160309}
We note that upon writing $s_\pm (x, \hbar) = \lambda_\pm (x) + \hbar \SS_\pm (x, \hbar)$, the monodromy \eqref{210525155939} is expressed as
\eqntag{\label{210525160002}
	a_\pm (\hbar)
		= \exp \left( - \frac{1}{\hbar} \oint\nolimits_{\gamma_\pm} \lambda_\pm (x) \dd{x} \right)
		  \exp \left( - \oint\nolimits_{\gamma_\pm} \SS_\pm (x,\hbar) \dd{x} \right)
\fullstop
}
This expression is notable because the complex number $\oint\nolimits_{\gamma_\pm} \lambda_\pm (x) \dd{x}$ is a \textit{period} of a certain covering Riemann surface (called \textit{spectral curve}) naturally associated with our differential equation (namely, the one given by the leading-order characteristic equation \eqref{210415145506}).
These numbers, and therefore functions \eqref{210525160002}, play pivotal role in the global analysis of such differential equations and more general meromorphic connections on Riemann surfaces.
These topics are beyond the scope of this paper, but see for example \cite{MR3003931}.
More comments will appear in \cite{MY210517181728}.
\end{rem}

%===============================================================================
%===============================================================================
\subsubsection*{Generic WKB Trajectories}
\addcontentsline{toc}{subsubsection}{Generic WKB Trajectories}
%===============================================================================
%===============================================================================

For generic WKB rays, condition (1) in \autoref{210116200501} is automatically taken care of by insisting (in the definition of generic rays) that the limiting point $x_\infty$ is an infinite critical point; i.e., a pole of $\DD_0$ of order $m \geq 2$.
At the same time, condition (2) must still be imposed but it is somewhat simplified by the fact that $\DD_0$ behaves like $(x-x_\infty)^{-m}$ near $x_\infty$.
Altogether, we have the following statement.

%===============================================================================
\begin{cor}[\bfseries Existence and uniqueness for generic WKB rays]{210519103256}
\leavevmode\newline
Fix a sign $\alpha \in \set{+, -}$ and a phase $\theta \in \Theta$.
Suppose that the WKB $(\theta,\alpha)$-ray emanating from $x_0$ is generic.
Let $x_\infty \in \Complex_x$ be the limiting infinite critical point of order $m \geq 2$.
In addition, we make the following assumption on the coefficients $p,q$:
\eqntag{
\label{210514194406}
\begin{gathered}
	(x-x_\infty)^{m/2} p (x,\hbar) \simeq (x-x_\infty)^{m/2} \hat{p} (x,\hbar)
\fullstop{,}
\\
	(x-x_\infty)^m q (x,\hbar) \simeq (x-x_\infty)^m \hat{q} (x,\hbar)
\end{gathered}
}
as $\hbar \to 0$ along $\bar{A}$, uniformly for all $x \in X$ sufficiently close to $x_\infty$.
Then $x_0$ has a neighbourhood $U \subset X$ such that all the conclusions of \autoref{210116200501}, \autoref{210603110237}, \autoref{210524175026}, and \autoref{210221112347} hold verbatim.
\end{cor}

Specifically, $U \subset X$ can be chosen to be any simply connected domain containing $x_0$ which is free of turning points and such that every WKB $(\theta,\alpha)$-ray emanating from $U$ is generic and tends to $x_\infty$.
Note that the assumption in \autoref{210519103256} that the WKB $(\theta,\alpha)$-ray emanating from $x_0$ is generic guarantees that such a neighbourhood $U$ always exists, see \autoref{210524105411}.

%===============================================================================
\begin{example}[Mildly perturbed coefficients]{210612114025}
If the coefficients $p,q$ of the differential equation \eqref{210115121038} are at most polynomial in $\hbar$ (i.e., if $p,q \in \cal{O} (X) [\hbar]$) then condition \eqref{210514194406} simplifies even further as follows: for every $k \geq 0$,
\eqntag{\label{210612114403}
	\textup{(pole order of $p_k$ at $x_\infty$)} \leq \tfrac{1}{2} m
\qtext{and}
	\textup{(pole order of $q_k$ at $x_\infty$)} \leq m
\fullstop
}
\end{example}

\enlargethispage{10pt}
%===============================================================================
\begin{cor}[\bfseries Existence and uniqueness in wider sectors]{210603165446}
More generally, suppose that for every $\theta \in \Theta$, the WKB $\theta$-trajectory passing through $x_0$ is generic, and that condition \eqref{210514194406} is satisfied at the (necessarily $\theta$-independent) limiting infinite critical point $x_\infty$.
Then $U$ can be chosen sufficiently small such that all the conclusions of \autoref{210527195440} hold verbatim.
\end{cor}

%===============================================================================
\paragraph{The exact WKB basis on a WKB strip domain.}
If the WKB trajectory through $x_0$ is generic, then \autoref{210519103256} yields two exact WKB solutions, one for each WKB ray emanating from $x_0$.
These exact WKB solutions define a basis of exact solutions on any WKB strip domain containing the WKB trajectory through $x_0$.
To be precise, we have the following statement.

%===============================================================================
\begin{cor}{210525175018}
Suppose the WKB $\theta$-trajectory passing through $x_0$ is generic, and let $U \subset X$ be any WKB $\theta$-strip containing $x_0$.
Let $x_{\pm\infty} \in \Complex_x$ be the two limiting infinite critical points of order $m_\pm \geq 2$.
In addition, we make the following assumption on the coefficients $p,q$:
\eqntag{
\label{210525175435}
\begin{gathered}
	(x-x_{\pm\infty})^{m_\pm/2} p (x,\hbar) \simeq (x-x_{\pm\infty})^{m_\pm/2} \hat{p} (x,\hbar)
\fullstop{,}
\\
	(x-x_{\pm\infty})^{m_\pm} q (x,\hbar) \simeq (x-x_{\pm\infty})^{m_\pm} \hat{q} (x,\hbar)
\end{gathered}
}
as $\hbar \to 0$ along $\bar{A}_\theta$, uniformly for all $x \in X$ sufficiently close to $x_{\pm\infty}$.
Then all the conclusions of \autoref{210116200501}, \autoref{210116200501}, \autoref{210603110237}, \autoref{210524175026}, and \autoref{210221112347} hold verbatim for both $\alpha = \pm$.
\end{cor}

%===============================================================================
%===============================================================================
\subsection{Relation to Previous Work}
\label{210721173245}
%===============================================================================
%===============================================================================

In this final subsection, we explain how our results relate to other works about the existence of WKB solutions.

%===============================================================================
\begin{rem}[\textbf{Relation to the work of Nemes}]{210614164551}
In the recent paper \cite{MR4226390}, Nemes considers Schrödinger equations of the form\footnote{Explicitly, the notations compare as follows: equation (1.3) in \cite{MR4226390} is our equation \eqref{210614204252} with $u \leftrightarrow \hbar^{-1}$, $\xi \leftrightarrow z$, $\varphi \leftrightarrow a_1$, $\psi \leftrightarrow a_2$, $\WW (u, \xi) \leftrightarrow \WW (z, \hbar)$, and $\mathrm{\mathbf{G}} \leftrightarrow \Omega$.}
\eqntag{\label{210614204252}
	\hbar^2 \del^2_z \WW  = \big( 1 + a_1 (z) \hbar + a_2 (z) \hbar^2 \big) \WW
\fullstop{,}
}
where $a_1, a_2$ are holomorphic functions on a domain $\Omega \subset \Complex_z$ which contains an infinite horizontal strip.
Using a Banach fixed-point theorem argument, he shows (see \cite[Theorem 1.1]{MR4226390}) that under certain boundedness assumptions on the coefficients $a_1, a_2$ (see \cite[Conditions 1.1 and 1.2]{MR4226390}), the Schrödinger equation \eqref{210614204252} has (in our terminology) two exact solutions $\WW^\pm = \WW^\pm (x, \hbar)$, defined for all $(z, \hbar) \in \Omega^\pm_0 \times S'$ where $S'$ is a Borel disc with opening $A \coleq (-\pi/2, +\pi/2)$ and $\Omega_0^\pm \Subset \Omega$ are any properly contained horizontal halfstrips (unbounded respectively as $z \to \pm \infty$).
The following proposition asserts that this existence result is a corollary of our main theorem.
\end{rem}

\enlargethispage{10pt}
\begin{prop}{210614232200}
\autoref{210116200501} (or, more specifically, \autoref{210614221544}) implies Theorem 1.1 and the first assertion of Theorem 2.1 in \cite{MR4226390}.
\end{prop}

\begin{proof}
For the equation \eqref{210614204252}, the leading-order characteristic discriminant $\DD_0 = 1$, so condition (1) of \eqref{210116200501} is vacuously true.
The boundedness Conditions 1.1 and 1.2 in \cite{MR4226390} imply in particular that the coefficients $a_1, a_2$ are bounded on $\Omega^\pm_0$, which by the discussion in \autoref{210614221544} means condition (2) of \autoref{210116200501} is met.
So Theorem 1.1 in \cite{MR4226390} follows.
Finally, the first assertion of Theorem 2.1 in \cite{MR4226390} is a special case of \autoref{210603145334}.
\end{proof}

%===============================================================================
\paragraph{}
Equations of the form \eqref{210614204252} can be related by means of a Liouville transformation $z = \Phi (x)$ from \eqref{210119093734} to Schrödinger equations of the form \eqref{210115121042} with potentials that are at most quadratic in $\hbar$; i.e., $\QQ (x,\hbar) = \QQ_0 (x) + \hbar \QQ_1 (x) + \hbar^2 \QQ_2 (x)$.
%\eqntag{
%	\hbar^2 \del^2_x \psi = \big( \QQ_0 (x) + \hbar \QQ_1 (x) + \hbar^2 \QQ_2 (x) \big) \psi
%}
Explicitly, the unknown variables $\psi$ and $\WW$ are related by
%$\psi (x, \hbar) = \QQ_0^{-1/4} (x) \WW \big( \tfrac{1}{2} \Phi (x), \hbar \big)$,
\eqn{
	\psi (x, \hbar) = \QQ_0^{-1/4} (x) \WW \big( \tfrac{1}{2} \Phi (x), \hbar \big)
\fullstop{,}
}
and the coefficients are related by 
%$a_1 \big( \tfrac{1}{2}\Phi(x) \big) \coleq \QQ^{-1}_0(x) \QQ_1 (x)$ and $a_2 \big( \tfrac{1}{2}\Phi (x) \big) \coleq \QQ^{-1}_0(x) \QQ_2 (x) - \QQ^{-3/4}_0 (x) \del_x^2 \left( \QQ_0^{-1/4} (x) \right)$.
\eqn{
	a_1 \big( \tfrac{1}{2}\Phi(x) \big) \coleq \frac{ \QQ_1 (x) }{\QQ_0 (x) }
\qtext{and}
	a_2 \big( \tfrac{1}{2}\Phi (x) \big) \coleq \frac{\QQ_2 (x)}{\QQ_0 (x)} 
		- \frac{1}{\QQ^{3/4}_0 (x)} \del_x^2 \left( \frac{1}{\QQ_0^{1/4} (x)} \right)
\fullstop
}
However, note that this transformation of the unknown variable involves a choice of a fourth-root branch $\QQ_0^{1/4}$ and, more importantly, even if $\WW = \WW (z, \hbar)$ is a solution of \eqref{210614204252} for $z$ in some domain $\Omega^\pm_0 \subset \Omega$, then $\psi (x, \hbar)$ is a well-defined solution only if $\WW (z_1, \hbar) = \WW (z_2, \hbar)$ whenever $\Phi^{-1} (z_1) = \Phi^{-1} (z_2)$.
The explicit approach pursued in our paper (aided specifically by the recursive formula of \autoref{210527135451}) makes this verification obvious.

%===============================================================================
\begin{rem}[\textbf{Relation to the work of Koike-Schäfke}]{210612121553}
Some of the results in a number of references mentioned in the introduction (see paragraph 5 of \autoref{210614233439}) rely on the statement of Theorem 2.17 presented in \cite{MR3280000} from an unpublished work of Koike and Schäfke on the Borel summability of formal WKB solutions of Schrödinger equations with polynomial potential.
The following proposition asserts that our results imply part (a) of Theorem 2.17 in \cite{MR3280000}.
Part (b) of Theorem 2.17 in \cite{MR3280000} will be derived from a more general result in \cite{MY210604104440}.
It is also stated in Theorem 2.18 in \cite{MR3280000} that Theorem 2.17 in \cite{MR3280000} holds for any compact Riemann surface: this theorem will also be derived as a special case of a more general result in \cite{MY210517181728}.
\end{rem}

\begin{prop}{210612144425}
\autoref{210519103256} implies part (a) of Theorem 2.17 in \cite{MR3280000}. 
\end{prop}

\begin{proof}
The main assumption for Theorem 2.17 in \cite{MR3280000} is that the potential $\QQ$ is a polynomial in $\hbar$ with rational coefficients $\QQ_k$ whose behaviour at the poles is as stated in Assumption 2.5 and the third bullet point of Assumption 2.3 in \cite{MR3280000}.
We claim that these assumptions are a special case of \eqref{210612114403} in \autoref{210612114025}

First, let us explain how the notation in \cite{MR3280000} compares with ours.
In \cite{MR3280000}, the equation variable $z$ is the same as our variable $x$, and the large parameter $\eta$ is our $\hbar^{-1}$.
The authors consider Schrödinger equations of the form \eqref{210115121042} but where $\QQ (x, \hbar) = \QQ_0 (x) + \QQ_1 (x) \hbar + \cdots$ is a polynomial in $\hbar$ (cf. \cite[equation (2.2)]{MR3280000}).
This is the situation in \autoref{210612114025} with $p_k = 0$ for all $k$.
In the statement of Theorem 2.17 (a) in \cite{MR3280000}, the chosen point $z \in \DD$ in a Stokes region $\DD$ ($=$ a maximal WKB strip domain) is our regular basepoint $x_0 \in X$.

\enlargethispage{10pt}
Let $x_\infty$ be the pole in question either on the boundary of $X$ or at infinity in $\Complex_x$.
By the assumptions in \autoref{210519103256}, $x_\infty$ is an infinite critical point, which means in particular that the pole order of $\QQ_0$ at $x_\infty$ is $m \geq 2$, which coincides with the third bullet point of Assumption 2.3 in \cite{MR3280000}.
Parts (i) and (iii) of Assumption 2.5 in \cite{MR3280000} are also clearly included in \eqref{210612114403}.
Finally, part (ii) of Assumption 2.5 in \cite{MR3280000} is included in \eqref{210612114403} because $1 + \tfrac{1}{2} m < m$ whenever $m \geq 3$.
\end{proof}

%===============================================================================
%:	Appendix
%===============================================================================
%\newpage
\begin{appendices}
\appendixsectionformat
%===============================================================================
%===============================================================================
\section{Basics of Usual Asymptotics}
\setcounter{section}{1}
\setcounter{paragraph}{0}
\label{210217113936}
%===============================================================================
%===============================================================================

%===============================================================================
\paragraph{Sectorial domains.}
\label{210217114252}
%===============================================================================
Fix a circle $\Sphere^1 \coleq \Real / 2\pi \Integer$ once and for all.
We refer to its points as \dfn{directions}, and we think of it as the set of directions at the origin in $\Complex_\hbar$ when $\hbar$ is written in polar coordinates.
More precisely, we consider the \textit{real-oriented blowup} of the complex plane $\Complex_\hbar$ at the origin, which by definition is the bordered Riemann surface $[\Complex_\hbar : 0] \coleq \Real_+ \times \Sphere^1$ with coordinates $(r, \theta)$, where $\Real_+$ is the nonnegative reals.
The projection $[\Complex_\hbar : 0] \to \Complex_\hbar$ sends $(r, \theta) \mapsto r e^{i\theta}$, which is a biholomorphism away from the circle of directions.
See \autoref{210618080005} for an illustration.
%==============================
\begin{figure}[t]
\centering
\includegraphics{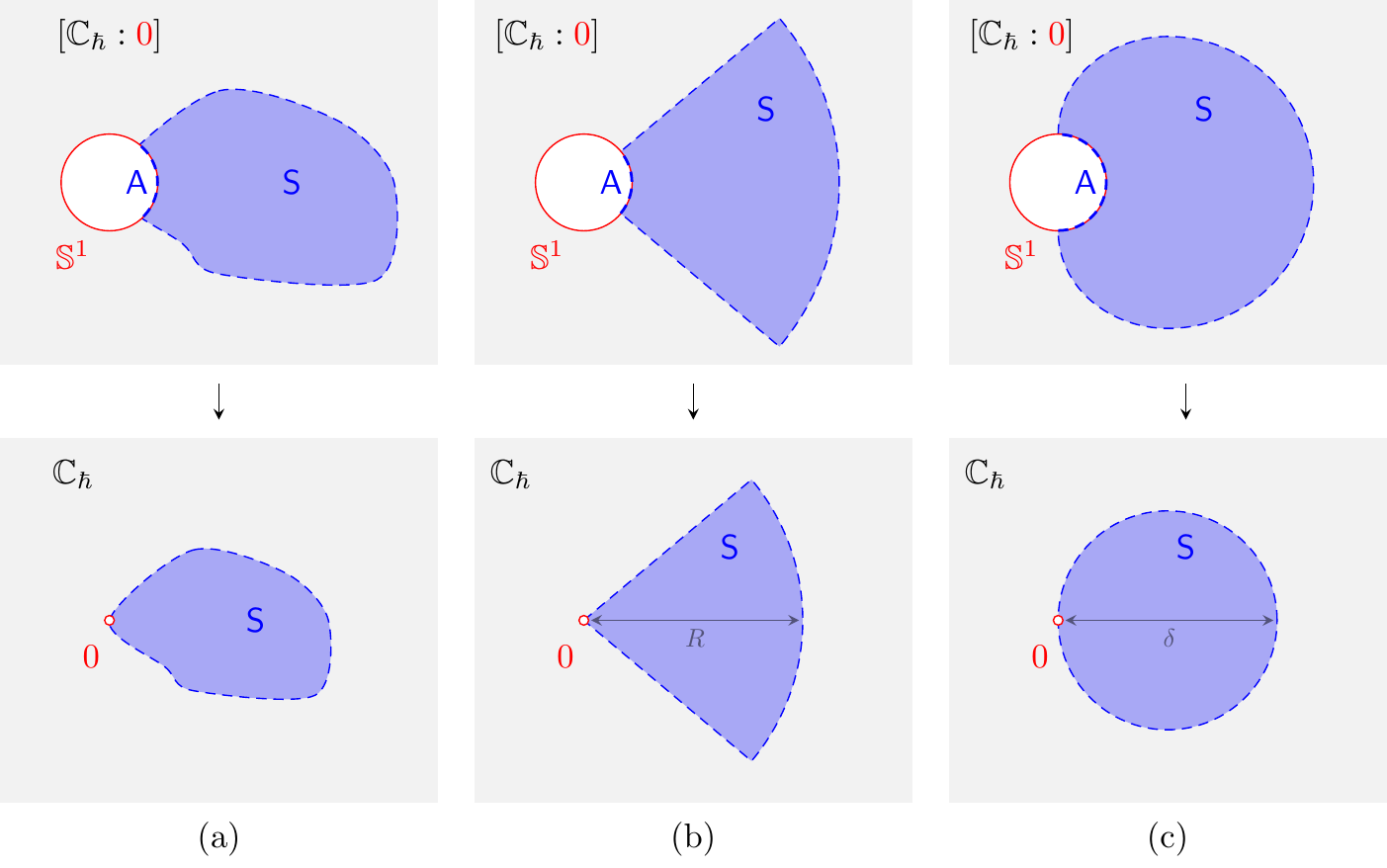}
\caption{Examples of sectorial domains.
(a): an arbitrary sectorial domain with opening $A$.
(b): straight sector with opening $A$ and some radius $\RR > 0$.
(c): a Borel disc with diameter $\delta > 0$.}
\label{210618080005}
\end{figure}
%==============================

A \dfn{sectorial domain} near the origin in $\Complex_\hbar$ is a simply connected domain $S \subset \Complex_\hbar^\ast = \Complex_\hbar \setminus \set{0}$ whose closure $\bar{S}$ in $[\Complex_\hbar : 0]$ intersects the boundary circle $\Sphere^1$ in a closed arc $\bar{A} \subset \Sphere^1$ with nonzero length.
In this case, the open arc $A$ is called the \dfn{opening} of $S$, and its length $|A|$ is called the \dfn{opening angle} of $S$.
A \dfn{proper subsectorial domain} $S_0 \subset S$ is one whose closure $\bar{S}_0$ in $[\Complex_\hbar : 0]$ is contained in $S$.
This means in particular that the opening $A_0$ of $S_0$ is compactly contained in $A$; i.e., $\bar{A}_0 \subset A$.

\enlargethispage{20pt}
The simplest example of a sectorial domain $S$ is of course a \textit{straight sector} of radius $\delta > 0$ and opening $A$, which is the set of points $\hbar \in \Complex_\hbar$ satisfying $\arg (\hbar) \in A$ and $0 < |\hbar| < \RR $.
The most typical example of a sectorial domain encountered in this paper is a \dfn{Borel disc} of \dfn{diameter} $\delta >0$:
\eqntag{
	S = \set{ \hbar \in \Complex_\hbar ~\big|~ \Re (1/\hbar) > 1/\delta }
\fullstop
}
Its opening is $A = (-\pi/2, +\pi/2)$.
Notice that any straight sector with opening $A$ contains a Borel disc, but a Borel disc contains no straight sectors with opening $A$.
It is also not difficult to see that any sectorial domain with opening $A$ contains a Borel disc.
More generally, we will consider Borel discs bisected by some direction $\theta \in \Sphere^1$:
\eqntag{
	S_\theta = \set{ \hbar \in \Complex_\hbar ~\big|~ \Re (e^{i\theta}/\hbar) > 1/\delta }
\fullstop
}

%===============================================================================
%===============================================================================
\subsection{Poincaré Asymptotics in One Dimension}
\label{210220155700}
%===============================================================================
%===============================================================================

First, for the benefit of the reader and to fix some notation, let us briefly recall some basic notions from asymptotic analysis in one complex variable.
We denote by $\Complex \bbrac{\hbar}$ the ring of formal power series in $\hbar$, and by $\Complex \set{\hbar}$ the ring of convergent power series in $\hbar$.
Fix an arc of directions $A \subset \Sphere^1$.
We denote by $\cal{O} (S)$ the ring of holomorphic functions on a sectorial domain $S \subset \Complex_\hbar$ with opening $A$.

%===============================================================================
\paragraph{Sectorial germs.}
\label{210225113610}
%===============================================================================
For the purpose of asymptotic behaviour as $\hbar \to 0$, the actual nonzero radial size of $S$ is irrelevant.
So it is better to consider \textit{germs} of holomorphic functions defined on sectorial domains with opening $A$, formally defined next.

\begin{defn}{210616134514}
A \dfn{sectorial germ} on $A$ is an equivalence class of pairs $(f,S)$ where $S \subset \Complex_\hbar$ is a sectorial domain with opening $A$ and $f$ is a holomorphic function on $S$.
Two such pairs $(f,S)$ and $(f',S')$ are considered equivalent if the intersection ${S \cap S'}$ contains a sectorial domain $S''$ with opening $A$ on which $f$ and $f'$ are equal.
\end{defn}

Sectorial germs on $A$ form a ring which we denote by $\cal{O} (A)$.
For any sectorial domain $S \subset \Complex_\hbar$ with opening $A$, there is a map $\cal{O} (S) \to \cal{O} (A)$ that sends a holomorphic function $f$ to the corresponding sectorial germ.

\emph{Terminology:} For the benefit of the reader who is uncomfortable with the language of germs, we stress that every sectorial germ on $A$ can be represented by an actual holomorphic function $f \in \cal{O} (S)$ on some sectorial domain $S$.
In fact, we normally denote the equivalence class of any pair $(f, S)$ simply by ``$f$'' and we often even refer to it as a \textit{holomorphic function on $A$}.

%===============================================================================
\paragraph{Poincaré asymptotics.}
\label{210225114212}
%===============================================================================
Recall that a holomorphic function $f \in \cal{O} (S)$ defined on a sectorial domain $S$ is said to admit (\dfn{Poincaré}) \dfn{asymptotics as $\hbar \to 0$ along $A$} if there is a formal power series $\hat{f} (\hbar) \in \Complex \bbrac{\hbar}$ such that the order-$n$ remainder
\eqntag{\label{200720152534}
	\RR_n (\hbar) \coleq f(\hbar) - \sum_{k=0}^{n-1} f_k \hbar^k
}
is bounded by $\hbar^n$ for all sufficiently small $\hbar \in S$.
That is, for every ${n \geq 0}$, and every compactly contained subarc $A_0 \Subset A$, there is a sectorial subdomain $S_0 \subset S$ with opening $A_0$ and a real constant $\CC_{n,0} > 0$ such that
\eqntag{\label{200720153758}
	\big| \RR_n (\hbar) \big| \leq \CC_{n,0} |\hbar|^n
}
for all $\hbar \in S_0$.
The constants $\CC_{n,0}$ may depend on $n$ and the opening $A_0$.
If this is the case, we write
\eqntag{\label{200720175735}
	f (\hbar) \sim \hat{f} (\hbar)
\qqqquad
	\text{as $\hbar \to 0$ along $A$\fullstop}
}
Sectorial germs with this property form a subring $\cal{A} (A) \subset \cal{O}(A)$, and the asymptotic expansion map defines a ring homomorphism $\ae : \cal{A}(A) \to \Complex \bbrac{\hbar}$.

%===============================================================================
\paragraph{Asymptotics along a closed arc.}
\label{210225134124}
%===============================================================================
Furthermore, we will write
\eqntag{\label{210220160756}
	f (\hbar) \sim \hat{f} (\hbar)
\qqqquad
	\text{as $\hbar \to 0$ along $\bar{A}$\fullstop{,}}
}
if the constants $\CC_{n,0}$ in \eqref{200720153758} can be chosen uniformly for all compactly contained subarcs $A_0 \Subset A$ (i.e., independent of $A_0$ so that $\CC_{n,0} = \CC_n$ for all $n$).
Obviously, if a function $f$ admits asymptotics along $A$, then it admits asymptotics along $\bar{A}_0$ for any $A_0 \Subset A$.
Sectorial germs with this property form a subring $\cal{A} (\bar{A}) \subset \cal{A}(A)$.

For example, the function $e^{-1/\hbar}$ admits asymptotics as $\hbar \to 0$ along the open arc $A = (-\pi/2, +\pi/2)$ (where it is asymptotic to $0$), but \textit{not} along the closed arc $\bar{A} = [-\pi/2, +\pi/2]$, because the constants $\CC_{n,0}$ in the asymptotic estimates \eqref{200720153758} blow up as $A_0$ approaches $A$.
Thus, $e^{-1/\hbar}$ is an element of $\cal{A} (A)$ but not of $\cal{A} (\bar{A})$.

%===============================================================================
%===============================================================================
\subsection{Uniform Poincaré Asymptotics}
\label{210225141606}
%===============================================================================
%===============================================================================

Now, suppose we also have a domain $U \subset \Complex_x$.

%===============================================================================
\paragraph{Power series with holomorphic coefficients.}
\label{200721171954}
%===============================================================================
We denote by $\cal{O} (U) \bbrac{\hbar}$ the set of formal power series in $\hbar$ with holomorphic coefficients on $U$:
\eqntag{\label{200624150851}
	\hat{f} (x, \hbar) = \sum_{k=0}^\infty f_k (x) \hbar^k 
	\quad \in \quad
	\cal{O} (U) \bbrac{\hbar}
\fullstop
}
Let us also introduce the subsets $\cal{O}_\rm{u} (U) \set{\hbar}$ and $\cal{O} (U) \set{\hbar}$ of $\cal{O} (U) \bbrac{\hbar}$ consisting of, respectively, uniformly and locally-uniformly convergent power series on $U$.
We will not have any use for pointwise convergence (or other pointwise regularity statements), so we do not introduce any special notation for those.

%===============================================================================
\paragraph{Semisectorial germs.}
\label{210226112200}
%===============================================================================
Since we are only interested in keep track of the asymptotic behaviour as $\hbar \to 0$ along $A$, we focus our attention on holomorphic functions of $(x, \hbar)$ which behave like sectorial germs in $\hbar$.
More precisely, we introduce the following definition.

%We are primarily concerned with the asymptotics as $\hbar \to 0$ along $A$ of holomorphic functions defined on product domains $X \times S$ where $S$ is a sectorial domain with opening $A$.
%Whereas the radial size of $S$ is unimportant for this purpose, the domain $U$ is kept fixed.
%More generally, we need to consider holomorphic functions defined on domains which are product domain only locally in $x$.

%===============================================================================
\begin{defn}{210226161539}
A \dfn{semisectorial germ} on $(U; A)$ is an equivalence class of pairs $(f,\mathbb{U})$ where $f$ is a holomorphic function on the domain $\mathbb{U} \subset U \times \Complex_\hbar^\ast$ with the following property: for every point $x_0 \in U$, there is a neighbourhood $U_0 \subset U$ of $x_0$ and a sectorial domain $S_0$ with opening $A$ such that $U_0 \times S_0 \subset \mathbb{U}$.
Any two such pairs $(f,\mathbb{U})$ and $(f',\mathbb{U}')$ are considered equivalent if, for every $x_0 \in U$, the intersection $S_0 \cap S'_0$ contains a sectorial domain $S''_0$ with opening $A$ such that the restrictions of $f$ and $f'$ to $U_0 \times S''_0$ are equal.
\end{defn}

\emph{Terminology:} We will often abuse terminology and refer to semisectorial germs $f \in \cal{O} (U;A)$ as \textit{holomorphic functions on $(U;A)$}.

Typically, $\mathbb{U}$ is a product domain $U \times S$ for some $S$ or a (possibly countable) union of product domains.
Notice that in particular the projection of $\mathbb{U}$ onto the first component is necessarily $U$.
Semisectorial germs form a ring which we denote by $\cal{O} (U; A)$.
For any domain $\UUU \subset \Complex_{x\hbar}^2$ as above, there is a map $\cal{O} (\UUU) \to \cal{O} (U;A)$ sending a holomorphic function $f$ on $\UUU$ to the corresponding semisectorial germ; i.e., the equivalence class of $f$ in $\cal{O} (U;A)$.
There is also a map $\cal{O} (U) \set{\hbar} \to \cal{O} (U) (A)$ for any $A$ given by restriction.

%===============================================================================
\paragraph{Uniform Poincaré asymptotics.}
\label{200721171635}
%===============================================================================
Consider the product space $\UUU = U \times S$ where $S$ is a sectorial domain with opening $A$, or more generally let $\UUU$ be a domain of the form described in \autoref{210226161539}.
Recall that a holomorphic function $f \in \cal{O} (\UUU)$ is said to admit (pointwise Poincaré) \dfn{asymptotics as $\hbar \to 0$ along $A$} if there is a formal power series $\hat{f} (x, \hbar) \in \cal{O} (U) \bbrac{\hbar}$ such that for every ${n \geq 0}$, every $x \in U$, and every compactly contained subarc $A_0 \Subset A$, there is a sectorial domain $S_0$ with opening $A_0$ and a real constant $\CC_{n,x,0} > 0$ which satisfies the following inequality:
\eqntag{\label{200305151735}
	\Big| \RR_n (x, \hbar) \Big|
		= \left| f(x, \hbar) - \sum_{k=0}^{n-1} f_k (x) \hbar^k \right|
		\leq \CC_{n,x,0} |\hbar|^n
}
for all $\hbar \in S_0$.
The constant $\CC_{n,x,0}$ may depend on $n,x$, and the opening $A_0$.
We say that $f$ admits \dfn{uniform asymptotics} \textit{on $U$ as $\hbar \to 0$ along $A$} if $\CC_{n,x,0}$ can be chosen to be independent of $x$ (i.e., so that $\CC_{n,x,0} = \CC_{n,0}$).
We also say that $f$ admits \dfn{locally uniform asymptotics} \textit{on $U$ as $\hbar \to 0$ along $A$} if every point in $U$ has a neighbourhood on which $f$ admits uniform asymptotics.
In these cases, we write, respectively,
\eqnstag{\label{210225121434}
	f (x,\hbar) &\sim \hat{f} (x,\hbar)
\qqqquad
	\text{as $\hbar \to 0$ along $A$, unif. $\forall x \in U$\fullstop{;}}
\\
\label{210225121453}
	f (x,\hbar) &\sim \hat{f} (x,\hbar)
\qqqquad
	\text{as $\hbar \to 0$ along $A$, loc.unif. $\forall x \in U$\fullstop}
}
We denote the subrings consisting of semisectorial germs satisfying \eqref{210225121434} or \eqref{210225121453} respectively by $\cal{A}_\rm{u} (U; A)$ and $\cal{A} (U; A)$.
The asymptotic expansion map defines a ring homomorphism $\ae : \cal{A} (U;A) \to \cal{O} (U) \bbrac{\hbar}$.
An elementary application of the Cauchy integral formula shows that the ring $\cal{A} (U; A)$ (but not $\cal{A}_\rm{u} (U; A)$) is preserved under differentiation with respect to $x$; that is, $\del_x \big( \cal{A} (U; A) \big) \subset \cal{A} (U; A)$.

%===============================================================================
\paragraph{Uniform Poincaré asymptotics along a closed arc.}
\label{210225142705}
%===============================================================================
If in addition to \eqref{210225121434} or \eqref{210225121453}, the constants $\CC_{n,x,0}$ can be chosen uniformly for all $A_0 \Subset A$ (i.e., so that $\CC_{n,x,0} = \CC_{n,x}$ for all $n$ and $x$), we will write, respectively,
\eqnstag{\label{210225123426}
	f (x,\hbar) &\sim \hat{f} (x,\hbar)
\qqqquad
	\text{as $\hbar \to 0$ along $\bar{A}$, unif. $\forall x \in U$\fullstop{;}}
\\
\label{210225123428}
	f (x,\hbar) &\sim \hat{f} (x,\hbar)
\qqqquad
	\text{as $\hbar \to 0$ along $\bar{A}$, loc.unif. $\forall x \in U$\fullstop}
}
Holomorphic functions satisfying these conditions form subrings which we denote respectively by $\cal{A}_\rm{u} (U; \bar{A}) \subset \cal{A}_\rm{u} (U; A) $ and $\cal{A} (U; \bar{A}) \subset \cal{A} (U; A)$.
Again, the ring $\cal{A} (U; \bar{A})$ (but not the ring $\cal{A}_\rm{u} (U; \bar{A})$) is preserved by differentiation: $\del_x \big( \cal{A} (U; \bar{A}) \big) \subset \cal{A} (U; \bar{A})$.

%===============================================================================
%===============================================================================
\subsection{Gevrey Asymptotics}
\label{210224181219}
%===============================================================================
%===============================================================================

\enlargethispage{15pt}
For the purposes of the main construction in this paper, the notion of Poincaré asymptotics is too weak.
A powerful and systematic way to refine Poincaré asymptotics is known as \textit{Gevrey asymptotics}.
The basic principle behind it is to strengthen the asymptotic requirements by specifying the dependence on $n$ of the constants $\CC_{n,0}$ in \eqref{200720153758} and $\CC_{n,x,0}$ in \eqref{200305151735}.
In this paper, we use only the simplest Gevrey regularity class (more properly known as \textit{1-Gevrey asymptotics}) which requires the asymptotic bounds to grow essentially like $n!$.
See, for example, \cite[\S1.2]{MR3495546} for a more general introduction to Gevrey asymptotics.
As before, for the benefit of the reader we first recall Gevrey asymptotics in one complex dimension.

%===============================================================================
\paragraph{Gevrey asymptotics in one dimension.}
\label{200722160857}
%===============================================================================
A holomorphic function $f \in \cal{O} (S)$ defined on a sectorial domain $S$ with opening $A$ is said to admit \dfn{Gevrey asymptotics as $\hbar \to 0$ along $A$} if the constants $\CC_{n,0}$ in \eqref{200720153758} depend on $n$ like $\CC_0 \MM_0^n n!$.
More explicitly, there is a formal power series $\hat{f} (\hbar) \in \Complex \bbrac{\hbar}$ such that for every compactly contained subarc $A_0 \Subset A$, there is a sectorial domain $S_0 \subset S$ with opening $A_0 \Subset A$ and real constants $\CC_0, \MM_0 > 0$ which for all $n \geq 0$ give the bounds
\eqntag{\label{200722160158}
	\big| \RR_n (\hbar) \big| \leq \CC_0 \MM_0^n n! |\hbar|^n
}
for all $\hbar \in S_0$.
We will use the symbol ``$\simeq$'' to distinguish Gevrey asymptotics from Poincaré asymptotics.
Thus, if $f$ admits Gevrey asymptotics as $\hbar \to 0$ along $A$, we will write
\eqntag{\label{210225131044}
	f (\hbar) \simeq \hat{f} (\hbar)
\qqqquad
	\text{as $\hbar \to 0$ along $A$\fullstop}
}
We denote the subring of sectorial germs satisfying this property by $\cal{G} (A) \subset \cal{A} (A)$.
If in addition to \eqref{200722160158}, the constants $\CC_0, \MM_0$ can be chosen uniformly for all $A_0 \Subset A$, then we will write
\eqntag{\label{210225134416}
	f (\hbar) \simeq \hat{f} (\hbar)
\qqqquad
	\text{as $\hbar \to 0$ along $\bar{A}$\fullstop}
}
We denote the subring of sectorial germs satisfying this property by $\cal{G} (\bar{A}) \subset \cal{G} (A)$.
Explicitly, $f \in \cal{O} (\bar{A})$ if and only if there is a sectorial domain $S_0$ with opening $A$ and real constants $\CC, \MM > 0$ which give the following bounds for all $n \geq 0$ and $\hbar \in S_0$:
\eqntag{\label{210225134829}
	\big| \RR_n (\hbar) \big| \leq \CC \MM^n n! |\hbar|^n
\fullstop
}

%===============================================================================
\paragraph{Gevrey series in one dimension.}
\label{200722151439}
%===============================================================================
It is easy to see that if $f \in \cal{G} (A)$ then the coefficients of its asymptotic expansion also grow essentially like $n!$.
By definition, a formal power series $\hat{f} (\hbar) = \sum f_n \hbar^n \in \Complex \bbrac{\hbar}$ is a \dfn{Gevrey power series} if there are constants $\CC, \MM > 0$ such that for all $n \geq 0$,
\eqntag{\label{200723182724}
	| f_n | \leq \CC \MM^n n!
\fullstop
}
Gevrey series form a subring $\Gevrey \bbrac{\hbar} \subset \Complex \bbrac{\hbar}$, and the asymptotic expansion map restricts to a ring homomorphism $\ae : \cal{G}(A) \to \Gevrey \bbrac{\hbar}$.
The ring $\cal{G} (\bar{A})$ for an arc with opening $|A| = \pi$ plays a central role in Gevrey asymptotics because it is possible to identify and describe a subclass $\bar{\Gevrey} \bbrac{\hbar} \subset \Gevrey \bbrac{\hbar}$ such that the asymptotic expansion map restricts to a bijection $\ae : \cal{G} (\bar{A}) \iso \bar{\Gevrey} \bbrac{\hbar}$.
This identification is done using a theorem of Nevanlinna and requires techniques from the Borel-Laplace theory, see \autoref{210616130753}.

%===============================================================================
\paragraph{Examples in one dimension.}
\label{210617074256}
Any function which is holomorphic at $\hbar = 0$ automatically admits Gevrey asymptotics as $\hbar \to 0$ along any arc: its asymptotic expansion is nothing but its convergent Taylor series at $\hbar =  0$ whose coefficients necessarily grow at most exponentially.
Also, if a function admits Gevrey asymptotics as $\hbar \to 0$ along a strictly larger arc $A'$ which contains the closed arc $\bar{A}$, then it automatically admits Gevrey asymptotics along $\bar{A}$.

The function $e^{-1/\hbar}$ admits Gevrey asymptotics as $\hbar \to 0$ in the right halfplane where it is asymptotic to $0$.
So if $A = (-\pi/2, +\pi/2)$, then $e^{-1/\hbar} \in \cal{G} (A)$, but $e^{-1/\hbar} \not\in \cal{G} (\bar{A})$ as discussed before.
On the other hand, the function $e^{-1/\hbar^\alpha}$ for any real number $0 < \alpha < 1$ does not admit Gevrey asymptotics as $\hbar \to 0$ along any arc.

%===============================================================================
\paragraph{Gevrey series with holomorphic coefficients.}
\label{210225135829}
Now, suppose again that in addition we have a domain $U \subset \Complex_x$.
A formal power series $\hat{f} (x, \hbar) \in \cal{O} (U) \bbrac{\hbar}$ on $U$ is called a \dfn{Gevrey series} if, for every $x \in U$, there are constants $\CC_x, \MM_x > 0$ such that, for all $n \geq 0$,
\eqntag{\label{190303174313}
	\big| f_n (x) \big| \leq \CC_x \MM_x^n n!
\fullstop
}
Furthermore, $\hat{f}$ is a \dfn{uniformly Gevrey series} on $U$ if the constants $\CC_x, \MM_x$ can be chosen to be independent of $x \in U$.
$\hat{f}$ is a \dfn{locally uniformly Gevrey series} on $U$ if every $x_0 \in U$ has a neighbourhood $U_0 \subset U$ where $\hat{f}$ is uniformly Gevrey.
Such power series form subrings $\cal{G}_\rm{u} (U) \bbrac{\hbar}$ and $\cal{G} (U) \bbrac{\hbar}$ of $\cal{O} (U) \bbrac{\hbar}$ respectively.
The ring $\cal{G} (U) \bbrac{\hbar}$ is preserved by $\del_x$.

%===============================================================================
\paragraph{Uniform Gevrey asymptotics.}
\label{210225135947}
%===============================================================================
Again, consider the product space $\UUU = U \times S$ where $S$ is a sectorial domain with opening $A$, or more generally let $\UUU$ be a domain of the form described in \autoref{210226161539}.
A holomorphic function $f \in \cal{O} (\UUU)$ is said to admit (pointwise) \dfn{Gevrey asymptotics on $U$ as $\hbar \to 0$ along $A$} if for every $x \in U$ and every compactly contained subarc $A_0 \Subset A$, there is a sectorial domain $S_0$ with opening $A_0$ and constants $\CC_{x,0}, \MM_{x,0} > 0$ such that for all $n \geq 0$ and all $\hbar \in S_0$,
\eqntag{\label{200624162243}
	\big| \RR_n (x, \hbar) \big|
		= \left| f(x, \hbar) - \sum_{k=0}^{n-1} f_k (x) \hbar^k \right|
		\leq \CC_{x,0} \MM^n_{x,0} n! |\hbar|^n
\fullstop
}
We say that $f$ admits \dfn{uniform Gevrey asymptotics} on $U$ as $\hbar \to 0$ along $A$ if the constants $\CC_{x,0}, \MM_{x,0}$ can be chosen to be independent of $x \in U$ (i.e., so that $\CC_{x,0} = \CC_0, \MM_{x,0} = \MM_0$).
We also say $f$ admits \dfn{locally uniform Gevrey asymptotics} on $U$ if every point $x_0 \in U$ has a neighbourhood $U_0 \subset U$ on which $f$ has uniform Gevrey asymptotics.
In these cases, we write, respectively,
\eqnstag{\label{210225142615}
	f (x,\hbar) &\simeq \hat{f} (x,\hbar)
\qqqquad
	\text{as $\hbar \to 0$ along $A$, unif. $\forall x \in U$\fullstop{;}}
\\
\label{210225142617}
	f (x,\hbar) &\simeq \hat{f} (x,\hbar)
\qqqquad
	\text{as $\hbar \to 0$ along $A$, loc.unif. $\forall x \in U$\fullstop}
}
Such functions form subrings $\cal{G}_\rm{u} (U; A) \subset \cal{A}_\rm{u} (U; A)$ and $\cal{G} (U; A) \subset \cal{A} (U; A)$ respectively.
The asymptotic expansion map $\ae$ restricts to ring homomorphisms $\cal{G}_\rm{u} (U;A) \to \cal{G}_\rm{u} (U) \bbrac{\hbar}$ and $\cal{G} (U;A) \to \cal{G} (U) \bbrac{\hbar}$.
As in the case of uniform Poincaré asymptotics, an application of the Cauchy integral formula shows that the ring $\cal{G} (U; A)$ (but not the ring $\cal{G}_\rm{u} (U; A)$) is preserved by differentiation: it has the property $\del_x \big( \cal{G} (U; A) \big) \subset \cal{G} (U; A)$.

%===============================================================================
\paragraph{Uniform Gevrey asymptotics along a closed arc.}
\label{210225142900}
%===============================================================================
If in addition to \eqref{210225142615} or \eqref{210225142617}, the constants $\CC_{x,0}, \MM_{x,0}$ can be chosen uniformly for all $A_0 \Subset A$ (i.e., so that $\CC_{x,0} = \CC_x$ and $\MM_{x,0} = \MM_x$), we will write, respectively,
\eqnstag{\label{210226091301}
	f (x,\hbar) &\simeq \hat{f} (x,\hbar)
\qqqquad
	\text{as $\hbar \to 0$ along $\bar{A}$, unif. $\forall x \in U$\fullstop{;}}
\\
\label{210226091307}
	f (x,\hbar) &\simeq \hat{f} (x,\hbar)
\qqqquad
	\text{as $\hbar \to 0$ along $\bar{A}$, loc.unif. $\forall x \in U$\fullstop}
}
Such functions form subrings $\cal{G}_\rm{u} (U; \bar{A}) \subset \cal{G}_\rm{u} (U; \bar{A})$ and $\cal{G} (U; \bar{A}) \subset \cal{G} (U; A)$ respectively.
Again, we have $\del_x \big( \cal{G} (U; \bar{A}) \big) \subset \cal{G} (U; \bar{A})$.

%===============================================================================
\paragraph{Example.}
\label{200708191952}
%===============================================================================
The following example is related to what is sometimes called the \textit{Euler series} \cite[Example 1.1.4]{MR3495546}.
Consider the following formal series on $U = \Complex_x^\ast$:
\eqntag{\label{200722153024}
	\hat{\EE} (x, \hbar) 
		\coleq - \sum_{k=1}^\infty (-x)^{-k} (k-1)! \hbar^k
		= \frac{\hbar}{x} \sum_{k=0}^\infty \frac{k!}{(-x)^k} \hbar^k
	~\in~ \cal{O} (U) \bbrac{\hbar}
\fullstop
}
Incidentally, $\hat{\EE}$ is a formal solution of the differential equation $\hbar^2 \del_\hbar \EE + x \EE = \hbar$, but this fact is not important for the discussion in this example.

It is easy to see that the power series $\hat{\EE}$ has zero radius of convergence for any fixed nonzero $x$, but it is a Gevrey series for which the bounds \eqref{190303174313} can be satisfied by taking $\CC_x = 1$ and $\MM_x = |x|^{-k}$.
This demonstrates that $\hat{\EE}$ is a locally uniform Gevrey series on $\Complex_x^\ast$.
One can show that $\hat{\EE}$ is not a uniform Gevrey series.
Thus, $\hat{\EE} \in \cal{G} (U) \bbrac{\hbar}$ but $\not\in \cal{G}_\rm{u} (U) \bbrac{\hbar}$.

Consider the function
\eqn{
	\EE (x, \hbar)
		\coleq \int_0^{+\infty} \frac{e^{- \xi/\hbar}}{x + \xi} \dd{\xi}
\fullstop
}
It is well-defined and holomorphic for all $x$ in the cut plane $U' \coleq \Complex_x \setminus \Real_-$ and all $\hbar$ with ${\Re (\hbar) > 0}$.
It is not holomorphic at $\hbar = 0$ for any $x \in U$, but it is bounded as $\hbar \to 0$ in the right halfplane.
In fact, $\EE$ admits the power series $\hat{\EE}$ as its locally uniform Gevrey asymptotics:
\eqntag{\label{210225150455}
	\EE (x, \hbar) \simeq \hat{\EE} (x, \hbar)
\qquad
	\text{as $\hbar \to 0$ along $[-\nicefrac{\pi}{2}, +\nicefrac{\pi}{2}]$, loc.unif. $\forall x \in U'$\fullstop}
}
In symbols, $\EE \in \cal{G} (U'; A)$.
To see this, we can write:
\eqn{
	\frac{1}{x + \xi}
	= \frac{1}{x} \frac{1}{1 + \xi/x}
	= \frac{1}{x} \sum_{k=0}^{n-2} 
		\frac{1}{(-x)^k} \xi^k
		+ \frac{1}{(-x)^{n-1}} \frac{\xi^{n-1}}{x + \xi}
\fullstop
}
Therefore, we obtain the relation
\eqn{
	\EE (x, \hbar) 
		= \frac{\hbar}{x} \sum_{k=0}^{n-2} \frac{k!}{(-x)^k} \hbar^k
			+ \frac{1}{(-x)^{n-1}} \int_0^{+\infty} 
				\frac{\xi^{n-1} e^{- \xi / \hbar}}{x + \xi} \dd{\xi}
\fullstop
}
So to demonstrate \eqref{210225150455}, one can find a locally uniform bound on the integral which is valid uniformly for all directions in the halfplane arc $(-\pi/2, +\pi/2)$.

\section{Basics of Exponential Asymptotics}
\setcounter{section}{2}
\setcounter{paragraph}{0}
\label{210220170857}
%===============================================================================
%===============================================================================

In this appendix section, we define the appropriate notion of parametric asymptotics necessary for the analysis of linear ODEs of the form \eqref{210115121038}.
As ever, for the reader's convenience, we begin with a description in one complex variable.

%===============================================================================
%===============================================================================
\subsection{Exponential Asymptotics in One Dimension}
%===============================================================================
%===============================================================================

\begin{defn}{210218205742}
An \dfn{exponential power series} in $\hbar$ is a formal expression of the form
\eqn{
	\hat{\psi} = e^{\Phi / \hbar} \, \hat{\Psi}
}
where $\Phi = \Phi (\hbar^{-1}) \in \Complex [\hbar^{-1}]$ and $\hat{\Psi} = \hat{\Psi} (\hbar) \in \Complex \bbrac{\hbar}$.
Moreover, $\hat{\psi}$ is an \dfn{exponential Gevrey series} if $\hat{\Psi} \in \Gevrey \bbrac{\hbar}$.
The polynomial $\Phi$ is called the \dfn{exponent} of $\hat{\psi}$.
\end{defn}

Notice that the exponential prefactor $e^{\Phi / \hbar}$ is just a holomorphic function of $\hbar \in \Complex^\ast$ that has an essential singularity at $\hbar = 0$.
The usual formal power series $\Complex \bbrac{\hbar}$ are the exponential power series with $\Phi = 0$.

%===============================================================================
\paragraph{Exponential transseries.}
The set of all exponential power series in $\hbar$ with a fixed exponent $\Phi$ has the structure of a $\Complex$-vector space, but unless $\Phi$ is the zero polynomial it does not inherit the usual ring structure because it is not closed under products.
At the same time, the set of all exponential power series in $\hbar$ is closed under multiplication but it loses the structure of a $\Complex$-vector space because sums of exponential power series with distinct exponents cannot be written as exponential power series.
We therefore want to consider finite linear combinations of exponential power series to yield a well-behaved algebraic structure.
So we introduce the following definition.

%===============================================================================
\begin{defn}{210226094325}
We define the ring of \dfn{exponential transseries} in $\hbar$ as the set of all formal finite combinations of exponential power series:
\eqn{
	\Complex^{\exp} \bbrac{\hbar}
	\coleq \set{ \left. ~\sum_k^{\textup{\tiny{finite}}} e^{\Phi_k / \hbar} 
				\, \hat{\Psi}_k
		 ~~\right|~ \Phi_k \in \Complex [\hbar^{-1}], 
		 		\, \hat{\Psi}_k \in \Complex \bbrac{\hbar}}
\fullstop
}
Similarly, an \dfn{exponential Gevrey transseries} in $\hbar$ is one such that each formal power series $\hat{\Psi}_k$ is a Gevrey series.
The ring of exponential Gevrey transseries will be denoted by $\Gevrey^{\exp} \bbrac{\hbar}$.
\end{defn}

%===============================================================================
\begin{rem}{210616171345}
Our ring $\Complex^{\exp} \bbrac{\hbar}$ is a much simpler version of analogous formal objects first systematically considered in asymptotic analysis by Ilyashenko in \cite{zbMATH00052132} (see \textit{Dulac's exponential series}), and at the same time independently by Écalle in \cite{zbMATH06048675} (whence the terminology ``transseries'' is borrowed).
(See also \cite{MR1209700}.)
Since then, transseries have grown into a vast subject which we do not need here.
Quite approachable introductions may be found in \cite{zbMATH02202755,zbMATH05904821,zbMATH07130996}.
\end{rem}

%===============================================================================
\begin{defn}{210616171635}
A holomorphic function $\psi \in \cal{O} (A)$ admits \dfn{exponential asymptotics as $\hbar \to 0$ along $A$ (resp.\ $\bar{A}$)} if there exists a polynomial $\Phi \in \Complex [\hbar^{-1}]$ such that the holomorphic function $e^{-\Phi / \hbar} \psi$ admits asymptotics as $\hbar \to 0$ along $A$ (resp.\ $\bar{A}$) in the usual sense of \eqref{200720175735} (resp.\ \eqref{210220160756}).
In symbols, $e^{-\Phi / \hbar} \psi \in \cal{A} (A)$ (resp.\ $\in \cal{A} (\bar{A})$).
In this case, we call $\Phi$ the \dfn{asymptotic exponent} of $\psi$.
If $\hat{\Psi} \in \Complex \bbrac{\hbar}$ is the asymptotic expansion of $e^{- \Phi / \hbar} \psi$ in the usual sense, then we say that $\hat{\psi} = e^{\Phi / \hbar} \, \hat{\Psi}$ is the \dfn{exponential asymptotic expansion} of $\psi$, and write:
\eqntag{\label{210225202543}
	\psi \sim \hat{\psi} = e^{\Phi / \hbar} \, \hat{\Psi}
\qqqquad
	\text{as $\hbar \to 0$ along $A$ (resp.\ $\bar{A}$)\fullstop}
}
Similarly, we will say that $\psi \in \cal{O} (A)$ admits \dfn{exponential Gevrey asymptotics as $\hbar \to 0$ along $A$ (resp.\ $\bar{A}$)} if $e^{- \Phi / \hbar} \psi$ admits a Gevrey asymptotic expansion $\hat{\Psi} \in \Gevrey \bbrac{\hbar}$ as $\hbar \to 0$ along $A$ (resp.\ $\bar{A}$) in the usual sense of \eqref{210225131044} (resp.\ \eqref{210225134416}).
In symbols, $e^{- \Phi / \hbar} \psi \in \cal{G} (A)$ (resp.\ $\in \cal{G} (\bar{A})$).
 and we continue to write
\eqntag{\label{210225203332}
	\psi \simeq \hat{\psi} = e^{\Phi / \hbar} \, \hat{\Psi}
\qqqquad
	\text{as $\hbar \to 0$ along $A$ (resp.\ $\bar{A}$)\fullstop}
}
\end{defn}

%===============================================================================
\paragraph{}
\label{210617074224}
We denote by $\cal{A}^{\exp} (A)$ and $\cal{G}^{\exp} (A)$ (and similarly for $\bar{A}$) the rings of all holomorphic functions $\psi \in \cal{O} (A)$ which admit exponential Poincaré and Gevrey asymptotics in the sense of \eqref{210225202543} and \eqref{210225203332}, respectively.
The usual asymptotic expansion map $\ae : \cal{A} (A) \to \Complex \bbrac{\hbar}$ extends to a map $\ae : \cal{A}^{\exp} (A) \to \Complex^{\exp} \bbrac{\hbar}$.
However, it is no longer a ring homomorphism because of the complicated dominance relations brought about by the exponential prefactors.
It is possible to formalise this phenomenon (as is indeed done in the more general subject of transseries \cite{zbMATH02202755}), but we do not need this here.

For a simple example, the holomorphic function $e^{1/\hbar}$ does not admit an asymptotic expansion as $\hbar \to 0$ along any subarc of $(-\nicefrac{\pi}{2},+\nicefrac{\pi}{2})$, but it does admit an exponential Gevrey asymptotic expansion with exponent $\Phi = 1$ and $\hat{\Psi} = 1$.
In symbols, $e^{1/\hbar} \in \cal{G}^{\exp} (A)$ but $\not\in \cal{A} (A)$.

%===============================================================================
%===============================================================================
\subsection{Uniform Exponential Asymptotics}
%===============================================================================
%===============================================================================

Now we suppose that in addition we have a domain $U \subset \Complex_x$.
We make all the definitions from the previous subsection uniform in the variable $x$.

%===============================================================================
\begin{defn}{210226085039}
An \dfn{exponential power series} in $\hbar$ with holomorphic coefficients on $U$ is a formal expression of the form
\eqn{
	\hat{\psi} = e^{\Phi / \hbar} \, \hat{\Psi}
}
where $\Phi = \Phi (x, \hbar^{-1}) \in \cal{O} (U) [\hbar^{-1}]$ and $\hat{\Psi} = \hat{\Psi} (x, \hbar) \in \cal{O} (U) \bbrac{\hbar}$.
The polynomial $\Phi$ is called the \dfn{exponent} of $\hat{\psi}$.
We will say that $\hat{\psi}$ is a \dfn{uniform} or \dfn{locally uniform exponential Gevrey series} if $\hat{\Psi} \in \cal{G}_\rm{u} (U) \bbrac{\hbar}$ or $\hat{\Psi} \in \cal{G} (U) \bbrac{\hbar}$, respectively.
\end{defn}

Like in the single-variable case, the exponential prefactor $e^{\Phi / \hbar}$ is a holomorphic function of $(x, \hbar) \in U \times \Complex^\ast$.
The differentiation $\del_x \hat{\psi}$ is defined in the obvious way using the Leibniz rule.
Notice that the set of all exponential power series on $U$ with a fixed exponent $\Phi$ is not preserved by the action of the derivative $\del_x$ unless $\Phi$ is constant in $x$.
But it is preserved by the operator $\hbar^{m+1} \del_x$ where $m \geq 0$ is the degree of $\Phi$ in $\hbar^{-1}$.

%===============================================================================
\begin{defn}{210226114421}
We define the ring of \dfn{exponential transseries} with holomorphic coefficients on $U$ as the set of all finite combinations of exponential power series with holomorphic coefficients:
\eqn{
	\cal{O}^{\exp} (U) \bbrac{\hbar}
	\coleq \set{ \left. ~\sum_k^{\textup{\tiny{finite}}} e^{\Phi_k / \hbar} 
				\, \hat{\Psi}_k
		 ~~\right|~ \Phi_k \in \cal{O} (U)  [\hbar^{-1}], 
		 		\, \hat{\Psi}_k \in \cal{O} (U)  \bbrac{\hbar}}
\fullstop
}
Similarly, we define subrings of uniform and locally uniform \dfn{exponential Gevrey transseries} on $U$ by requiring that each $\hat{\Psi}_k$ is in $\cal{G}_\rm{u} (U) \bbrac{\hbar}$ or $\cal{G} (U) \bbrac{\hbar}$.
These subrings are denoted by $\cal{G}_\rm{u}^{\exp} (U) \bbrac{\hbar}$ or $\cal{G}^{\exp} (U) \bbrac{\hbar}$, respectively.
\end{defn}

%===============================================================================
\begin{defn}{210226090851}
A holomorphic function $\psi \in \cal{O} (U; A)$ admits \dfn{uniform} or \dfn{locally uniform exponential asymptotics as $\hbar \to 0$ along $A$ (resp.\ $\bar{A}$)} if there exists $\Phi = \Phi (x, \hbar^{-1}) \in \cal{O} (U) [\hbar^{-1}]$ such that the holomorphic function $e^{-\Phi / \hbar} \psi \in \cal{O} (U; A)$ admits respectively uniform or locally uniform asymptotics as $\hbar \to 0$ along $A$ (resp.\ $\bar{A}$) in the usual sense of \eqref{210225121434}-\eqref{210225121453} (resp.\ \eqref{210225123426}-\eqref{210225123428}).
In this case, we call $\Phi$ the \dfn{asymptotic exponent} of $\psi$, and if $\hat{\Psi} \in \cal{O} (U) \bbrac{\hbar}$ is the asymptotic expansion of $e^{- \Phi / \hbar} \psi$ in the usual sense, then we say that $\hat{\psi} = e^{\Phi / \hbar} \, \hat{\Psi}$ is the \dfn{exponential asymptotic expansion} of $\psi$.
We denote these facts respectively by
\eqnstag{\label{210226092415}
	\psi &\sim \hat{\psi} = e^{\Phi / \hbar} \, \hat{\Psi}
\qqqquad
	\text{as $\hbar \to 0$ along $A$ (resp.\ $\bar{A}$), unif. $\forall x \in U$\fullstop{;}}
\\\label{210226092421}
	\psi &\sim \hat{\psi} = e^{\Phi / \hbar} \, \hat{\Psi}
\qqqquad
	\text{as $\hbar \to 0$ along $A$ (resp.\ $\bar{A}$), loc.unif. $\forall x \in U$\fullstop}
}
Similarly, we say that $\psi \in \cal{O} (U; A)$ admits \dfn{uniform} or \dfn{locally uniform exponential Gevrey asymptotics as $\hbar \to 0$ along $A$ (resp.\ $\bar{A}$)} if $e^{- \Phi / \hbar} \psi$ admits respectively uniform or locally uniform Gevrey asymptotics as $\hbar \to 0$ along $A$ (resp.\ $\bar{A}$) in the usual sense of \eqref{210225142615}-\eqref{210225142617} (resp.\ \eqref{210226091301}-\eqref{210226091307}).
We denote these facts respectively by
\eqnstag{\label{210226092412}
	\psi &\simeq \hat{\psi} = e^{\Phi / \hbar} \, \hat{\Psi}
\qqqquad
	\text{as $\hbar \to 0$ along $A$ (resp.\ $\bar{A}$), unif. $\forall x \in U$\fullstop{;}}
\\\label{210226092501}
	\psi &\simeq \hat{\psi} = e^{\Phi / \hbar} \, \hat{\Psi}
\qqqquad
	\text{as $\hbar \to 0$ along $A$ (resp.\ $\bar{A}$), loc.unif. $\forall x \in U$\fullstop}
}
\end{defn}

We denote the corresponding rings of holomorphic functions satisfying these definitions by $\cal{A}^{\exp} (U; A)$, $\cal{G}^{\exp} (U; A)$, and so on; i.e., by adding a superscript ``$\exp$'' in the notation for the usual asymptotics.

%===============================================================================
%===============================================================================
%===============================================================================
\section{Basics of the Borel-Laplace Theory}
\label{210616130753}
%===============================================================================
%===============================================================================
%===============================================================================

In this appendix section, we recall some basic definitions from the theory of Borel-Laplace transforms.

%===============================================================================
\paragraph{}
Let $U \subset \Complex_x$ be a domain.
Fix a direction $\theta \in \Sphere^1$, let $A_\theta$ be the halfplane arc bisected by $\theta$, and let $S_\theta$ be the Borel disc bisected by $\theta$ of some diameter $\delta > 0$:
\eqntag{\label{210616181227}
	A_\theta \coleq (\theta -\tfrac{\pi}{2}, \theta + \tfrac{\pi}{2})
\qtext{and}
	S_\theta \coleq \set{ \Re \big( \smash{e^{i\theta}} / \hbar \big) > 1/\delta}
\fullstop
}
Introduce another complex plane $\Complex_\xi$, sometimes called the \dfn{Borel plane}.
In the same vein, the complex plane $\Complex_\hbar$ is sometimes called the \dfn{Laplace plane}.
Let $e^{i\theta} \Real_+ \subset \Complex_\xi$ be the nonnegative real ray in the direction $\theta$.
By a \dfn{tubular neighbourhood} of $e^{i\theta} \Real_+$ of some \textit{thickness} $\epsilon > 0$ we mean a domain of the form 
\eqntag{
	\Xi_\theta \coleq \set{ \xi \in \Complex_\xi ~\big|~ \op{dist} (\xi, e^{i\theta}\Real_+) < \epsilon}
\fullstop
}

%===============================================================================
%===============================================================================
\subsection{The Laplace Transform}
%===============================================================================
%===============================================================================

%===============================================================================
\paragraph{}
Let us first recall some well-known properties of the Laplace transform.
Suppose $\Xi_\theta \subset \Complex_\xi$ is a tubular neighbourhood of $e^{i \theta} \Real_+$, and $\phi = \phi (x, \xi)$ is a holomorphic function on $U \times \Xi_\theta$.
Its \dfn{Laplace transform} in the direction $\theta$ is defined by the formula:
\eqntag{\label{200624181217}
	\Laplace_\theta [\, \phi \,] (x, \hbar)
		\coleq \int\nolimits_{e^{i\theta} \Real_+} \phi (x, \xi) e^{-\xi/\hbar} \dd{\xi}
\fullstop
}
The function $\phi$ is called \dfn{uniformly} or \dfn{locally uniformly Laplace transformable} in the direction $\theta$ if this integral is respectively uniformly or locally uniformly convergent.
Clearly, $\phi$ is uniformly Laplace transformable in the direction $\theta$ if $\phi$ has uniform \dfn{at-most-exponential growth} as $|\xi| \to + \infty$ along the ray $e^{i\theta} \Real_+$.
Explicitly, this means there are constants $\AA, \LL > 0$ such that for all $(x,\xi) \in U \times \Xi_\theta$,
\eqntag{
	\big| \phi (x, \xi) \big| \leq \AA e^{\LL |\xi|}
\fullstop
}

%===============================================================================
\paragraph{Properties of the Laplace transform.}
Recall that the Laplace transform converts (a) the convolution product of functions into the pointwise multiplication of their Laplace transforms, and (b) differentiation by $\xi$ into multiplication by $\hbar^{-1}$.
Thus, if $\phi, \sigma$ are two uniformly Laplace transformable holomorphic functions on $U \times \Xi_\theta$, then:
\eqn{
	\Laplace_\theta [\, \phi \ast \sigma \,] (x, \hbar)
		= \Laplace_\theta [\, \phi \, ] (x, \hbar) \cdot \Laplace_\theta [\, \sigma \, ] (x, \hbar)
\qtext{and}
	\Laplace_\theta [\, \del_\xi \phi \, ] (x, \hbar)
		= \hbar^{-1} \Laplace_\theta [\, \phi \,] (x, \hbar)
\fullstop
}
Let us also note that the convolution product is taken with respect to the variable $\hbar$ and recall that it is defined by the following formula:
\eqntag{
	\phi \ast \sigma (x, \xi)
	\coleq 
	\int\nolimits_0^\xi \phi (x, \xi - y) \sigma (x, y) \dd{y}
\fullstop{,}
}
where the path of integration is a straight line segment from $0$ to $\xi$.
Finally, if $\phi$ is a uniformly Laplace transformable holomorphic function on $U \times \Xi_\theta$, then for any $x_0 \in U$, and all $(x, \xi) \in U \times \Xi_\theta$,
\eqntag{
	\Laplace_\theta \left[ \, \int\nolimits_{x_0}^x \phi (t, \xi) \dd{t} \, \right]
		= \int\nolimits_{x_0}^x \Laplace_\theta \big[ \, \phi \, \big] (t, \xi) \dd{t}
\fullstop{,}
}
where the path of integration is assumed to lie entirely in $U$.

%===============================================================================
%===============================================================================
\subsection{The Borel Transform}
%===============================================================================
%===============================================================================

%===============================================================================
\paragraph{}
Let $f = f(x, \hbar)$ be a holomorphic function on a product domain $U \times S_\theta$, or more generally a holomorphic function on the pair $(U, A_\theta)$ as described in \autoref{210226161539}.
In the latter situation, given $x \in U$, take a sufficiently small Borel disc $S_\theta$ such that $U_0 \times S_\theta$ is contained in $\UUU$ for some neighbourhood $U_0$ of $x$.

The \dfn{analytic Borel transform} (a.k.a., the \dfn{inverse Laplace transform}) of $f$ in the direction $\theta$ is defined by the following formula:
\eqntag{\label{210617101748}
	\Borel_\theta [\, f \,] (x, \xi)
		\coleq \frac{1}{2\pi i} \oint\nolimits_\theta f(x, \hbar) e^{\xi / \hbar} \frac{\dd{\hbar}}{\hbar^2}
\fullstop
}
Here, the notation ``$\oint_\theta$'' means that the integration is done along the boundary $\wp_\theta \coleq \set{ \Re (e^{i\theta}/\hbar) = 1 / \delta'}$ of a Borel disc $S'_\theta \subsetneq S_\theta$ of strictly smaller diameter $\delta' < \delta$, traversed anticlockwise (i.e., emanating from the singular point $\hbar = 0$ in the direction $\theta - \pi/2$ and reentering in the direction $\theta + \pi/2$).
Observe that the integral kernel $e^{\xi / \hbar} \hbar^{-2}$ has an essential singularity at $\hbar = 0$ for every nonzero $\xi$.
We therefore interpret this improper integral as the Cauchy principal value, which means it is defined as the limit of an integral over a sequence of path segments on the boundary $\del S'_\theta$ approaching the singular point $\hbar = 0$ in both directions at the same rate.
Explicitly, parameterise the boundary path $\wp_\theta$ by $t \in \Real$ as $\hbar (t) = e^{i \theta} ( \delta' + it )^{-1}$.
Then for every $\TT > 0$, we take the path segment $\wp_\theta (\TT)$ for $t \in [-\TT,+\TT]$ and define
\eqntag{\label{210617080526}
	\oint\nolimits_\theta \coleq \lim_{\TT \to +\infty} \int\nolimits_{\wp_\theta (\TT)}
}
Finally, since the integrand $f(x, \hbar) e^{\xi / \hbar} \hbar^{-2}$ is holomorphic in the interior of the Borel disc $S_\theta$, it is not difficult to see that the integral in \eqref{210617101748} is independent of $\delta'$, provided that it exists.
If the integral \eqref{210617101748} converges for all $\xi \in e^{i\theta} \Real_+$ and uniformly (resp. locally uniformly) for all $x \in U$, then we say $f$ is \dfn{uniformly} (resp. \dfn{locally uniformly}) \dfn{Borel transformable} in the direction $\theta$.
The following lemma gives a criterion for Borel transformability.

%===============================================================================
\begin{lem}{210617101734}
If $f \in \cal{O} (U; A_\theta)$ admits uniform (resp. locally uniform) Gevrey asymptotics as $\hbar \to 0$ along the closed arc $\bar{A}_\theta$ (in symbols, $f \in \cal{G} (U; \bar{A}_\theta)$), then it is uniformly (resp. locally uniformly) Borel transformable in the direction $\theta$.
\end{lem}

A proof of this proposition (which is a simple complex analytic argument using the Gevrey bounds \eqref{210225134829}) can be found in \cite[Lemma B.5 and Theorem B.11]{MY2008.06492}.
We stress the importance in this lemma of having Gevrey asymptotics along the \textit{closed} arc $\bar{A}_\theta$ and not just $A_\theta$.
For example, recall from \autoref{210617074256} that the function $f(\hbar) = e^{-1/\hbar}$ is in $\cal{G} (A)$, where $A = (-\pi/2, +\pi/2)$, but not in $\cal{G} (\bar{A})$, and we can see that its Borel transform in the direction $\theta = 0$ is not well-defined when $\xi = 1$.

%===============================================================================
\paragraph{}
If $f$ is holomorphic at $\hbar = 0$ (i.e., if $f \in \cal{O} (U) \set{\hbar}$), then it is necessarily Borel transformable in every direction, and all these Borel transforms agree and define a holomorphic function $\Borel [\, f \,] (x, \xi)$ of $(x, \xi) \in U \times \Complex_\xi$.
In fact, in this case the integration contour in \eqref{210617101748} can be deformed to a circle around the origin, so that the Borel transform of $f$ for any $\theta$ is nothing but the residue integral:
\eqntag{\label{210617101920}
	\Borel [\, f \,] (x, \xi)
		= \underset{\hbar = 0}{\Res}\: \frac{f(x, \hbar) e^{\xi / \hbar}}{\hbar^2}
\fullstop
}

%===============================================================================
\paragraph{}
Using the residue calculus expression \eqref{210617101920}, it is easy to deduce the following helpful formulas:
\eqntag{\label{200704112520}
	\Borel [\, 1 \,] = 0
\qqqtext{and}
	\Borel [\, \hbar^{k+1} \,] = \frac{\xi^k}{k!}
\qquad \text{for all $k \geq 0$.}
}
They can be used to extend the Borel transform to formal power series in $\hbar$ by defining the \dfn{formal Borel transform}:
for any $\hat{f} (x, \hbar) \in \cal{O} (U) \bbrac{\hbar}$,
\eqntag{
\mbox{}\hspace{-10pt}
	\hat{\phi} (x, \xi) =
	\hat{\Borel} [ \, \hat{f} \, ] (x, \xi)
		\coleq \sum_{n=0}^\infty \phi_n (x) \xi^n
		\in \cal{O} (U) \bbrac{\xi}
\qtext{where}
	\phi_k (x) \coleq \tfrac{1}{k!} f_{k+1} (x)
\GREY{.}
}
Thus, the formal Borel transform essentially `divides' the coefficients of the power series by $n!$.
The following lemma follows immediately from this formula and the Gevrey power series bounds \eqref{190303174313}.

%===============================================================================
\begin{lem}{210617102534}
If $\hat{f}$ is a uniformly or locally-uniformly Gevrey series on $U$ (in symbols, $\hat{f} \in \cal{G}_\rm{u} (U) \bbrac{\hbar}$ or $\hat{f} \in \cal{G} (U) \bbrac{\hbar}$, respectively), then its formal Borel transform $\hat{\phi}$ is respectively a uniformly or locally uniformly convergent series in $\xi$.
In symbols, $\hat{\phi} \in \cal{O}_\rm{u} (U) \set{\xi}$ or $\hat{\phi} \in \cal{O} (U) \set{\xi}$, respectively.
\end{lem}

%===============================================================================
\paragraph{Properties of the Borel transform.}
%===============================================================================
Finally, we mention a few important properties of the Borel transform.
Given two uniformly Borel $\theta$-transformable holomorphic functions $f,g$ on $(U, A_\theta)$, the following identities hold for all $(x, \xi) \in U \times \Xi_\theta$:
\begin{gather}
	\Borel_\theta [ \, fg \, ]
		= \big( \Borel_\theta [ \, f \, ] \ast \Borel_\theta [ \, g \, ] \big)
\GREY{;}
\\
	\Borel_\theta [ \, \hbar^{-1} f \, ] = \del_\xi \Borel_\theta [ \, f \, ]
\qtext{and}
	\Borel_\theta [ \, \del_x f \, ] = \del_x \Borel_\theta [ \, f \, ]
\fullstop{;}
\\
	\Borel_\theta \left[ \, \int\nolimits_{x_0}^x f (t, \hbar) \dd{t} \, \right] (x, \xi)
		= \int\nolimits_{x_0}^x \Borel_\theta \big[ \, f \, \big] (t, \xi) \dd{t}
\fullstop{,}
\end{gather}
for any $x_0 \in U$ where the path of integration is assumed to lie entirely in $U$.

%===============================================================================
%===============================================================================
\subsection{Borel Resummation}
%===============================================================================
%===============================================================================

One of the most fundamental theorems in Gevrey asymptotics is a theorem of Nevanlinna \cite[pp.44-45]{nevanlinna1918theorie} which was rediscovered and clarified decades later by Sokal \cite{MR558468}; see also \cite[p.182]{zbMATH00797135}, and \cite[Theorem 5.3.9]{MR3495546}.
It identifies a subclass of Gevrey formal power series with a class of sectorial germs admitting Gevrey asymptotics in a halfplane.
This identification is often called \textit{Borel resummation} and it should be thought of in complete analogy with the familiar identification of convergent power series with germs of holomorphic functions.
The exceptions are that the identification for Gevrey series depends on a direction $\theta$ and the operation of converting a formal power serious into a holomorphic function is much more involved.
For clarity, we also provide definitions and statements in the single-variable case.

%===============================================================================
\begin{defn}{210617122738}
A Gevrey power series $\hat{f} (\hbar) \in \Gevrey \bbrac{\hbar}$ is a \dfn{Borel summable series} in the direction $\theta$ if its formal Borel transform $\hat{\phi} (\xi) \in \Complex \set{\xi}$ admits an analytic continuation $\phi (\xi) \coleq \rm{AnCont}_\theta [\, \hat{\phi} \,] (\xi)$ to a tubular neighbourhood $\Xi_\theta$ of the ray $e^{i\theta} \Real_+$ with at-most-exponential growth as $|\xi| \to + \infty$ in $\Xi_\theta$.
The subring of Borel summable series in the direction $\theta$ will be denoted by $\bar{\Gevrey}_\theta \bbrac{\hbar} \subset \Gevrey \bbrac{\hbar}$.

More generally in the parametric situation, a uniformly Gevrey power series $\hat{f} (x, \hbar) \in \cal{G}_\rm{u} (U) \bbrac{\hbar}$ is called a \dfn{uniformly Borel summable series} in the direction $\theta$ if its (necessarily uniformly convergent) formal Borel transform $\hat{\phi} (x, \xi) = \hat{\Borel}_\theta [\, \hat{f} \,] \in \cal{O}_\rm{u} (U) \set{\xi}$ admits an analytic continuation $	\phi (x, \xi) \coleq \rm{AnCont}_\theta [\, \hat{\phi} \,] (x, \xi)$ to a domain $U \times \Xi_\theta$ for some tubular neighbourhood $\Xi_\theta$ of the ray $e^{i\theta} \Real_+$ with uniformly at-most-exponential growth as $|\xi| \to + \infty$ in $\Xi_\theta$.
\dfn{Locally uniformly Borel summable series} on $U$ are defined the same way by allowing the thickness of the tubular neighbourhood $\Xi_\theta$ to have a locally constant dependence on $x$.
We denote the subalgebras of uniformly and locally uniformly Borel summable series in the direction $\theta$ by $\bar{\cal{G}}_{\theta, \rm{u}} (U) \bbrac{\hbar} \subset \cal{G}_\rm{u} (U) \bbrac{\hbar}$ and $\bar{\cal{G}}_\theta (U) \bbrac{\hbar} \subset \cal{G} (U) \bbrac{\hbar}$, respectively.
\end{defn}

%===============================================================================
\begin{thm}[\textbf{Nevanlinna's Theorem}]{210617120300}
In both the one-dimensional and the parametric cases, for any direction $\theta$, the asymptotic expansion map $\op{\ae}$ along the halfplane arc $A_\theta = (\theta - \pi/2, \theta + \pi/2)$ bisected by $\theta$ restricts to an algebra isomorphism
\eqntag{
	\op{\ae}: \cal{G} (\bar{A}_\theta) \iso \bar{\Gevrey}_{\theta} \bbrac{\hbar}
\qtext{or}
	\op{\ae}: \cal{G} (U; \bar{A}_\theta) \iso \bar{\cal{G}}_{\theta} (U) \bbrac{\hbar}
\fullstop
}
\end{thm}

For a detailed proof, see \cite[Appendix \S B.4]{MY2008.06492}.
Note that $\op{\ae}$ restricts further to an isomorphism $\cal{G}_\rm{u} (U; \bar{A}_\theta) \iso \bar{\cal{G}}_{\theta, \rm{u}} (U) \bbrac{\hbar}$.

%===============================================================================
\begin{defn}{210617132122}
The inverse algebra isomorphism
\eqntag{
	\cal{S}_\theta \coleq \op{\ae}^{-1} : \bar{\Gevrey}_{\theta} \bbrac{\hbar} \iso \cal{G} (\bar{A}_\theta)
\qtext{or}
	\cal{S}_\theta \coleq \op{\ae}^{-1} : \bar{\cal{G}}_{\theta} (U) \bbrac{\hbar} \iso \cal{G} (U; \bar{A}_\theta)
}
is called the \dfn{Borel resummation} in the direction $\theta$.
In particular, $\cal{S}_\theta$ restricts to be the identity map on the subalgebra of convergent power series: i.e., if $\hat{f} (\hbar) \in \Complex \set{\hbar}$ or $\hat{f} (x, \hbar) \in \cal{O} (U) \set{\hbar}$, then $\cal{S}_\theta (\hat{f}) = \hat{f}$.
\end{defn}

%===============================================================================
\paragraph{Properties of Borel resummation.}
As an algebra isomorphism, Borel resummation respects sum, products, and scalar multiplication: if $\hat{f}, \hat{g} \in \bar{\Gevrey}_{\theta} \bbrac{\hbar}$ (or $\hat{f}, \hat{g} \in \bar{\cal{G}}_{\theta} \bbrac{\hbar}$) are any two (respectively locally uniformly) Borel summable series in the direction $\theta$, and $c \in \Complex$ (or respectively $c = c(x) \in \cal{O} (U)$), then
\eqntag{
	\cal{S}_\theta [\, \hat{f} + c\hat{g} \,]
		= \cal{S}_\theta [\, \hat{f} \,] + c\cal{S}_\theta [\, \hat{g} \,]
\qtext{and}
	\cal{S}_\theta [\, \hat{f} \cdot \hat{g} \,]
		= \cal{S}_\theta [\, \hat{f} \,] \cdot \cal{S}_\theta [\, \hat{g} \,]
\fullstop
}
It is also compatible with composition with convergent power series.
In particular, the identity that we use most often in this paper is
\eqntag{
	\cal{S}_\theta [\, \exp \big( \, \hat{f} \, \big) \,]
		= \exp \Big( \cal{S}_\theta [\, \hat{f} \,] \Big)
\fullstop
}
We also note two more properties with respect to calculus in the $x$-variable.
If $\hat{f} (x, \hbar) \in \bar{\cal{G}}_{\theta} (U) \bbrac{\hbar}$ is a locally uniformly Borel summable series in the direction $\theta$, then
\eqntag{
	\del_x \cal{S}_\theta [\, \hat{f} \,]
		= \cal{S}_\theta [\, \del_x \hat{f} \,]
\qtext{and}
	\int\nolimits_{x_0}^x \cal{S}_\theta [\, \hat{f} \,]
		= \cal{S}_\theta \left[\, \int\nolimits_{x_0}^x \hat{f} \,\right]
\fullstop{,}
}
for any $x_0 \in U$ where the path of integration is assumed to lie entirely in $U$.

From the proof of Nevanlinna's Theorem \ref*{210617120300} we can extract the following more explicit statement that yields a formula for the Borel resummation which is actually often taken as the definition of Borel resummation.

%===============================================================================
\begin{lem}[\textbf{Borel-Laplace identity}]{210617095806}
Let $f \in \cal{G} (\bar{A}_\theta)$ and $\hat{f} \in \bar{\Gevrey}_{\theta} \bbrac{\hbar}$ be such that $\op{\ae} (f) = \hat{f}$ and $\cal{S}_\theta [\, \hat{f} \,] = f$.
Let $\hat{\phi} (\xi) \coleq \hat{\Borel} [\, \hat{f} \,] (\xi) \in \Complex \set{\xi}$ be the (necessarily convergent) formal Borel transform of $\hat{f}$, and let $\phi (\xi) \coleq \rm{AnCont}_\theta [\, \hat{\phi} \,] (\xi)$ be its analytic continuation to a tubular neighbourhood $\Xi_\theta$ of the ray $e^{i\theta} \Real_+$ with at-most-exponential growth as $|\xi| \to + \infty$ in $\Xi_\theta$.
Then we have the following two identities:
\eqntag{
	f (\hbar) = f_0 + \Laplace_\theta [\, \phi \,] (\hbar)
\qtext{and}
	\phi (\xi) = \Borel_\theta [\, f \,] (\xi)
\fullstop
}
More generally in the parametric situation, let $f \in \cal{G}_\rm{u} (U; \bar{A}_\theta)$ and $\hat{f} \in \bar{\cal{G}}_{\theta, \rm{u}} (U; \bar{A}_\theta)$ be such that $\op{\ae} (f) = \hat{f}$ and $\cal{S}_\theta [\, \hat{f} \,] = f$.
Let $\hat{\phi} (x, \xi) \coleq \hat{\Borel} [\, \hat{f} \,] (\xi) \in \cal{O}_\rm{u} (U) \set{\xi}$ be the (necessarily uniformly convergent) formal Borel transform of $\hat{f}$, and let $\phi (x, \xi) \coleq \rm{AnCont}_\theta [\, \hat{\phi} \,] (x,\xi)$ be its analytic continuation to the domain $U \times \Xi_\theta$, where $\Xi_\theta$ in a tubular neighbourhood of the ray $e^{i\theta} \Real_+$, with uniformly at-most-exponential growth as $|\xi| \to + \infty$ in $\Xi_\theta$.
Then we have the following two identities:
\eqntag{\label{210617162329}
	f (x, \hbar) = f_0 (x) + \Laplace_\theta [\, \phi \,] (x, \hbar)
\qtext{and}
	\phi (x, \xi) = \Borel_\theta [\, f \,] (x, \xi)
\fullstop
}
In other words, in both situations we have the following formula for the Borel resummation:
\eqntag{\label{210617162325}
	f
		= \cal{S}_\theta [\, \hat{f} \,]
		= f_0 + \Laplace_\theta \Big[ \, \op{AnCont}_\theta \big[ \, \hat{\Borel}_\theta [\, \hat{f} \,] \, \big] \, \Big]
\fullstop
}
\end{lem}

%===============================================================================
\paragraph{Borel resummation of exponential series.}
Finally, we extend Borel resummation to exponential power series by simply factorising out the exponential prefactor.
Thus, an exponential series $\hat{\psi} = e^{\Phi / \hbar} \hat{\Psi} \in \Complex^{\exp} \bbrac{\hbar}$ is a \dfn{Borel summable exponential series} in the direction $\theta$ if $\hat{\Psi}$ is Borel summable in the direction $\theta$, and we put
\eqntag{
	\cal{S}_\theta [\, \hat{\psi} \,]
		\coleq e^{\Phi / \hbar} \cal{S}_\theta [\, \hat{\Psi} \,]
\fullstop
}
More generally in the parametric case, an exponential power series $\hat{\psi} = e^{\Phi / \hbar} \hat{\Psi} \in \cal{O}^{\exp} (U) \bbrac{\hbar}$ is a \dfn{Borel summable exponential series} in the direction $\theta$ if $\hat{\Psi}$ is Borel summable in the direction $\theta$, and we again put $\cal{S}_\theta [\, \hat{\psi} \,] \coleq e^{\Phi / \hbar} \cal{S}_\theta [\, \hat{\Psi} \,]$.
These definition are then extended to all exponential transseries by linearity of $\cal{S}_\theta$.

%===============================================================================
%===============================================================================
%===============================================================================
\section{Proofs}
\setcounter{section}{3}
\setcounter{paragraph}{0}
\label{210525081835}
%===============================================================================
%===============================================================================
%===============================================================================

In this appendix, we collect the longer proofs from the main body of the text.

%===============================================================================
%===============================================================================
\subsection[\autoref*{210118113644}: Existence and Uniqueness of Formal WKB Solutions]{\autoref{210118113644}:\,Existence and Uniqueness of Formal WKB Solutions}
\label{210525083336}
%===============================================================================
%===============================================================================

\begin{proof}
We prove existence, uniqueness, and the property of forming the basis.
All other claims are established in the course of the proof.

%==============================
\textsc{1. Existence.}
To prove existence, we search for exponential power series solutions in the form of the formal WKB ansatz:
\eqntag{\label{210302174003}
	\hat{\psi} (x, \hbar) 
		= \exp \left( - \frac{1}{\hbar} \int\nolimits_{x_0}^x \hat{s} (t, \hbar) \dd{t} \right)
\fullstop{,}
}
where $\hat{s}$ is the unknown power series in $\cal{O} (U) \bbrac{\hbar}$ that we solve for.
Note that $\hat{\psi} (x_0, \hbar) = 1$ for any $\hat{s}$.
Substituting this expression back into the differential equation \eqref{210302151519}, we find that $\hat{s}$ must satisfy the formal Riccati equation \eqref{210115165557}.
Expanding \eqref{210115165557} term by term in $\hbar$ and using the fact that $U$ contains no turning points, we obtain the two exponential power series solutions \eqref{210115215219}.

%==============================
\textsc{2. Uniqueness.}
Let $\hat{\psi} \coleq e^{- \Phi/\hbar} \hat{\Psi}$, with $\Phi \in \cal{O} (U) [\hbar^{-1}]$ and $\hat{\Psi} \in \cal{O} (U) \bbrac{\hbar}$, be any exponential power series solution with normalisation $\hat{\psi} (x_0, \hbar) = 1$.
Let $m \geq 0$ be the degree in $\hbar^{-1}$ of $\Phi$ and write $\Phi = \Phi_0 + \ldots + \Phi_{-m} \hbar^{-m}$.
First, the normalisation condition forces $\Phi (x_0, \hbar^{-1}) = 0$ and $\hat{\Psi} (x_0, \hbar) = 1$.
In particular, it means that $\Phi_{-i} (x_0) = 0$ for all $i$, $\Psi_k (x_0) = 0$ for all $k \geq 1$, and $\Psi_0 (x_0) = 1$.

Next, plugging $\hat{\psi}$ into \eqref{210302151519} and eliminating the exponential prefactor $e^{-\Phi/\hbar}$, we obtain an equation in formal Laurent $\hbar$-series:
\eqntag{\label{210223145210}
	\hbar^2 \del_x^2 \hat{\Psi}
	+ \Big( -2 (\del_x \Phi) + \hat{p} \Big) \hbar \del_x \hat{\Psi}
	+ \Big( (\del_x \Phi)^2 - \hbar (\del_x^2 \Phi) - \hat{p} (\del_x \Phi) + \hat{q} \Big) \hat{\Psi}
	= 0
\fullstop
}
Suppose first that $m \geq 1$.
Then the lowest order in $\hbar$ of equation \eqref{210223145210} is at $\hbar^{-2m}$, which is $(\del_x \Phi_{-m})^2 \Psi_0 = 0$.
Therefore, $\Phi_{-m}$ is identically zero, so the degree of $\Phi$ is actually $m-1$.
Continuing in this fashion, we conclude that $\Phi_0$ is the only possible nonzero coefficient of $\Phi$.
Therefore, the leading-order part of \eqref{210223145210} is at $\hbar^0$, which is
\eqn{
	\big( \del_x \Phi_0 \big)^2 - p_0 \big( \del_x \Phi_0 \big) + q_0 = 0
\fullstop
}
This is nothing but the leading-order characteristic equation \eqref{210415145506} for $s = \del_x \Phi_0$.
It follows that $\del_x \Phi_0$ is either $\lambda_+$ or $\lambda_-$, and the normalisation condition $\Phi (x_0, \hbar^{-1}) = 0$ forces $\Phi_0$ to be equal to either $\Phi_+$ or $\Phi_-$ from \eqref{210223093541}.

Next, at order $\hbar^1$, after substituting $\lambda_\alpha$ for $\alpha \in \set{+,-}$ in place of $\del_x \Phi_0$, equation \eqref{210223145210} reduces to
\eqn{
	-(2\lambda_\alpha - p_0) \del_x \Psi_0 
	+ \big( \del_x \lambda_\alpha - p_1 \lambda_\alpha + q_1 \big) \Psi_0
	+ \big( \lambda_\alpha^2 - p_0 \lambda_\alpha + q_0 \big) \Psi_1
	= 0
\fullstop
}
The factor in front of $\del_x \Psi_0$ is $-\varepsilon_\alpha \sqrt{\DD_0}$ where $\varepsilon_\pm \coleq \pm 1$.
The factor in front of $\Psi_1$ is zero, so after comparing with \eqref{210223183123}, we get:
\eqn{
	\del_x \log \Psi_0 = \varepsilon_\alpha \frac{1}{\sqrt{\DD_0}} 
			\Big( \del_x \lambda_\alpha - p_1 \lambda_\alpha + q_1 \Big)
			= s_\alpha^\pto{1}
\fullstop
}
Then the normalisation condition $\Psi_0 (x_0) = 1$ implies $\Psi_0 = \exp \int_{x_0}^x s_\alpha^\pto{1} (t) \dd{t}$.

In particular, $\Psi_0$ is nonvanishing on $U$, so we can write $\hat{\Psi} = \exp \hat{\RR}$ for some power series $\hat{\RR} \in \cal{O} (U) \bbrac{\hbar}$ which satisfies $\RR_k (x_0) = 0$ for all $k \geq 1$.
Then the same calculation shows that $\del_x \RR_k = s_\alpha^\pto{k}$ for each $k \geq 1$, and hence $\RR_k (x) = \int_{x_0}^x s_\alpha^\pto{k} (t) \dd{t}$.
Therefore, $\hat{\RR} = \int_{x_0}^x \hat{\SS}_\alpha (t) \dd{t}$, demonstrating uniqueness of formal WKB solutions amongst all possible exponential power series solutions.

%==============================
\textsc{3. Basis.}
To see that the formal WKB solutions $\hat{\psi}_+, \hat{\psi}_-$ form a basis of generators for the $\Complex^{\exp} \bbrac{\hbar}$-module $\hat{\mathbb{ES}} (U)$ of all formal solutions, we note first that they are linearly independent over $\Complex^{\exp} \bbrac{\hbar}$ because their exponents $\Phi_+, \Phi_-$ are distinct nonconstant functions on $U$.
Now, suppose
\eqn{
	\hat{\psi}
	= \sum_{\alpha}^{\text{\tiny{finite}}} \hat{\CC}_\alpha e^{-\Phi_\alpha / \hbar} \hat{\Psi}_\alpha 
	\in \hat{\mathbb{S}} (U)
}
is a formal solution of \eqref{210302151519} with all $\hat{\CC}_\alpha = \hat{\CC}_\alpha (\hbar) \in \Complex^{\exp} \bbrac{\hbar}$ nonzero, and all exponents $\Phi_\alpha \in \cal{O} (U) [ \hbar^{-1} ]$ mutually linearly independent over $\Complex [\hbar^{-1}]$.
We need to show that there are $\hat{\CC}_+, \hat{\CC}_- \in \Complex^{\exp} \bbrac{\hbar}$ such that $\hat{\psi} = \hat{\CC}_+ \hat{\psi}_+ + \hat{\CC}_- \hat{\psi}_-$.

We can immediately assume that the leading-order $\Psi_\alpha^\pto{0}$ of every power series $\hat{\Psi}_\alpha$ is nonzero by absorbing any excess powers of $\hbar$ into the coefficient $\hat{\CC}_\alpha$.
Then, upon substituting $\hat{\psi}$ into \eqref{210302151519}, the same calculation that led to \eqref{210223145210} yields
{\small
\eqn{
	\sum_{\alpha}^{\text{\tiny{finite}}} \hat{\CC}_\alpha
	\Bigg[ \hbar^2 \del_x^2 \hat{\Psi}_\alpha
	+ \Big( - 2 (\del_x \Phi_\alpha) + \hat{p} \Big) \hbar \del_x \hat{\Psi}_\alpha
	+ \Big( (\del_x \Phi_\alpha)^2 - \hbar (\del_x^2 \Phi_\alpha) - \hat{p} (\del_x \Phi_\alpha) + \hat{q} \Big) \hat{\Psi}_\alpha
	\Bigg] e^{-\Phi_\alpha / \hbar}
	= 0
\fullstop
}}%
This is a sum of exponential power series with linearly independent exponents, so for it to be zero, each Laurent series in large square brackets must be zero.
Let $\Phi^\pto{-m}_\alpha$ be the highest in $\hbar^{-1}$ degree part of $\Phi_\alpha$.
As before, if $m \geq 1$, the vanishing of the expression in the bracket leads to the equation $(\del_x \Phi^\pto{-m}_\alpha)^2 = 0$ forcing $\Phi^\pto{-m}_\alpha$ to be a constant.
As a result, each factor $e^{-\Phi^\pto{-m}_\alpha \hbar^{-m-1}}$ can be absorbed into the coefficient $\hat{\CC}_\alpha$.
Continuing in this fashion to remove such exponential factors one by one from $e^{-\Phi_\alpha/\hbar}$, we are left with only the zeroth-order term $\Phi^\pto{0}_\alpha (x)$, which must satisfy the equation
\eqn{
	\big( \del_x \Phi^\pto{0}_\alpha \big)^2 - p_0 \big( \del_x \Phi^\pto{0}_\alpha \big) + q_0 = 0
\fullstop
}
This is nothing but the leading-order equation \eqref{210415145506} for $\lambda = \del_x \Phi^\pto{0}_\alpha$.
Therefore, $\del_x \Phi^\pto{0}_\alpha$ can only either be the leading-order solution $\lambda_+$ or $\lambda_-$, and so $\Phi^\pto{0}_\alpha$ can only differ from either $\Phi_+$ or $\Phi_-$ by a constant which can once again be absorbed into $\hat{\CC}_\alpha$.
Thus, we conclude that each polynomial $\Phi_\alpha$ is either $\Phi_+$ or $\Phi_-$, and therefore the formal solution $\hat{\psi}$ is actually of the form $\hat{\psi} = \hat{\CC}_+ e^{-\Phi_+ / \hbar} \hat{\Psi}_+ + \hat{\CC}_- e^{-\Phi_- / \hbar} \hat{\Psi}_-$ for some power series $\hat{\Psi}_\pm$, which by the exact same calculation as above must be $\exp \int_{x_0}^x \hat{\SS}_\pm \dd{t}$.
\end{proof}

%===============================================================================
%===============================================================================
\subsection{Existence and Uniqueness of Exact Characteristic Solutions}
\label{210525082957}
%===============================================================================
%===============================================================================

The strategy to construction an exact solution of the Riccati equation \eqref{210304143646} is to restrict the problem to a horizontal halfstrip domain containing $x_0$ and transform the equation into standard form which is easier to solve using the Borel-Laplace method.
First, we describe this standard form of the Riccati equation and explain how to solve it.

%===============================================================================
\paragraph{Singularly perturbed Riccati equations in standard form.}
Fix $\epsilon, \delta > 0$.
Let $\Omega_+ \subset \Complex_z$ and $S_+ \subset \Complex_\hbar$ be respectively a tubular neighbourhood of radius $\epsilon$ of the positive real axis $\Real_+ \subset \Complex_z$ and a Borel disc of diameter $\delta$ bisected by the positive real axis:
\eqn{
	\Omega_+ \coleq \set{z ~\big|~ \op{dist} (z, \Real_+) < \epsilon}
\qtext{and}
	S_+ \coleq \set{ \hbar ~\big|~ \Re (\hbar^{-1}) > 1/\delta}
\fullstop
}

Consider the following singularly perturbed Riccati equation on $\Omega_+ \times S_+$:
\eqntag{
\label{210418182218}
	\hbar \del_z f - f = \hbar \big( \AA_0 + \AA_1 f + f^2\big)
\fullstop{,}
}
where $\AA_1, \AA_0$ are holomorphic functions of $(z, \hbar) \in \Omega_+ \times S_+$ which admit uniform Gevrey asymptotics:
\eqn{
	\AA_i (z, \hbar) \simeq \hat{\AA}_i (z, \hbar)
\quad
\text{as $\hbar \to 0$ along $[-\pi/2, +\pi/2]$, unif. $\forall z \in \Omega_+$\fullstop}
}

\begin{lem}[\textbf{Main Technical Lemma}{ \cite[Lemma 5.6]{MY2008.06492}}]{210421183535}
The Riccati equation \eqref{210418182218} has a unique formal solution $\hat{f} \in \cal{O}(\Omega_+) \bbrac{\hbar}$ and its leading-order term is $0$; i.e., $\hat{f} \in \hbar \cal{O}(\Omega_+) \bbrac{\hbar}$.
Moreover, for every $\epsilon' \in (0, \epsilon)$, there is some $\delta' \in (0, \delta]$ such that the Riccati equation \eqref{210418182218} has a unique holomorphic solution $f$ on
\eqntag{\label{210522170909}
	\Omega'_+ \times S'_+
	\coleq 
	\set{z ~\big|~ \op{dist} (z, \Real_+) < \epsilon'} 
		\times \set{ \hbar ~\big|~ \Re (\hbar^{-1}) > 1/\delta'}
	\subset
	\Omega_+ \times S_+
\fullstop{,}
}
which admits $\hat{f}$ as a uniform Gevrey asymptotic expansion in the right halfplane: %as $\hbar \to 0$ along $[-\pi/2, +\pi/2]$.
\eqntag{\label{210522173054}
	f (z, \hbar) \simeq \hat{f} (z, \hbar)
\quad
\text{as $\hbar \to 0$ along $[-\pi/2, +\pi/2]$, unif. $\forall z \in \Omega'_+$\fullstop}
}
\end{lem}

A proof is presented in \cite{MY2008.06492}, but since it does most of the heavy-lifting in the present paper, we include here a sketch for completeness.

\begin{proof}[Proof sketch.]
Existence and uniqueness of the formal solution $\hat{f}$ is easy to deduce by expanding the Riccati equation \eqref{210418182218} in powers of $\hbar$ (see \cite[Theorem 3.8]{MY2008.06492} for details).
We now prove the existence and uniqueness of $f$.

%==============================
\textsc{1. The analytic Borel transform.}
Let $\epsilon' \in (0, \epsilon)$ be fixed.
The Riccati equation \eqref{210418182218} is solved using the Borel-Laplace method.
The first step is to apply the analytic Borel transform:
\eqntag{
	\alpha_i (z, \xi) \coleq \Borel_+ [ \: \AA_i \: ] (z, \xi)
\fullstop
}
Here, $\Borel_+$ is the Borel transform \eqref{210617101748} for $\theta = 0$.

Since each $\AA_i$ admits uniform Gevrey asymptotics on $\Omega_+$ as $\hbar \to 0$ along the closed arc $[-\pi/2, +\pi/2]$, it follows from Nevanlinna's \autoref{210617120300} that there is some tubular neighbourhood $\Xi_+ \coleq \set{ \xi ~\big|~ \op{dist} (\xi, \Real_+) < \rho}$ of radius $\rho > 0$ (which we take to be so small that $\epsilon' + \rho < \epsilon$) such that $\alpha_1, \alpha_2$ define holomorphic functions on $\Omega_+ \times \Xi_+ \subset \Complex^2_{z\xi}$ with uniformly at most exponential growth as $|\xi| \to + \infty$.
They furthermore satisfy the following relations: for all $(z,\xi) \in \Omega_+ \times \Xi_+$,
\eqntag{
	\AA_0 (z, \hbar) = a_0 (z) + \Laplace_+ [ \: \alpha_0 \: ]
\qtext{and}
	\AA_1 (z, \hbar) = a_1 (z) + \Laplace_+ [ \: \alpha_1 \: ]
\fullstop
}
Here, $\Laplace_+$ is the Laplace transform \eqref{200624181217} for $\theta = 0$.
Thus, dividing \eqref{210418182218} through by $\hbar$ and applying the analytic Borel transform $\Borel_+$, we obtain the following PDE with convolution product:
\eqntag{\label{210519160413}
	\del_z \phi - \del_\xi \phi 
		= \alpha_0
			+ a_1 \phi
			+ \alpha_1 \ast \phi
			+ \phi \ast \phi
\fullstop{,}
}
where the unknown variables $f$ and $\phi$ are related by
\eqntag{\label{210522174710}
	\phi = \Borel_+ [\: f \:]
\qtext{and}
	f = \Laplace_+ [\: \phi \:]
\fullstop
}

%==============================
\textsc{2. Solving the PDE.}
To solve the PDE \eqref{210519160413}, we consider another holomorphic change of variables sending $(z, \xi) \mapsto (z+\xi, \xi) \eqcol (w,t)$ which transforms the differential operator $\del_z - \del_\xi$ into $-\del_t$.
Then \eqref{210519160413} can be written equivalently as the following integral equation:
\eqntag{
\label{210519162053}
	\phi (z, \xi) = a_0 (z) - \int\nolimits_0^\xi \evat{\Big( \alpha_0
			+ a_1 \phi
			+ \alpha_1 \ast \phi
			+ \phi \ast \phi \Big)}{(z + \xi - u, u)} \dd{u}
\fullstop
}
Introduce the following shorthand notation: for any function $\alpha = \alpha (z, \xi)$, let
\eqntag{
\label{210611155612}
	\II_+ [\: \alpha \:] (z, \xi) 
		\coleq - \int\nolimits_0^\xi \alpha (z + \xi - u, u) \dd{u}
		= - \int\nolimits_0^\xi \alpha (z + t, \xi - t) \dd{t}
\fullstop
}
In this notation, the integral equation \eqref{210519162053} becomes
\eqntag{
\label{210519162121}
	\phi = a_0 + \II_+ \big[ \: \alpha_0
			+ a_1 \phi
			+ \alpha_1 \ast \phi
			+ \phi \ast \phi
			\: \big]
\fullstop
}

%==============================
\textsc{3. Method of successive approximations.}
To solve \eqref{210519162121}, we use the method of successive approximations.
Consider a sequence of holomorphic functions $\set{\phi_n}_{n=0}^\infty$ defined recursively by $\phi_0 \coleq a_0$, $\phi_1 \coleq \II_+ \big[ \alpha_0 + a_1 \phi_0 \big]$, and for $n \geq 2$ by
\eqntag{\label{210522160020}
	\phi_n 
		\coleq \II_+ \left[ 
			a_1 \phi_{n-1} 
			+ \alpha_1 \ast \phi_{n-2} 
			+ \sum_{\substack{i,j \geq 0 \\ i + j = n-2}}
				\phi_i \ast \phi_j
			\right]
\fullstop
}
Notice that each $\phi_n$ defines a holomorphic function on the domain
\eqn{
	W \coleq \set{ (z, \xi) \in \Omega_+ \times \Xi_+ ~\big|~ z + \xi \in \Omega_+}
\fullstop
}
Notice furthermore that $\Omega'_+ \times \Xi_+ \subset W$.
The most technical part of the proof is to show that the infinite series 
\eqntag{\label{210522160040}
	\phi (z, \xi) \coleq \sum_{n=0}^\infty \phi_n (z, \xi)
}
is uniformly convergent for all $(z, \xi) \in W$ to a holomorphic solution of the integral equation \eqref{210519162121} on $W$ with uniformly at-most-exponential growth at infinity in $\xi$.

Assuming that $\phi$ is uniformly convergent on $W$, it is a straightforward computation to show that $\phi$ indeed satisfies \eqref{210519162121}.
The bulk of the proof is therefore to demonstrate uniform convergence of \eqref{210522160040}.

%==============================
\textsc{4. Uniform convergence.}
Let $\MM, \LL > 0$ be constants such that $|a_i| \leq \MM$ and $|\alpha_i| \leq \MM e^{\LL |\xi|}$ for all ${(z, \xi) \in W}$.
These constants exist because $W \subset \Omega_+ \times \Xi_+$.
We need to show that there are constants $\AA, \BB > 0$ such that, for all $n \geq 0$,
\eqntag{\label{200702211322}
	\llap{$\big(\forall (z, \xi) \in W \big)$\qqqqquad}
	\big| \phi_n (z, \xi) \big| \leq \AA \BB^n \frac{|\xi|^n}{n!} e^{\LL |\xi|}
\fullstop
}
To show this, we first recursively construct a sequence of positive real numbers $(\MM_n)_{n=0}^\infty$ such that, for all $n \geq 0$,
\eqntag{\label{200702211319}
	\llap{$\big(\forall (z, \xi) \in W \big)$\qqqqquad}
	\big| \phi_n (z, \xi) \big| \leq \MM_n \frac{|\xi|^n}{n!} e^{\LL |\xi|}
\fullstop	
}
We then show that there are constants $\AA, \BB > 0$ such that $\MM_n \leq \AA \BB^n$ for all $n$.

Since $\phi_0 = a_0$, the constant $\MM_0$ can be taken to be $1$.
The constants $\MM_n$ can then be constructed by induction with the help some elementary integral estimates.
In the end, our constants $\MM_n$ are given by the following recursive formula:
\eqntag{\label{200702211315}
	\MM_n \coleq 
			\MM 
			\bigg( \MM_{n-1} + \MM_{n-2}
				+ \sum_{\substack{i,j \geq 0 \\ i + j = n-2}} 
					\MM_i \MM_j
			\bigg)
\fullstop
}
To obtain constants $\AA, \BB$ for the bounds \eqref{200702211322}, we consider the following power series in an abstract variable $t$:
\eqn{
	\hat{p} (t) \coleq \sum_{n=0}^\infty \MM_n t^n \in \Complex \bbrac{t}
\fullstop
}
It is enough to prove that $\hat{p} (t)$ is convergent.
The key is to observe that it satisfies the algebraic equation $\hat{p} = \Big( 1 + \hat{p} t + \hat{p} t^2 + \hat{p}^2 t^2 \Big)$.
Then we consider the holomorphic function $\FF = \FF (p,t) \coleq - p + \Big( 1 + pt + pt^2 + p^2t^2 \Big)$ of two variables.
It satisfies $\FF (1,0) = 0$ and has a nonvanishing derivative with respect to $p$ at $(p,t) = (1,0)$, so we can use the Holomorphic Implicit Function Theorem to conclude that there is a function $p (t)$, holomorphic at $t = 0$, whose Taylor series is $\hat{p} (t)$.

Finally, uniform convergence of $\phi$ follows from the following calculation, which simultaneously shows that $\phi$ has at most exponential growth at infinity in $\xi$:
\eqntag{\label{200814082657}
	\big| \phi (z, \xi) \big|
		\leq \sum_{n=0}^\infty |\phi_n|
		\leq \sum_{n=0}^\infty \AA \BB^n \frac{|\xi|^n}{n!} e^{\LL |\xi|}
		\leq \AA e^{ (\BB + \LL) |\xi|}
\fullstop
}

%==============================
\textsc{5. The Laplace transform.}
Applying the Laplace transform to $\phi$ yields
\eqntag{\label{200702212531}
	f (z, \hbar) 
		\coleq \Laplace \big[ \phi \big] (z, \hbar)
		= \int_{0}^{+\infty} e^{- \xi / \hbar} \phi (z, \xi) \dd{\xi}
\fullstop
}
Thus, whenever $\delta' < (\BB + \LL)^{-1}$ and $\leq \delta$, this integral converges uniformly on the Borel disc $S'_+ \coleq \set{ \hbar ~\big|~ \Re (\hbar^{-1}) > 1/\delta'}$, yielding therefore a holomorphic solution $f$ of the Riccati equation.
Moreover, Nevanlinna's Theorem implies that $f$ admits uniform Gevrey asymptotics on $\Omega'_+$ as $\hbar \to 0$ along $[-\pi/2, +\pi/2]$.

%==============================
\textsc{6. Uniqueness.}
Uniqueness of $f$ also follows from Nevanlinna's Theorem, for if $f'$ is another such solution, then $f - f'$ is a holomorphic function on $\Omega'_+ \times S'_+$ which is uniformly Gevrey asymptotic to $0$ as $\hbar \to 0$ along $[-\pi/2, +\pi/2]$, hence must be identically zero.
\end{proof}

Collecting the various steps of the proof of \autoref{210421183535} and using Nevanlinna's Theorem yields immediately the following statements.

\begin{lem}{210522171904}
The solution $f$ from \autoref{210421183535} has the following properties.
%==============================
\begin{enumerate}
%==============================
\item [\textup{(P1)}]
The analytic Borel transform 
\eqntag{
	\phi (z, \xi) = \Borel_+ [\, f \,] (z, \xi)
		\coleq \frac{1}{2\pi i} \oint f (z, \hbar) e^{\xi / \hbar} \frac{\dd{\hbar}}{\hbar^2}
}
is uniformly convergent for all $(z, \xi) \in \Omega'_+ \times \Xi''_+$ where $\Xi''_+ \coleq \set{ \xi ~\big|~ \op{dist} (\xi, \Real_+) < \rho}$ is a tubular neighbourhood of any radius $\rho \in (0, \epsilon - \epsilon')$.
%==============================
\item [\textup{(P2)}]
The Laplace transform of $\phi$ is uniformly convergent for all $(z, \hbar) \in \Omega'_+ \times S'_+$ and
\eqntag{\label{210522161859}
	f (z, \hbar) 
		= \Laplace_+ \big[ \, \phi \, \big] (z, \hbar)
		= \int_0^{+\infty} e^{- \xi / \hbar} \phi (z, \xi) \dd{\xi}
\fullstop
}
%==============================
\item [\textup{(P3)}]
In other words, \textup{(P1)} and \textup{(P2)} together mean that $f$ is the uniform Borel resummation of its asymptotic power series $\hat{f}$: for all $(z, \hbar) \in \Omega'_+ \times S'_+$,
\eqntag{
	f (z, \hbar) = \cal{S}_+ \big[ \, \hat{f} \, \big] (z, \hbar)
\fullstop
}
%==============================
\item [\textup{(P4)}]
If the coefficients $\AA_0, \AA_1$ in \eqref{210418182218} are periodic in $z$ with period $\omega \in \Complex$, then $f$ is also periodic in $z$ with the same period $\omega$.
%==============================
\end{enumerate}
%==============================
\end{lem}

\begin{proof}
The only statement that requires an additional comment is (P4).
Referring to the proof of \autoref{210421183535}, we can see that since the analytic Borel transform of $\AA_i$ is done with respect to the variable $\hbar$, periodicity of $\AA_i$ in the variable $z$ implies the same periodicity of the coefficients $\alpha_i$ and $a_i$ in the differential equation \eqref{210519160413}.
Then the periodicity of the solution $\phi$ boils down to whether the integral operator $\II_+$ from \eqref{210611155612} preserves the periodic property, which it clearly does.
\end{proof}

%===============================================================================
%===============================================================================
\subsection[\autoref*{210116200501}: Existence and Uniqueness of Exact WKB Solutions]{\autoref{210116200501}: Existence and Uniqueness of Exact WKB Solutions}
\label{210525083210}
%===============================================================================
%===============================================================================

We prove both \autoref{210116200501} and \autoref{210603110237} simultaneously.

\begin{proof}
First, the uniqueness of $\psi_\alpha$ follows from the asymptotic property \eqref{210516153319}.
Indeed, if $\psi'_\alpha$ is another such solution on $U \times S'$, then they are both asymptotic to $\hat{\psi}_\alpha$, and hence their difference must be locally uniformly Gevrey asymptotic to $0$ as $\hbar \to 0$ along $\bar{A}$.
Then Nevanlinna's \autoref{210617120300} implies that $\psi_\alpha - \psi'_\alpha$ must be identically $0$.
Now we prove the existence of $\psi_\alpha$.

%==============================
\textsc{1. Boiling the proof down to a local construction.}
Given any $x_i \in U$, change the basepoint of the Liouville transformation $\Phi$ from $x_0$ to $x_i$ by defining $\Phi_i (x) \coleq \Phi (x) + \Phi (x_i)$.
Since $U$ is simply connected, this definition is unambiguous.
Take $\epsilon_i > 0$ sufficiently small that the WKB disc $V_i \coleq \Phi_i^{-1} \big( \set{|z| < \epsilon_i} \big)$ centred at $x_i$ satisfies all the hypotheses of the theorem (with ``$V$'' replaced by ``$V_i$'').
Choose any $\epsilon'_i \in (0, \epsilon_0)$ and let $V'_i \coleq \Phi_i^{-1} \big( \set{|z| < \epsilon'_i} \big)$.

The construction of the exact WKB solution $\psi_\alpha$ proceeds as follows.
First, we take $x_i = x_0$ and construct the unique solution $\psi_{\alpha, 0}$ on $V'_{0} \times S_0$ satisfying conditions \eqref{210516153316} and \eqref{210516153319}, where $S_0 \subset S_\theta$ is a possibly smaller Borel disc.
Then we take any $x_i \neq x_0$ sufficiently close to $x_0$ such that $x_0 \in V'_{i}$, and construct the unique solution $\psi_{\alpha,i}$ on $V'_{i} \times S_i$ (where $S_\alpha \subset S_\theta$ is a possibly smaller Borel disc) which satisfies condition \eqref{210516153319} as well as the normalisation condition $\psi_{\alpha,i} (x_i, \hbar) = \psi_{\alpha,0} (x_1, \hbar)$ replacing condition \eqref{210516153316}.
Take $S'$ to be the smaller of the two Borel discs $S_0, S_i$.
Then by uniqueness, $\psi_{\alpha,0}$ and $\psi_{\alpha,i}$ agree on the intersection $V'_{i} \cap V'_{0}$ and therefore define the unique extension of the solution $\psi_{\alpha,0}$ to $(V'_{0} \cup V'_{i}) \times S'$ satisfying conditions \eqref{210516153316} and \eqref{210516153319}.

Continue to extend the solution $\psi_{\alpha,0}$ in this fashion to larger and larger subsets of $U$.
We conclude that for any compactly contained $V' \subset U$, there is a Borel disc $S' \subset S$ of nonzero diameter $\delta' \in (0, \delta]$ such that $\psi_{\alpha,0}$ extends to a unique holomorphic solution $\psi_\alpha$ on $V' \times S'$ satisfying conditions \eqref{210516153316} and \eqref{210516153319}.
So let us now fix any such $V'$.
For example, $V' = V'_0 = \Phi^{-1} \big( \set{|z| < \epsilon_0} \big)$ from before.
Finally, using the usual Parametric Existence and Uniqueness Theorem for linear ODEs (see, e.g., \cite[Theorem 24.1]{MR0460820}), the solution $\psi_\alpha$ can be analytically continued to any simply connected domain in $X$, thus in particular defining the desired solution on $U \times S'$.

%==============================
\textsc{2. Local construction.}
Thus, it remains to construct the solution $\psi_{\alpha,i}$ on $V'_{i} \times S_i$ for any $i$.
The strategy is to construct the unique exact solution $s_{\alpha,i}$ of the Riccati equation \eqref{210304143646} on $V'_{i}$ with leading-order $\lambda_\alpha$, and then use formula \eqref{210218230927} to define $\psi_{\alpha,i}$.
From now on, we drop the label ``$i$'' because the derivation of $s_{\alpha,i}$ is verbatim the same for any $i$.
Thus, we consider the WKB discs $V \coleq \Phi^{-1} \big( \set{|z| < \epsilon} \big)$ and $V' \coleq \Phi^{-1} \big( \set{|z| < \epsilon'} \big)$, as well as a pair of nested horizontal halfstrips in the $z$-plane,
\eqn{
	\Omega'_+ \coleq \set{\op{dist} (z, \Real_+) < \epsilon'}
		~\subset~
	\Omega_+ \coleq \set{\op{dist} (z, \Real_+) < \epsilon}
\fullstop{,}
}
so that $V_{\theta,\alpha} \coleq \Phi^{-1} (\varepsilon_\alpha e^{i\theta}\Omega_+)$ and $V'_{\theta,\alpha} \coleq \Phi^{-1} (\varepsilon_\alpha e^{i\theta}\Omega'_+)$.
Let $S_+ \coleq e^{-i\theta} S_\theta = \set{ \Re (\hbar^{-1}) > 1/\delta}$.
Recall that $\Phi^{-1} : \varepsilon_\alpha e^{i\theta} \Omega_+ \to V_{\theta,\alpha}$ is a local biholomorphism.
If it is many-to-one, let $\omega > 0$ be the trajectory period; i.e., the smallest positive real number such that $\Phi^{-1} (\varepsilon_\alpha e^{i\theta}\omega) = \Phi^{-1} (0)$.

%==============================
\textsc{3. Transformation to the Riccati equation in standard form.}
Now we transform the Riccati equation \eqref{210304143646} into its standard form on $\Omega_+ \times S_+$ satisfying the hypotheses of \autoref{210421183535}.
We break this transformation down into a sequence of three steps.
It is convenient to define functions $p_\ast$ and $q_\ast$ by the following relations:
\eqntag{
	p = p_0 + p_1 \hbar + p_\ast \hbar^2
\qqtext{and}
	q = q_0 + q_1 \hbar + q_\ast \hbar^2
\fullstop
}
Thus, for example, the leading-order in $\hbar$ of $p_\ast, q_\ast$ is respectively $p_2, q_2$.

First, we change the unknown variable $s$ to $s_\ast$ given by the relation
\eqntag{
\label{210526065629}
	s = \lambda_\alpha + s_\alpha^\pto{1} \hbar + s_\ast \hbar^2
\fullstop
}
In other words, we want to derive an equation for the subleading-orders of the solution $s$ by removing the leading-order and the next-to-leading-order characteristic roots $\lambda_\alpha, s_\alpha^\pto{1}$.
Substituting these expressions into the Riccati equation \eqref{210304143646}, using the identity $\varepsilon_\alpha \sqrt{\DD_0} = 2\lambda_\alpha - p_0$ where $\varepsilon_\pm = \pm 1$, and eliminating the leading-order and the next-to-leading-order parts in $\hbar$ using identities \eqref{210415145506} and \eqref{210223183123}, we get:
\begin{multline*}
	\hbar^2 \del_x s_\ast - \varepsilon_\alpha \sqrt{\DD_0} \hbar s_\ast
\\	= \hbar \Big( \hbar^2 s_\ast^2 + (\hbar p_\ast + p_1 - 2s_\alpha^\pto{1}) \hbar s_\ast 
		+ \big( q_\ast - p_\ast (\lambda_\alpha + \hbar s_\alpha^\pto{1}) + (s_\alpha^\pto{1})^2 - \del_x s_\alpha^\pto{1} - p_1 s_\alpha^\pto{1}\big) \Big)
\fullstop
\end{multline*}
Next, we transform $s_\ast$ to $\TT$ via 
\eqntag{\label{210526070036}
	\hbar s_\ast = \varepsilon_\alpha \sqrt{\DD_0} \, \TT
\fullstop
}
Substituting and dividing through by $\DD_0$ leads to the following equation for $\TT$:
\eqntag{
	\varepsilon_\alpha \frac{\hbar}{\sqrt{\DD_0}} \del_x \TT
	- \TT
	= \hbar \Big( b_0 + \BB_0 + (b_1 + \BB_1) \TT + \TT^2 \Big)
\fullstop{,}
}
where
\eqntag{\label{210526124328}
\begin{aligned}
	\BB_0 &\coleq \frac{q_\ast - p_\ast (\lambda_\alpha + \hbar s_\alpha^\pto{1}) - q_2 + p_2 \lambda_\alpha}{\DD_0}
\qtext{and}
	\BB_1 \coleq \frac{\hbar p_\ast}{\varepsilon_\alpha \sqrt{\DD_0}},
\\
	b_0 &\coleq \frac{q_2 - p_2 \lambda_\alpha + (s_\alpha^\pto{1})^2 - \del_x s_\alpha^\pto{1} - p_1 s_\alpha^\pto{1}}{\DD_0}
\qtext{and}
	b_1 \coleq \frac{p_1 - 2s_\alpha^\pto{1} - \del_x \log \sqrt{\DD_0}}{\varepsilon_\alpha \sqrt{\DD_0}}.
\end{aligned}
}
Notice that $\BB_0$ and $\BB_1$ are both zero in the limit as $\hbar \to 0$, and furthermore an examination of \eqref{210525213123} reveals that
\eqn{
	b_0 = - \frac{\varepsilon_\alpha}{\sqrt{\DD_0}} s_\alpha^\pto{2}
\fullstop
}
Finally, we change the unknown variable $\TT = \TT (x, \hbar)$ to $f = f (z, \hbar)$ using the Liouville transformation $z (x) = \varepsilon_\alpha e^{-i \theta} \Phi (x)$ and a rotation in the $\hbar$-plane:
\eqntag{\label{210526072356}
	\TT (x, \hbar) = f \big( z(x) , e^{-i\theta}\hbar \big)
\fullstop
}
Then $f$ satisfies the Riccati equation in standard form \eqref{210418182218}, which is
\eqntag{
\label{210526124707}
	\hbar \del_z f - f = \hbar \big( \AA_0 + \AA_1 f + f^2\big)
\fullstop{,}
}
where the coefficients $\AA_1, \AA_2$ are given by 
\eqntag{
	\AA_k \big( z(x) , e^{-i\theta}\hbar \big)
		= b_k (x) + \BB_k (x, \hbar)
\fullstop
}
Combining all three transformations \eqref{210526065629}, \eqref{210526070036}, and \eqref{210526072356}, the total change of variables for $s$ to $f$ needed to take the original Riccati equation \eqref{210304143646} to the standard one \eqref{210418182218} is
\eqntag{\label{210526072953}
	s (x,\hbar) 
		= \lambda_\alpha (x) 
			+ \hbar \Big( s_\alpha^\pto{1} (x,\hbar) 
				+ \varepsilon_\alpha \sqrt{\DD_0 (x)} 
					f \big( z(x) , e^{-i\theta}\hbar \big) \Big)
\fullstop
}

%==============================
\textsc{4. Finish the construction by applying the Main Technical Lemma.}
By taking $\epsilon$ sufficiently small, we can assume that conditions (1) and (2) in \autoref{210116200501} are satisfied uniformly for all $x \in V_{\theta,\alpha}$.
Then it is clear from expressions \eqref{210526124328} that assumptions (1) and (2) in \autoref{210116200501} imply that $\AA_0, \AA_1$ admit Gevrey asymptotics as $\hbar \to 0$ along $[-\pi/2,+\pi/2]$ uniformly for all $z \in \Omega_+$.
Therefore, by \autoref{210421183535}, there is some $\delta' \in (0, \delta]$ such that the Riccati equation \eqref{210526124707} has a unique holomorphic solution $f$ on $\Omega'_+ \times S'_+$ with uniform Gevrey asymptotics $f (z, \hbar) \simeq \hat{f} (z, \hbar)$ as $\hbar \to 0$ along $[-\pi/2, +\pi/2]$.
Finally, note that if $\Phi$ is $\omega$-periodic for some $\omega \in \Complex$, the so are the coefficients $\AA_i$, and therefore so is the solution $f$ by \autoref{210522171904}.
Consequently, formula \eqref{210526072953} yields a unique holomorphic solution $s_\alpha$ on $V'_{\theta,\alpha} \times S'_\theta$ which admits uniform Gevrey asymptotics $s_\alpha (z, \hbar) \simeq \hat{s}_\alpha (z, \hbar)$ as $\hbar \to 0$ along $\bar{A}_\theta = [\theta - \tfrac{\pi}{2}, \theta + \tfrac{\pi}{2}]$ with leading-order $\lambda_\alpha$.

%==============================
\textsc{5. Linear independence.}
To prove linear independence, we compute the Wronskian of $\psi_+, \psi_-$ at the basepoint $x_0$ using formula \eqref{210218230927}:
\eqn{
	\WW \big( \psi_-, \psi_+ \big)
	= \det \mtx{ \psi_- & \psi_+  \\ \del_x \psi_- & \del_x \psi_+}
	= \det \mtx{ 1 & 1 \\ \frac{1}{\hbar} s_- & \frac{1}{\hbar} s_+ }
	= \frac{1}{\hbar} \big( s_+ - s_- \big)
\fullstop{,}
}
which is nonzero for $x = x_0$ since the leading-order part in $\hbar$ of $s_+ - s_-$ at $x_0$ is $\lambda_+ (x_0) - \lambda_- (x_0) = \sqrt{\DD_0 (x_0)} \neq 0$ by the assumption that $x_0$ is a regular point.
This competes the proofs of \autoref{210603110237} and \autoref{210116200501}.
\end{proof}
\end{appendices}

%===============================================================================
%:	REFERENCES
%===============================================================================
%\newpage
\begin{adjustwidth}{-2cm}{-1.5cm}
{\footnotesize
\bibliographystyle{nikolaev}
%\bibliography{References}
\bibliography{/Users/Nikita/Documents/Library/References}
}
\end{adjustwidth}
%===============================================================================
%:	DOCUMENT ENDS
%===============================================================================
\end{document}